\documentclass[11pt,a4paper]{article}

\usepackage{mathtools}
\usepackage{amsthm,amsmath}
\usepackage{amssymb,MnSymbol}
\usepackage[svgnames]{xcolor}
\usepackage{tabto}
\usepackage{paralist}
\usepackage[
colorlinks=true,
linkcolor=MidnightBlue,
citecolor=MidnightBlue,
urlcolor=MidnightBlue,
pagebackref=true]{hyperref}
\usepackage[alphabetic]{amsrefs}
\usepackage{graphicx}
\usepackage{pinlabel}
\usepackage{enumitem}
\usepackage[norefs]{refcheck}
\usepackage[margin=3.3cm]{geometry}

\newcommand{\R}{\mathbb{R}}
\newcommand{\C}{\mathbb{C}}
\renewcommand{\H}{\mathbb{H}}

\newcommand{\id}{\mathrm{id}}
\newcommand{\Gr}{\mathrm{Gr}}
\newcommand{\Span}{\mathrm{Span}}
\newcommand{\PGL}{\mathrm{PGL}}
\newcommand{\PO}{\mathrm{PO}}
\newcommand{\GL}{\mathrm{GL}}
\newcommand{\PSL}{\mathrm{PSL}}

\newcommand{\Pb}{\mathbb{P}}

\newcommand{\Bc}{\mathcal{B}}
\newcommand{\Cc}{\mathcal{C}}
\newcommand{\Dc}{\mathcal{D}}

\newcommand{\Fc}{\mathcal{F}}
\newcommand{\Tmc}{\mathcal{T}}

\newcommand{\Out}{\mathrm{Out}}

\newcommand{\Hom}{\mathrm{Hom}}
\newcommand{\tr}{\mathrm{tr}}

\numberwithin{equation}{section}

\theoremstyle{plain}
\newtheorem{theorem}{Theorem}[section]
\newtheorem{corollary}[theorem]{Corollary}

\newtheorem{lemma}[theorem]{Lemma}
\newtheorem{proposition}[theorem]{Proposition}
\newtheorem{observation}[theorem]{Observation}
\newtheorem*{claim*}{Claim}
\newtheorem*{theorem*}{Theorem}
\newtheorem*{proposition*}{Proposition}

\theoremstyle{definition}
\newtheorem{definition}[theorem]{Definition}
\newtheorem{example}[theorem]{Example}

\theoremstyle{remark}
\newtheorem{remark}[theorem]{Remark}

\title{Weakly positive and directed Anosov representations}
\author{Sungwoon Kim\footnote{SK was partially supported by the NRF grant NRF-2018R1D1A1B07043321.}, Ser Peow Tan\footnote{ST was partially supported by the NUS-MOE grant R-146-000-289-114.}, and Tengren Zhang\footnote{TZ was partially supported by the NUS-MOE grant R-146-000-270-133.}
}

\date{}
\begin{document}

\maketitle

\begin{abstract}
Given a finitely generated group $\Gamma$, a directed graph $\Lambda$, and a map $R:\Lambda\to\Gamma$, we introduce the notion of an $(R,\Lambda)$-directed Anosov representation. This is a weakening of the notion of Anosov representations. Our main theorem gives a procedure to construct $(R,\Lambda)$-directed Anosov representations using Fock-Goncharov positivity. As an application of our main theorem, we construct large families of primitive stable representations from $F_2$ to $\PGL(V)$, including non-discrete and non-faithful examples.
\end{abstract}

\tableofcontents

\section{Introduction}

The study of representations of the fundamental group $\pi_1(S)$ of a closed, orientable surface $S$ into non-compact semisimple Lie groups of higher rank has seen explosive growth in the last few decades thanks to pioneering work of Labourie \cite{Labourie} and Fock and Goncharov \cite{Fock-Goncharov}. An important motivation for this is the quest to understand generalizations due to Hitchin \cite{Hitchin} of Teichm\"uller Theory, which today has been given the name Higher Teichm\"uller Theory. Teichm\"uller theory can be thought of as the study of the Teichm\"uller component of ${\rm Hom}(\pi_1(S),\PGL (2, \mathbb R))/\PGL (2, \mathbb R)$, which is the component consisting of discrete faithful representations. The key result by Labourie \cite{Labourie} and Fock and Goncharov \cite{Fock-Goncharov} is that there is a distinguished component of ${\rm Hom}(\pi_1(S), \PGL(n,\mathbb R))/\PGL(n,\mathbb{R})$ which consists entirely of discrete, faithful representations. Strikingly, this component, now commonly named the Hitchin component, mirrors many properties enjoyed by the Teichm\"uller component. 

In his proof of this key result, Labourie introduced the notion of an Anosov representation. These are representations from a hyperbolic group to a semisimple Lie group, which one can think of as a generalization of convex cocompact representations to higher rank Lie groups. He also used this notion to show that the action of the mapping class group of $S$ on ${\rm Hom}(\pi_1(S),\PGL (n, \mathbb R))/\PGL (n, \mathbb R)$ by post composition restricts to a properly discontinuous action on the Hitchin component \cite{Labourie08}. 

On the other hand, if $F_d$ is a free group of rank $d\ge 2$, one can consider representations of $F_d$ into $\PGL(n,\mathbb{R})$. In this case, there is also a natural action of the outer automorphism group $\Out (F_d)$ on ${\rm Hom}(F_d,\PGL(n,\mathbb{R}))/\PGL(n,\mathbb{R})$. Minsky \cite{Minsky} defined a special class of representations from $F_d$ to $\PSL (2, \mathbb C)$, called primitive stable representations, and he showed that the $\Out (F_d)$-action on ${\rm Hom}(F_d,\PSL(2,\mathbb C))/\PSL(2,\mathbb C)$ restricts to a properly discontinuous action on the set of (conjugacy classes of) primitive stable representations, which is open in ${\rm Hom}(F_d,\PSL(2,\mathbb C))/\PSL(2,\mathbb C)$. Furthemore, he showed that while Schottky representations are primitive stable, there are primitive stable representations which are not discrete and faithful. Primitive stable representations were generalized to more general Lie groups $G$ by Guichard-Gueritaud-Kassel-Weinhard \cite{GGKW17}, and also studied by Kim-Kim \cite{KK}. In this more general setting, one can show that primitive stable representations are open in ${\rm Hom}(F_d,G)/G$ using the local-to-global principle for Morse quasi-geodesics developed by Kapovich-Leeb-Porti \cites{KLP14}. Also, it follows from a straightforward generalization of Minsky's argument that the $\Out (F_d)$-action on ${\rm Hom}(F_d,G)/G$ will still restrict to a properly discontinuous action on the set of (conjugacy classes of) primitive stable representations. 

In general, given a representation $\rho: F_d\to \PGL(n,\mathbb{R})$, it can be difficult to verify if $\rho$ is primitive stable. When $d=2$ and $G=\PGL(2,\mathbb R)$, Goldman-McShane-Stantchev-Tan \cite{GMST} proved a characterization of primitive stable representations as holonomies of (possibly singular) hyperbolic structures on the four surfaces whose fundamental group is $F_2$. However, very little is known in the higher rank case. A main goal of this paper is to establish some finite set of verifiable conditions which allows one to certify that a representation of $\rho:F_2\to\PGL(n,\mathbb{R})$ is primitive stable. 

We do this by introducing two notions. The first is the notion of a \emph{$(R,\Lambda)$-weakly positive representation} from a finitely generated group $\Gamma$ to $\PGL(n,\mathbb{R})$. This is a finite set of conditions, determined by a finite directed graph $\Lambda$ and an identification $R$ between the vertices of $\Lambda$ and a finite generating set of $\Gamma$, that one can verify for any representation from $\Gamma$ to $\PGL(n,\mathbb{R})$, see Definition \ref{def: weakly positive intro}. The second is the notion of a \emph{$(R,\Lambda)$-directed Anosov} representation from $\Gamma$ to $\PGL(n,\mathbb{R})$. Informally, this is a representation that is  (Borel)-Anosov along certain geodesic rays in $\Gamma$ that are specified by $R$ and $\Lambda$, see Definition \ref{def: directed Anosov intro}. The main theorem of this paper then states that under some mild hypothesis on $\Lambda$, if a representation is $(R,\Lambda)$-weakly positive, then it is $(R,\Lambda)$-directed Anosov, see Theorem \ref{thm: main intro} for a precise statement. Using this theorem, we prove an easily verified condition that certifies if a representation $\rho:F_2\to\PGL(n,\mathbb{R})$ is primitive stable, see Theorem \ref{thm: intro} and Theorem \ref{thm: intro last}. With this, we can construct new examples of primitive stable representations. 

 We describe in greater detail our set-up and results in the rest of this introduction.\\


\subsection{Weakly positive and directed Anosov representations}\label{sec: intro main}

To motivate the notion of directed Anosov representations, let us first recall the definition of Anosov representations. Let $\Gamma$ be a finitely generated group with finite generating set $S$, and equip $\Gamma$ with the word metric $d_S$ associated to $S\cup S^{-1}$. We say that a sequence $(\eta_i)_{i=0}^\infty$ in $\Gamma$ is \emph{rooted} if $\eta_0=\id$. Let $V$ be a $n$-dimensional real vector space, and fix a point $o$ in the $\PGL(V)$-Riemannian symmetric space $X$, i.e. a scalar class of inner products on $V$. This choice of $o$ determines a Cartan projection 
\[\mu:\PGL(V)\to\mathfrak a:=\left\{(x_1,\dots,x_n)\in\R^n:\sum_{i=1}^nx_i=0\right\}.\] 
Let $\Delta\subset \mathfrak a^*$ denote the set of simple roots of $\PGL(V)$. A representation $\rho:\Gamma\to\PGL(V)$ is \emph{(Borel) Anosov} if for all $k,k'>0$, there are constants $\kappa,\kappa'>0$ such that for all rooted $(k,k')$-quasi geodesic rays $(\eta_i)_{i=0}^\infty$ in $\Gamma$,
\begin{equation}\label{eqn: linear growth intro}
\alpha\circ\mu(\rho(\eta_{i}))\geq \kappa i-\kappa'\end{equation}
for all integers $i\geq 0$ and all $\alpha\in\Delta$.

Directed Anosov representations are a generalization of Anosov representations, where \eqref{eqn: linear growth intro} is not required to hold for all rooted geodesic rays, but only along certain sequences that are determined by a finite directed graph. More precisely, let $\Lambda=(A,B)$ be a finite directed graph, i.e. $A$ is a finite set, called the set of \emph{vertices}, and $B$ is a set of ordered pairs of vertices, called the set of \emph{directed edges}, see Figure \ref{fig: directed graph}. A \emph{directed ray in $\Lambda$} is a sequence $(p_i)_{i=1}^\infty$ in $A$ such that $(p_i,p_{i+1})\in B$ for all integers $i>0$. Let $R:A\to\Gamma$ be an injective map. A sequence $(\eta_i)_{i=0}^\infty$ in $\Gamma$ is \emph{$(R,\Lambda)$-directed} if it is rooted, and there is a directed ray $(p_i)_{i=1}^\infty$ in $\Lambda$ such that 
\[\eta_i=R(p_1)R(p_{i-1})\dots R(p_i)\]
for all integers $i>0$.

\begin{figure}[h]
    \centering
    \includegraphics[width=0.25\textwidth]{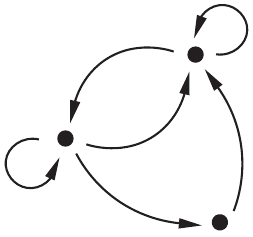}
    \caption{Directed graph.}
    \label{fig: directed graph}
\end{figure}

\begin{definition}\label{def: directed Anosov intro}
A representation $\rho :\Gamma \to\PGL(V)$ is \emph{$(R,\Lambda)$-directed (Borel) Anosov} if for all constants $k,k'>0$, there exists constants $\kappa, \kappa'>0$ such that for all $(R,\Lambda)$-directed, $(k,k')$-quasi-geodesics $(\eta_i)_{i=0}^\infty$ in $\Gamma$, \eqref{eqn: linear growth intro} holds for all integers $i\geq 0$ and all $\alpha\in\Delta$.
\end{definition} 

Even though our formulation is slightly different, the idea of encoding the data of quasi-geodesic rays in $\Gamma$ using a finite directed graph is from Bochi-Potrie-Sambarino \cite{BPS}. In their work, this appears under the guise of geodesic automatons, which they use to study nesting properties of the images (under the limit map of an Anosov representation) of the boundary of cone types in hyperbolic groups.

Clearly, every Anosov representation is $(R,\Lambda)$-directed Anosov for any $(R,\Lambda)$. However, for certain choices of $(R,\Lambda)$, there are examples of $(R,\Lambda)$-directed Anosov representations that are not discrete and faithful, and so are necessarily not Anosov representations. This additional flexibility is useful as a framework to study certain classes of representations that include non-discrete representations, such as primitive stable representations. 

The main theorem of this paper gives a procedure to construct $(R,\Lambda)$-directed Anosov representations using Fock-Goncharov positivity (see Section \ref{sec: positivity} for the definition of positive). To state this theorem, we define the notion of an $(R,\Lambda)$-weakly positive representation. 

Let $\Fc(V)$ denote the space of complete flags in $V$. Given a positive triple of flags $(F,H,K)$ in $\Fc(V)$, let 
\begin{equation}\label{eqn: open sets}\mathfrak U(F,H,K):=\{G\in\Fc(V):(F,G,H,K)\text{ is positive}\}.\end{equation}
Then for any map $f:A\to\PGL(V)$, a \emph{compatible system of forward domains for $(f,\Lambda)$} is an assignment to each $a\in A$ a positive triple of flags $(F_a,H_a,K_a)$ such that the following hold for all $(a_1,a_2)\in B$:
\begin{itemize}
\item $f(a_1)\cdot\mathfrak{U}(F_{a_2},H_{a_2},K_{a_2})\subset\mathfrak{U}(F_{a_1},H_{a_1},K_{a_1})$, and
\item $(F_{a_1},f(a_1)\cdot F_{a_2},f(a_1)\cdot H_{a_2},H_{a_1})$ is positive up to switching $F_{a_2}$ and $H_{a_2}$.
\end{itemize}
We say that $f$ is \emph{$\Lambda$-admissible} if there is a compatible system of forward domains for $(f,\Lambda)$.

\begin{definition}\label{def: weakly positive intro}
A representation $\rho:\Gamma\to\PGL(V)$ is \emph{$(R,\Lambda)$-weakly positive} if $\rho\circ R$ is $\Lambda$-admissible.
\end{definition}

Our main theorem relates the notions of weak positivity and directed Anosovness. We say that $\Lambda$ is \emph{path-symmetric} if for all $a,b\in\Lambda$, there is a directed path from $a$ to $b$ if and only if there is a directed path from $b$ to $a$. (See Figure \ref{fig: not path-symmetric} for a directed graph that is not path-symmetric.)

\begin{figure}[h]
    \centering
    \includegraphics[width=0.20\textwidth]{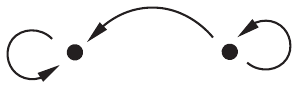}
    \tiny
\put (-21,3){$a$}
\put (-63,2){$b$}
    \caption{Directed graph that is not path-symmetric.}
    \label{fig: not path-symmetric}
\end{figure}

\begin{theorem}[Theorem \ref{thm: weakly positive is directed Anosov}]\label{thm: main intro}
Suppose that $\Lambda$ is path-symmetric. If $\rho:\Gamma\to\PGL(V)$ is $(R,\Lambda)$-weakly positive, then $\rho$ is $(R,\Lambda)$-directed Anosov.
\end{theorem}

Unlike the $(R,\Lambda)$-directed Anosov condition, whether or not a representation is $(R,\Lambda)$-weakly positive can be verified by checking finitely many explicit inequalities. Theorem \ref{thm: main intro} thus gives us an explicit way to construct $(R,\Lambda)$-directed Anosov representations. In this sense, one can think of Theorem \ref{thm: main intro} as a generalization of results by Burelle-Treib \cite{BT}, who constructed Schottky representations using positivity. See Section \ref{sec: adm} for more details.
 
 \subsection{Application to primitive stable representations}
 
Let $F_d$ denote the free group on $d$ generators. As an application of Theorem \ref{thm: main intro}, we construct new and explicit examples of \emph{primitive stable representations} $\rho:F_2\to\PGL(V)$, which we now define.  An element $\gamma_1\in F_d$ is \emph{primitive} if there are elements $\gamma_2,\dots,\gamma_d\in F_d$ such that $\{\gamma_1,\dots,\gamma_d\}$ is a generating set for $F_d$. If we equip $F_d$ with a word metric, then an \emph{axis} of a non-identity element $\gamma\in F_d$ is a geodesic in $F_d$ that is invariant under $\gamma$. A \emph{primitive geodesic} is an axis of a primitive element in $F_d$, and a \emph{primitive geodesic ray} is a geodesic ray that lies in a primitive geodesic.
 
\begin{definition}\label{def: primitive stable intro}
A representation $\rho:F_d\to\PGL(V)$ is \emph{(Borel) primitive stable} if for some (equivalently, any) word metric on $F_d$, there are constants $\kappa, \kappa'>0$ such that \eqref{eqn: linear growth intro} holds for all rooted, primitive geodesic rays $(\eta_i)_{i=0}^\infty$ in $F_d$, all integers $i\geq 0$, and all $k=1,\dots,n-1$.
\end{definition}

Our main application of Theorem \ref{thm: main intro} is to find easily verifiable conditions to construct primitive stable representations. For that purpose, we will consider path-symmetric, directed graphs of the following form: Let $\mathsf{K}(v,w)$ be the complete finite directed graph with vertices $v,w$, that is, $A=\{v,w\}$ and $B=A\times A$. For a positive integer $i$, let
$$\Lambda_i=(A_i,B_i):=\mathsf{K}(v_1,w_1)\cup \mathsf{K}(v_1', w_1')\cup \cdots \cup \mathsf{K}(v_i,w_i)\cup \mathsf{K}(v_i',w_i').$$
For a pair of elements $(\gamma_1,\gamma_2)$ of $F_2$, we say that a map $R:A_2\to F_2$ is \emph{defined by $(\gamma_1,\gamma_2)$} if
$$(R(v_1),R(w_1))=(\gamma_1, \gamma_2),\,\, (R(v_2),R(w_2))=(\gamma_1^{-1}, \gamma_2), $$
and $R(v_i')=R(v_i)^{-1}$, $R(w_i')=R(w_i)^{-1}$ for $i=1,2$. In addition, if we let $\gamma_3:=\gamma_2^{-1}\gamma_1^{-1}$, then we say that a map $R:A_3\to F_2$ is \emph{defined by} $(\gamma_1,\gamma_2)$ if $$(R(v_1),R(w_1))=(\gamma_1^{-1}, \gamma_2),\,\, (R(v_2),R(w_2))=(\gamma_2^{-1}, \gamma_3),\, \,(R(v_3),R(w_3))=(\gamma_3^{-1}, \gamma_1),$$
and $R(v_i')=R(v_i)^{-1}$, $R(w_i')=R(w_i)^{-1}$ for all $i=1,2,3$.

\begin{proposition}[Proposition \ref{prop: forward primitive}]\label{thm: intro'}
Let $(\gamma_1,\gamma_2)$ be a pair of generators of $F_2$, and for $i=2,3$, let $R_i:A_i\to F_2$ be defined by $(\gamma_1,\gamma_2)$. If $\rho:F_2\to\PGL(V)$ is $(R_2,\Lambda_2)$-directed Anosov or $(R_3,\Lambda_3)$-directed Anosov, then it is primitive stable. In particular, if $\rho:F_2\to\PGL(V)$ is $(R_2,\Lambda_2)$-weakly positive or $(R_3,\Lambda_3)$-weakly positive, then it is primitive stable.
\end{proposition}

We prove in the appendix that when $V=\R^2$, the converse of Proposition \ref{thm: intro'} holds.

Proposition \ref{thm: intro'} has several consequences. First, we use it to prove an easily verified condition under which a representation is guaranteed to be primitive stable. Henceforth, for any loxodromic $g\in\PGL(V)$, $g_-$ and $g_+$ will denote its repelling and attracting fixed point in the space of (complete) flags in $V$. Recall also that a loxodromic element $g\in\PGL(V)$ is \emph{positive loxodromic} if all of its eigenvalues have the same sign.

\begin{theorem}[Theorem \ref{thm: general n}]\label{thm: intro}
Let $b\in\PGL(V)$ be positive loxodromic, and let $a\in\PGL(V)$ be loxodromic. If $(b_-,a\cdot b_-,a_+,a\cdot b_+,b_+,a_-)$ is positive up to switching $a\cdot b_-$ and $a\cdot b_+$, then the representation $\rho:F_2\to\PGL(V)$ defined by $\rho(\gamma_1)=a$ and $\rho(\gamma_2)=b$ is primitive stable.
\end{theorem}

Of course, every Anosov representation from $F_2$ to $\PGL(V)$ is primitive stable, and these can be constructed using Ping-pong lemma type arguments. Theorem \ref{thm: intro} on the other hand, allows us to exhibit explicit families of primitive stable representations from $F_2$ to $\PGL(V)$ that are not Anosov, and whose images do not lie in $\iota(\PGL_2(\R))$. Examples include non-positive representations, see Section \ref{sec: example1}, as well as non-discrete and non-faithful representations Section \ref{sec: example3}.

Another feature of Theorem \ref{thm: intro} is that unlike the Ping-pong lemma, it guarantees primitive stability of a representation $\rho:F_2\to\PGL(V)$ without requiring the ratio of adjacent eigenvalues of $\rho(\gamma_1)$ and $\rho(\gamma_2)$ to be sufficiently different. As a consequence, we can construct, given a generating pair $\{\gamma_1,\gamma_2\}$ of $F_2$, an explicit family of primitive stable representations $\rho_t:F_2\to\PGL(V)$ that converge to the trivial representation, with the property that $\rho_t(\gamma_1)_\pm$ and $\rho_t(\gamma_2)_\pm$ do not vary with $t$, see Section \ref{sec: example2}. We can also ensure that the image of $\rho_t$ also does not lie in $\iota(\PGL_2(\R))$, where $\iota:\PGL_2(\R)\to\PGL(V)$ is an irreducible representation.

Finally, in the case when $V=\R^3$ and both $a$ and $b$ are positive loxodromic, we have a simpler version of Theorem \ref{thm: intro}.

\begin{theorem}[Theorem \ref{thm: positive quadruple}]\label{thm: intro last}
If $a,b\in\PGL_3(\R)$ are positive loxodromic elements such that $(b_-,a_+,b_+,a_-)$ is positive, then the representation $\rho:F_2\to\PGL(V)$ defined by $\rho(\gamma_1)=a$ and $\rho(\gamma_2)=b$ is primitive stable.
\end{theorem}

The problem for higher rank free groups is significantly harder, even when $G=\PGL(2,\mathbb{R})$. In fact, the existence of non-discrete primitive stable representations from $F_3$ to $\PGL(2,\mathbb{R})$ is a long-standing question of Minsky \cite{Minsky}. While we believe some of the techniques and results here will be useful for studying primitive stable representations of higher rank free groups, we are at this point unable to push our results further to this general case.

\subsection{Proof of Theorem \ref{thm: main intro}}

The proof of Theorem \ref{thm: main intro} has two broad steps. The first step is the following theorem.

\begin{theorem}[Theorem \ref{thm: main 2}]\label{thm: sym intro}
Let $\rho:\Gamma\to\PGL(V)$ be a representation with the following properties:
\begin{itemize}
\item There is some $C>0$ such that for every $(R,\Lambda)$-directed sequence $(\eta_i)_{i=0}^\infty$ in $\Gamma$, the sequence $(\rho(\eta_i)\cdot o)_{i=0}^\infty$ in $X$ is $C$-bounded from a maximal flat in $X$,
\item $\displaystyle\lim_{i\to\infty}\alpha\circ\mu(\rho(\gamma_i))=\infty$ for every $\alpha\in\Delta$ if for each $i$, there is a positive integer $l_i$ such that $\lim_{i\to\infty}l_i=\infty$ and a 
$(R,\Lambda)$-directed sequence $(\eta_{i,l})_{l=0}^\infty$ such that $\gamma_i=\eta_{i,l_i}$.
\end{itemize}
Then $\rho$ is $(R,\Lambda)$-directed Anosov. 
\end{theorem}

Informally, this theorem states that if the $o$-orbit in $X$ of $(R,\Lambda)$-directed sequence in $\Gamma$ stay uniformly close to flats, then the growth of $\alpha\circ\mu(\rho(\gamma_i))$ along sequences $(\gamma_i)_{i=0}^\infty$ in $\Gamma$ of increasing products of elements in $R(A)$, can be upgraded to the linear growth of $\alpha\circ\mu(\rho(\eta_i))$ along $(R,\Lambda)$-directed sequences $(\eta_i)_{i=0}^\infty$.

The second step is the following theorem about positive tuples of flags and the $k$-th Labourie cross ratios $B_k$ (see Definition \ref{def: cross ratio}).

\begin{theorem}[Theorem \ref{thm: general k}]\label{thm: proj intro}
Let $(F_i)_{i=1}^\infty$ and $(H_i)_{i=1}^\infty$ be sequences of flags in $\Fc(V)$ and $K\in\Fc(V)$ such that for all integers $l\geq 2$, $(F_1,\dots,F_l,H_l,\dots,H_1,K)$ is a positive tuple of flags. If there is some $D>1$ such that $B_k(H_i,F_i,F_{i+1},H_{i+1})\leq D$ for all integers $i>0$ and all $k=1,\dots,n-1$, then $\displaystyle\lim_{i\to\infty}\overline{\mathfrak U(F_i,H_i,K)}$ is a point. 
\end{theorem}

If $\rho:\Gamma\to\PGL(V)$ is $(R,\Lambda)$-weakly positive and 
\[\{\mathfrak{U}_a=\mathfrak{U}(F_a,H_a,K_a):a\in A\}\]
is a compatible system of forward domains for $(\rho\circ R,\Lambda)$, one can use Theorem \ref{thm: proj intro} to prove that for all $(R,\Lambda)$-directed sequences $(\eta_i)_{i=0}^\infty$ in $\Gamma$, there is some $\bar{a}\in A$ such that $(\overline{\rho(\eta_i)\cdot \mathfrak U_{\bar{a}}})_{i=1}^\infty$ is a nested sequence of compact sets whose intersection is a point. From this and the assumption that $\Lambda$ is path-symmetric, one deduces that the hypothesis of Theorem \ref{thm: sym intro} holds, thus proving Theorem \ref{thm: main intro}. See Section \ref{sec:admissible} for more details.

\subsection{Organization of paper}
The rest of the paper is organized as follows. The proofs of Theorem \ref{thm: sym intro} and Theorem \ref{thm: proj intro} are given in Section \ref{sec: directed Anosov} and Section \ref{sec: positive} respectively. Then in Section \ref{sec:admissible}, we use Theorem \ref{thm: sym intro} and Theorem \ref{thm: proj intro} to finish the proof of Theorem \ref{thm: main intro}. In Section \ref{sec: 6}, we apply Theorem \ref{thm: main intro} to primitive stable representations, and give explicit constructions of primitive stable representations. Finally, we give a proof of the converse of Proposition \ref{thm: intro'} in the appendix.\\

{\bf Acknowledgements:} We are indebted to the referee of the first version of this paper, who suggested that we generalize our results to the current version by using directed graphs to combinatorially describe directed sequences in the group. This allowed us to use the same proofs to arrive at cleaner and stronger results. 
\section{Directed Anosov representations}\label{sec: directed Anosov}

In this section, we prove Theorem \ref{thm: sym intro}, which is restated below as Theorem~\ref{thm: main 2}. Given a representation $\rho:\Gamma\to\PGL(V)$, this theorem specifies conditions on the induced $\Gamma$-action on the $\PGL(V)$-Riemannian symmetric space $X$, under which one can deduce that $\rho$ is directed Anosov.

In Section \ref{sec: Riemannian}, we recall some required background from the theory of Riemannian symmetric spaces, before restating Theorem \ref{thm: main 2} in Section \ref{sec: directed Anosov sub}. The remainder of the section is the proof of Theorem \ref{thm: main 2}.

\subsection{The $\PGL(V)$-Riemannian symmetric space}\label{sec: Riemannian}

First, we recall some basic results about the $\PGL(V)$-Riemannian symmetric space. For a more thorough  and general exposition of this topic, we refer the reader to Chapter 2 of Eberlein \cite{Eberlein} and Chapter VI.3 of Helgason \cite{Helgason}.

\subsubsection{Roots and Weyl chambers}\label{sec: roots}

For any integers $i,j=1,\dots,n$ such that $i\neq j$, the \emph{$(i,j)$-th root} of $\PGL(V)$ is the linear map $\alpha_{i,j}:\mathbb{R}^n\to\mathbb{R}$ given by $\alpha_{i.j}:(x_1,\dots,x_n)\mapsto x_i-x_j$. Collectively, the set
\[\Phi:=\{\alpha_{i,j}:i,j=1,\dots,n\text{ and }i\neq j\}\]
is called the \emph{set of roots of $\PGL(V)$}. A root $\alpha_{i,j}$ is \emph{positive} if $i<j$ and \emph{negative} if $i>j$. Any root of the form $\alpha_{k,k+1}$ for some $k=1,\dots,n-1$ is \emph{simple}. We often denote the \emph{$k$-th simple root} $\alpha_{k,k+1}$ simply by $\alpha_k$, and denote the set of simple roots of $\PGL(V)$ by $\Delta$. Note that every positive (resp. negative) root can be written uniquely as a linear combination of the simple roots where all the coefficients are non-negative (resp. non-positive) integers. 

For $k=1,\dots,n-1$, let $r_{\alpha_k}$ be the reflection about the kernel of $\alpha_k$. 
The \emph{Weyl group} of $\PGL(V)$ is then the subgroup of $\GL(n,\R)$ that is generated by $\{r_\alpha:\alpha\in\Delta\}$. Observe that the $W$-action on $\R^n$ leaves the subspace
\[\mathfrak a:=\left\{(x_1,\dots,x_n)\in\R^n:\sum_{i=1}^nx_i=0\right\}\]
invariant, and the set
\[\mathfrak a^+:=\{x\in\mathfrak a:\alpha(x)\geq 0\text{ for all }\alpha\in\Delta\}\]
serves as a fundamental domain for the $W$-action on $\mathfrak a$. We refer to $\mathfrak a^+$ as the \emph{fundamental Weyl chamber}, and any subset of $\mathfrak a$ of the form $\omega\cdot\mathfrak a^+$ for some $\omega\in W$ as a \emph{Weyl chamber} of $\mathfrak a$.

The \emph{longest element} $\omega_0$ in the Weyl group $W$ is the unique element that sends the fundamental Weyl chamber $\mathfrak a^+$ to the Weyl chamber $-\mathfrak a^+$.
Observe then that $-\omega_0:\mathbb{R}^n\to\mathbb{R}^n$ is an involution that leaves the fundamental Weyl chamber invariant, so its induced action on $(\mathbb{R}^n)^*$ preserves the set of simple roots $\Delta$. We refer to this action on $\Delta$ as the \emph{opposition involution}, and denote it by $\iota:\Delta\to\Delta$. 

\subsubsection{Flats in the $\PGL(V)$-Riemannian symmetric space}\label{sec: flats}

Let $X$ denote the $\PGL(V)$-Riemannian symmetric space, i.e. $X$ is the unique (up to scaling) Riemannian symmetric space whose isometry group is $\PGL(V)$. As a $\PGL(V)$-space, $X$ is isomorphic to $\widetilde{X}/\sim$, where $\widetilde{X}$ is the set of inner products on $V$, and two inner products are equivalent under $\sim$ if they are multiples of each other. 

In a Riemannian metric space, a \emph{flat} is a totally geodesic subspace whose sectional curvatures are all zero. In the case of $X$, every flat is isometric to $\mathbb{R}^k$ for some $k=1,\dots,n-1$, and the maximal flats are of dimension $n-1$. These maximal flats can be described as the orbits of certain subgroups of $\PGL(V)$ in the following way. 

A subgroup of $\PGL(V)$ is \emph{diagonalizable} if every element in that subgroup is diagonalizable over $\mathbb{R}$. Let $A\subset\PGL(V)$ be a maximal, diagonalizable, connected, abelian subgroup of $\PGL(V)$, and let $o\in X$ be a point. We say that an (ordered) basis $\Bc:=(e_1,\dots,e_n)$ of $V$ is an \emph{appropriate basis for $(A,o)$} if it has the following properties:
\begin{itemize}
\item every $a\in A$ is represented in the basis $\Bc$ by a diagonal matrix. 
\item $\Bc$ is an orthonormal basis for some inner product in the scalar class of inner products corresponding to $o\in X$.
\end{itemize}
Any appropriate basis for $(A,o)$ is unique (if it exists) up to permuting the vectors in the basis, replacing each vector in the basis with its negative, and scaling all the vectors in the basis by the same positive number.

If an appropriate basis for $(A,o)$ exists, then $\mathbf F_A:=A\cdot o\subset X$ is a maximal flat. Furthermore, for every maximal flat $\mathbf F\subset X$, there is a maximal, diagonalizable, connected, abelian subgroup $A\subset \PGL(V)$ such that $\mathbf F=\mathbf F_A$. As such, when convenient, we also refer to an appropriate basis for $(A,o)$ as an \emph{appropriate basis for $(\mathbf F_A,o)$}. This basis $\Bc$ defines a parameterization 
\[\phi_{\Bc}:\mathfrak a\to \mathbf F_A\] 
by $(x_1,\dots,x_n)\mapsto \mathrm{diag}(e^{x_1},\dots,e^{x_n})\cdot o$. We refer to the image under $\phi_\Bc$ of any Weyl chamber of $\mathfrak a$ as a \emph{Weyl chamber} of $(\mathbf F_A,o)$. Note that the Weyl chambers of $(\mathbf F_A,o)$ do not depend on the choice of appropriate basis for $(\mathbf F_A,o)$. 

\subsubsection{The Weyl chamber valued distance}\label{sec: Weyl chamber valued distance}
The $\PGL(V)$-action on $X$ by isometries induces a transitive $\PGL(V)$-action on the space of pointed maximal flats, i.e. pairs $(\mathbf F,o)$ such that $\mathbf F\subset X$ is a maximal flat and $o$ is a point in $\mathbf F$. It turns out that the stabilizer in $\PGL(V)$ of $(\mathbf F,o)$ is a finite group $W'$. Furthermore, the image of the obvious representation $W'\to\mathrm{Isom}(\mathbf F)$ is isomorphic to the Weyl group $W$. In fact, by choosing an appropriate basis $\Bc$ for $(\mathbf F,o)$ as we did above, the parameterization $\phi_{\Bc}:\mathfrak a\to\mathbf F$ intertwines the Weyl group action on $\mathfrak a$ and $\mathbf F$. 

Now, for any pair of points $(p_1,p_2)$ in $X$, choose a maximal flat $\mathbf F$ containing $p_1$ and $p_2$, and choose an appropriate basis $\Bc$ for $(\mathbf F,p_1)$. By permuting the vectors in the chosen basis,  we can ensure that $p_2$ lies in the Weyl chamber $\phi_{\Bc}(\mathfrak a^+)$, where $\mathfrak a^+$ is the fundamental Weyl chamber. Then define $d_\mathfrak a^+(p_1,p_2):=\phi_{\Bc}^{-1}(p_2)\in\mathfrak a^+$. One can verify that $d_\mathfrak a^+(p_1,p_2)$ does not depend on any of the choices made, and is entirely determined by the (ordered) pair of points $(p_1,p_2)$. This thus defines a map 
\[d_\mathfrak a^+:X\times X\to\mathfrak a^+\] 
called the \emph{Weyl chamber valued distance}. 

It follows from the definition of the opposition involution $\iota:\Delta\to\Delta$ that for any $p_1,p_2\in X$ and any $\alpha\in\Delta$, we have 
\begin{equation}\label{eq: opp}
\alpha(d_\mathfrak a^+(p_1,p_2))=\iota(\alpha)(d_\mathfrak a^+(p_2,p_1)).
\end{equation}
More generally, if $p_1,p_2,p_3\in X$ lie in a maximal flat ${\bf F}$, and $p_2$ and $p_3$ lie in the same Weyl chamber of $({\bf F},p_1)$, then
\begin{equation}\label{eqn: triangle}
d_\mathfrak a^+(p_1,p_2)-d_\mathfrak a^+(p_1,p_3)=\omega\cdot d_\mathfrak a^+(p_2,p_3)
\end{equation}
for some $\omega\in W$. 

Let $d_X:X\times X\to\mathbb{R}$ denote the metric on $X$ induced by the Riemannian metric. Then for all $p_1,p_2\in X$, 
\[d_X(p_1,p_2)=\|d_\mathfrak a^+(p_1,p_2)\|,\] 
where $\|\cdot\|:\mathbb{R}^n\to\mathbb{R}$ is the standard norm. Furthermore, it follows from Kapovich-Leeb-Millson \cite[Theorem 1.1]{KLM} that $d_\mathfrak a^+$ is $1$-Lipschitz in each entry, so
 \begin{align}\label{eqn:ddti}
\| d_\mathfrak a^+(p_1,p_2)-d_\mathfrak a^+(p_1',p_2')\|&\leq \| d_\mathfrak a^+(p_1,p_2)-d_\mathfrak a^+(p_1',p_2)\|+\|d_\mathfrak a^+(p_1',p_2)-d_\mathfrak a^+(p_1',p_2')\|\nonumber\\
& \leq d_X(p_1,p_1')+d_X(p_2,p_2').
\end{align}
for all $p_1,p_2,p_1',p_2'\in X$. 

\subsubsection{Jordan and Cartan projections}

For any linear map $\bar{g}\in\GL(V)$, let $\lambda_1(\bar{g})\geq\dots\geq\lambda_n(\bar{g})$ be the absolute values of the eigenvalues of $\bar{g}$. The \emph{Jordan projection of $\GL(V)$} is the map 
\[\lambda:\GL(V)\to\{v\in\mathbb{R}^n:\alpha(v)\geq 0\text{ for all }\alpha\in\Delta\}\] 
defined by $\lambda:\bar{g}\mapsto(\log\lambda_1(\bar{g}),\dots,\log\lambda_n(\bar{g}))$. Using this, define the \emph{Jordan projection of $\PGL(V)$} to be the map 
\[\lambda:\PGL(V)\to\mathfrak a^+\]
given by $\lambda:g\mapsto\lambda(\bar{g})$, where $\bar{g}\in\GL(V)$ is a linear representative of $g\in\PGL(V)$ such that $|\det(\bar{g})|=1$. It is straightforward to verify that $\lambda$ is well-defined.

To define the Cartan projection, choose an inner product on $V$. For any $\bar{g}\in\GL(V)$, let $\mu_1(\bar{g})\geq\dots\geq\mu_n(\bar{g})$ denote the singular values of $\bar{g}$. The \emph{Cartan projection of $\GL(V)$} is then the map 
\[\mu:\GL(V)\to\{v\in\mathbb{R}^n:\alpha(v)\geq 0\text{ for all }\alpha\in\Delta\}\] 
given by $\mu(\bar{g})=(\log\mu_1(\bar{g}),\dots,\log\mu_n(\bar{g}))$, where $\mu_1(\bar{g})\geq\dots\geq\mu_n(\bar{g})$ are the singular values of $\bar{g}$ in the chosen inner product. With this, we define the \emph{Cartan projection of $\PGL(V)$} to be the map 
\[\mu:\PGL(V)\to\mathfrak a^+\]
given by $\mu:g\mapsto\mu(\bar{g})$, where $\bar{g}\in\GL(V)$ is a linear representative of $g\in\PGL(V)$ such that $|\det(\bar{g})|=1$. As before, one verifies that $\mu$ is well-defined. Also, note that replacing the chosen inner product by a scalar multiple of itself leaves the singular values of any $\bar{g}\in\GL(V)$ unchanged. Thus, the choice of a point $o\in X$ determines a Cartan projection $\mu:\PGL(V)\to\mathfrak a^+$. 


Using $d_\mathfrak a^+$, we can give interpretations of the Cartan projection and Jordan projection in terms of the geometry of $X$. One can verify that for any isometry $g\in\PGL(V)$, 
\[d_\mathfrak a^+(o,g\cdot o)=\mu(g),\]
where $o\in X$ is the point that determines the Cartan projection $\mu$. In particular, if $\|\cdot\|$ is the standard norm on $\mathbb R^n$, then 
\[d_X(o,g\cdot o)=\|\mu(g)\|.\] 
On the other hand, if $g\in\PGL(V)$ is loxodromic, then $g$ lies in a unique maximal, diagonalizable, abelian subgroup of $\PGL(V)$. Denote the identity component of this subgroup by $A_g$, and note that $\mathbf F_g:=\mathbf F_{A_g}$ is the unique maximal flat that is invariant under the action of $g$ on $X$. 
One can then verify that if $x$ is a point that lies in $\mathbf F_g$, then 
\[d_\mathfrak a^+(x,g\cdot x)=\lambda(g).\]
Since $X$ is non-positively curved, the closest point projection $\pi_{\mathbf F_g}:X\to\mathbf F_g$ is $1$-Lipschitz, so
\[\inf_{x\in X}d_X(x,g\cdot x)=\|\lambda(g)\|.\]

\subsection{Directed Anosovness}\label{sec: directed Anosov sub}
Let $\Lambda=(A,B)$ be a finite directed graph with $A$ as its set of vertices and $B$ as its set of directed edges. Let $R:A\to\Gamma$ be an injective map. Fix a point $o\in X$ with which we define a Cartan projection $\mu:\PGL(V)\to\mathfrak a^+$. 

\begin{definition}\label{def: directed Anosov}
Let $\theta\subset\Delta$ be a non-empty subset. A representation $\rho :\Gamma \to\PGL(V)$ is \emph{$(\theta,R,\Lambda)$-directed Anosov} if for all $k,k'>0$, there exists constants $\kappa, \kappa'>0$ such that 
\begin{equation}\label{eqn: linear growth}
\alpha\circ\mu(\rho(\eta_i))\geq \kappa i-\kappa'\end{equation}
for all $(R,\Lambda)$-directed, $(k,k')$-quasigeodesic rays $(\eta_i)_{i=0}^\infty$ in $\Gamma$, all $\alpha\in\theta$, and all integers $i\geq 0$. 
\end{definition}

Observe that whether or not a representation is $(\theta,R,\Lambda)$-directed Anosov does not depend on the choice of $o$. 

The notion of a $(\theta,R,\Lambda)$-directed Anosov representation is a generalization of the more well-known notion of a $\theta$-\emph{Anosov representation} (also called a $P$-Anosov representation, where $P$ is a parabolic subgroup associated to $\theta$). These are representations $\rho:\Gamma\to\PGL(V)$ where the inequality \eqref{eqn: linear growth} holds for all $\alpha\in\theta$, all rooted geodesic rays in $\Gamma$ (for some choice of word metric on $\Gamma$), and all integers $i\geq 0$. Indeed, every $\theta$-Anosov representation is $(\theta,R,\Lambda)$-directed Anosov for any $(R,\Lambda)$. Also, if $\Lambda$ is complete and the inclusion of the semigroup generated by $R(A)$ into $\Gamma$ is a quasi-isometric embedding with respect to the word lengths, then a $(\theta,R,\Lambda)$-directed Anosov representation from $\Gamma$ to $\PGL(V)$ restricts to a $\theta$-Anosov representation of the semigroup generated by $R(A)$ in the sense of Kassel-Potrie \cite{KP}.

We previously observed that if $d_X$ denotes the metric on the $\PGL(V)$-Riemannian symmetric space $X$ induced by the Riemannian metric, then 
\[d_X(o,g\cdot o)=\|d_\mathfrak a^+(o,g\cdot o)\|=\sqrt{\sum_{\alpha\in\Delta}\alpha\circ\mu(g)^2}.\]
From this, it is a straightforward computation to prove the following observation.

\begin{observation}
If $\rho:\Gamma\to\PGL(V)$ is $(\theta,R,\Lambda)$-directed Anosov for any non-empty subset $\theta\subset\Delta$, then the orbit map $\Gamma\to X$ given by $\gamma\mapsto\rho(\gamma)\cdot o$ sends $(R,\Lambda)$-directed sequences in $\Gamma$ to uniform quasi-geodesic rays in $X$. Furthermore, the converse holds when $n=2$.
\end{observation}

For the rest of this paper, we focus on the case when $\theta=\Delta$. As such, we will drop $\Delta$ from our terminology, and refer to $(\Delta,R,\Lambda)$-directed Anosov representations simply as \emph{$(R,\Lambda)$-directed Anosov representations}. This recovers Definition \ref{def: directed Anosov intro}.

The following is the main theorem of this section.

\begin{theorem}\label{thm: main 2}
Let $\rho:\Gamma\to\PGL(V)$ be a representation with the following properties:
\begin{itemize}
\item There is some $C>0$ such that for every $(R,\Lambda)$-directed sequence $(\eta_i)_{i=0}^\infty$ in $\Gamma$, the sequence $(\rho(\eta_i)\cdot o)_{i=0}^\infty$ in $X$ is $C$-bounded from a maximal flat in $X$.
\item $\displaystyle\lim_{i\to\infty}\alpha\circ\mu(\rho(\gamma_i))=\infty$ for every $\alpha\in\Delta$  if for each $i$, there is a positive integer $l_i$ such that $\lim_{i\to\infty}l_i=\infty$ and a 
$(R,\Lambda)$-directed sequence $(\eta_{i,l})_{l=0}^\infty$ such that $\gamma_i=\eta_{i,l_i}$.
\end{itemize}
Then the inequality \eqref{eqn: linear growth} holds for all $(R,\Lambda)$-directed sequences $(\eta_i)_{i=0}^\infty$ in $\Gamma$. In particular, $\rho$ is $(R,\Lambda)$-directed Anosov. 
\end{theorem}

Theorem \ref{thm: main 2} is a special case of a more general theorem about sequences in $\PGL(V)$. To describe this, we need several definitions.

\begin{definition}\label{def: regulated}
Let $\mathcal{W}$ be a collection of sequences in $\PGL(V)$.
\begin{enumerate}
\item $\mathcal{W}$ is \emph{uniformly well-behaved} if there is a constant $C>0$ such that every sequence in $(w_i)_{i=0}^\infty$ in $\mathcal{W}$ has the property that $(w_i\cdot o)_{i=0}^\infty$ is $C$-bounded from a maximal flat in $X$, and $d_X(w_i\cdot o,w_{i+1}\cdot o)\leq C$ for all integers $i\geq 0$.
\item $\mathcal{W}$ is \emph{regulated} if for every $D>0$, there is an integer $N(D)>0$ such that 
\[\alpha\circ\mu(w_i^{-1}w_{i+j})\geq D\]
for all sequences $(w_i)_{i=0}^\infty$ in $\mathcal{W}$, all integers $i\geq 0$ and $j\geq N(D)$, and all $\alpha\in\Delta$.
\end{enumerate}
\end{definition}


\begin{proposition}\label{prop: other definition}
If $\mathcal{W}$ is a regulated and uniformly well-behaved collection of rooted sequences in $\PGL(V)$, then there exists constants $\kappa,\kappa'>0$ such that
\[\alpha\circ\mu(w_{i+j})-\alpha\circ\mu(w_i)\geq \kappa j-\kappa'.\]
for all sequences $(w_i)_{i=0}^\infty$ in $\mathcal{W}$, all integers $i,j\geq0$ and all $\alpha\in\Delta$.
\end{proposition}

\begin{remark}
Definition \ref{def: regulated} makes sense even when we replace $\PGL(V)$ with any semisimple Lie group of non-compact type. Proposition \ref{prop: other definition} also holds in this more general setting. Even though we write our proof only for $\PGL(V)$, our proof generalizes verbatim.
\end{remark}

\subsection{Proof of Proposition \ref{prop: other definition}}
We prove Proposition \ref{prop: other definition} via the following sequence of lemmas, each of which is an estimate required in the proof. 

\begin{lemma}\label{lem: div1}
Let $\mathcal{W}$ be a regulated and uniformly well-behaved family of rooted sequences in $\PGL(V)$. Then there exist a constant $A>0$ and an integer $N>0$ such that for every sequence $(w_i)_{i=0}^\infty\in\mathcal{W}$ and for each $\alpha\in\Delta$, there is some $\alpha'\in\Delta$ such that 
\[\left|\alpha\circ\mu(w_{i+j})-\alpha\circ\mu(w_i) \right| \geq \alpha'\circ\mu(w_i^{-1}w_{i+j})-A\]
for all integers $i\geq N$ and $j\geq 0$.
\end{lemma}

\begin{proof}
Let $C>0$ be the constant that $\mathcal W$ is uniformly well-behaved with respect to, see Definition \ref{def: regulated}. Since $\mathcal{W}$ is regulated and every sequence $(w_i)_{i=0}^\infty$ in $\mathcal{W}$ is rooted, there is an integer $N>0$ such that for all $\alpha\in\Delta$ and all $i\geq N$,
\begin{align}\label{eqn: N inequality}
\alpha(d_\mathfrak a^+(o, w_i\cdot o))= \alpha\circ \mu(w_i)=\alpha\circ \mu(w_0^{-1}w_i)\geq 5\sqrt{2}C.
\end{align}

Pick a sequence $(w_i)_{i=0}^\infty\in\mathcal{W}$. Let ${\bf F}$ be a maximal flat in $X$ such that $d_X(w_i\cdot o,{\bf F})\leq C$ for all integers $i\geq 0$, let $\pi_{\mathbf F}:X\to\mathbf F$ be the closest point projection onto $\mathbf F$, and denote $\widehat{x}:=\pi_{\mathbf F}(x)$ for all $x\in X$. Then let $\mathfrak a^+_0$ be the Weyl chamber of $({\bf F},\widehat{o})$ that contains $\widehat{w_N\cdot o}$. We first prove the claim that $\widehat{w_i\cdot o}$ lies in $\mathfrak a^+_0$ for all $i\geq N$. 

Since $d_X(w_i\cdot o,\widehat{w_i\cdot o})\leq C$ for all integers $i\geq 0$, observe that for all integers $i,j\geq 0$, 
\begin{equation}\label{eqn: i,j}
\|d_\mathfrak a^+(w_j\cdot o,w_i\cdot o)-d_\mathfrak a^+(\widehat{w_j\cdot o},\widehat{w_i\cdot o})\|\leq d_X(w_j\cdot o,\widehat{w_j\cdot o})+d_X(w_i\cdot o,\widehat{w_i\cdot o})\leq 2C,
\end{equation}
where the first inequality is (\ref{eqn:ddti}). Hence, if we denote the supremum norm of a linear map $\alpha:\mathbb{R}^n\to\mathbb{R}$ by $\|\alpha\|$, then for all $\alpha\in\Delta$, we have
\begin{equation}\label{eqn: norm inequal}
|\alpha(d_\mathfrak a^+(w_j\cdot o,w_i\cdot o)-d_\mathfrak a^+(\widehat{w_j\cdot o},\widehat{w_i\cdot o}))|\leq \|\alpha\|2C=2\sqrt{2}C.
\end{equation}
Since $(w_i)_{i=0}^\infty$ is rooted, \eqref{eqn: N inequality} and (\ref{eqn: norm inequal}) imply that
\begin{align}\label{eqn:wilb}
\alpha(d_\mathfrak a^+(\widehat{ o}, \widehat{w_i\cdot o})) \geq \alpha(d_\mathfrak a^+(o, w_i\cdot o))-2\sqrt{2}C\geq 3\sqrt{2}C.
\end{align}
for all $i\geq N$ and all $\alpha\in\Delta$.

We prove the claim by induction on $i\geq N$. The base case when $i=N$ holds by the definition of $\mathfrak a^+_0$. For the inductive step, suppose that $\widehat{w_i\cdot o}$ lies in $\mathfrak a^+_0$. Then \eqref{eqn:wilb} implies that the distance between $\widehat{w_i\cdot o}$ and any face of $\mathfrak a^+_0$ is at least $3\sqrt{2}C$. In particular, the ball in ${\bf F}$ of radius $3C$ centered at $\widehat{w_i\cdot o}$ is contained in $\mathfrak a^+_0$. Then observe that
\begin{align*}
d_X(\widehat{w_{i}\cdot o},\widehat{w_{i+1}\cdot o}) \leq d_X(\widehat{w_{i}\cdot o},w_{i}\cdot o)+d_X(w_{i}\cdot o,w_{i+1}\cdot o)+d_X(\widehat{w_{i+1}\cdot o},w_{i+1}\cdot o)\leq 3C,
\end{align*}
so $\widehat{w_{i+1}\cdot o}$ lies in $\mathfrak a^+_0$ as well. This proves the claim.

Next, let $v_{i,j}:=d_\mathfrak a^+(o,w_{i+j}\cdot o)-d_\mathfrak a^+(o,w_i\cdot o)$ and $\widehat{v}_{i,j}:=d_\mathfrak a^+(\widehat{o},\widehat{w_{i+j}\cdot o})-d_\mathfrak a^+(\widehat{o},\widehat{w_i\cdot o})$ for all integers $i,j\geq 0$. By the claim, $\widehat{w_i\cdot o}$ and $\widehat{w_{i+j}\cdot o}$ lie in $\mathfrak a^+_0$ for all $i\geq N$ and $j\geq 0$. Hence, we may apply (\ref{eqn: triangle}) to deduce that 
\begin{equation*}
\widehat{v}_{i,j}=\omega\cdot d_\mathfrak a^+(\widehat{w_i\cdot o},\widehat{w_{i+j}\cdot o})
\end{equation*}
for some $\omega$ in the Weyl group of $\PGL(V)$. Recall that the Weyl group action on $(\R^n)^*$ leaves the set of roots $\Phi$ of $\PGL(V)$ invariant. Thus, for every $\alpha\in\Delta$, there is a root $\beta \in \Phi$ such that 
\[\alpha(\widehat{v}_{i,j}) = \beta(d_\mathfrak a^+(\widehat{w_i\cdot o},\widehat{w_{i+j}\cdot o})).\] Since $\displaystyle\beta=\sum_{\epsilon\in\Delta} c_{\beta,\epsilon} \epsilon$, where $c_{\beta,\epsilon}$ are either all non-negative integers (when $\beta$ is a positive root) or all non-positive integers (when $\beta$ is a negative root), we have that
\begin{align}\label{eqn:app2}
|\alpha(\widehat{v}_{i,j})|&= | \beta(d_\mathfrak a^+(\widehat{w_i\cdot o},\widehat{w_{i+j}\cdot o}) )| \nonumber\\
 &= \sum_{\epsilon\in\Delta} |c_{\beta,\epsilon}| \epsilon(d_\mathfrak a^+(\widehat{w_i\cdot o},\widehat{w_{i+j}\cdot o})) \\
& \geq \alpha'(d_\mathfrak a^+(\widehat{w_i\cdot o},\widehat{w_{i+j}\cdot o})), \nonumber 
\end{align}
where $\alpha'\in\Delta$ is a simple root with the property that $c_{\beta,\alpha'}\neq 0$. The second equality holds because $d_\mathfrak a^+(\widehat{w_i\cdot o},\widehat{w_{i+j}\cdot o})$ lies in $\mathfrak a^+$, so $\epsilon(d_\mathfrak a^+(\widehat{w_i\cdot o},\widehat{w_{i+j}\cdot o}))\geq 0$ for all $\epsilon\in\Delta$. 

By (\ref{eqn: i,j}),
\begin{align*}
\|v_{i,j}-\widehat{v}_{i,j}\|&\leq\|d_\mathfrak a^+(o,w_{i+j}\cdot o)-d_\mathfrak a^+(\widehat{o},\widehat{w_{i+j}\cdot o})\|+\|d_\mathfrak a^+(o,w_i\cdot o)-d_\mathfrak a^+(\widehat{o},\widehat{w_i\cdot o})\|\leq 4C.
\end{align*}
It follows that for all $\alpha\in\Delta$, and all integers $i,j\geq 0$, we have
\begin{align}\label{eqn:gnmgnapp}
|\alpha(v_{i,j}-\widehat{v}_{i,j})| \leq \|\alpha\|4C=4\sqrt{2}C.
\end{align}
Combining (\ref{eqn: norm inequal}), (\ref{eqn:app2}), and (\ref{eqn:gnmgnapp}) gives
\[\left|\alpha(v_{i,j}) \right|  \geq \left|  \alpha( \widehat{v}_{i,j})\right| -4\sqrt{2}C \geq \alpha'( d_\mathfrak a^+(\widehat{w_i\cdot o},\widehat{w_{i+j}\cdot o}))-4\sqrt{2}C \geq \alpha'(d_\mathfrak a^+(w_i\cdot o,w_{i+j}\cdot o))-6\sqrt{2}C.\]
This implies that 
\[\left|\alpha\circ\mu(w_{i+j})-\alpha\circ\mu(w_i) \right|=|\alpha(v_{i,j})| \geq \alpha'(\mu(w_i^{-1}w_{i+j}))-6\sqrt{2}C\]
for all integers $i\geq N$ and $j\geq 0$. Set $A=6\sqrt{2}C$.
\end{proof}

\begin{lemma}\label{lem: mnmmnL}
Let $\mathcal{W}$ be a regulated and uniformly well-behaved family of rooted sequences in $\PGL(V)$. 
Let $N$ be the constant of Lemma \ref{lem: div1}.
For any constant $L>0$, there is an integer $M=M(L)>0$ such that 
\[\alpha\circ\mu(w_{i+j})-\alpha\circ\mu(w_i) \geq L\]
for any sequence $(w_i)_{i=0}^\infty$ in $\mathcal{W}$, any integers $i\geq N$ and $j\geq M$, and any $\alpha\in\Delta$.
\end{lemma}

\begin{proof}
Let $C>0$ be the constant that $\mathcal W$ is uniformly well-behaved with respect to, see Definition \ref{def: regulated}. First, we prove that there is an integer $M'=M'(L)>0$ such that 
\[|\alpha\circ\mu(w_{i+j})-\alpha\circ\mu(w_i)|\geq L\]
for any sequence $(w_i)_{i=0}^\infty$ in $\mathcal{W}$, any integers $i\geq N$ and $j\geq M'$, and any $\alpha\in\Delta$. 

Suppose for contradiction that this is not the case. Then for any integer $l>0$, there is a sequence $(w_{l,i})_{i=0}^\infty$ in $\mathcal{W}$, and integers $i_l\geq N$ and $j_l\geq l$, such that 
\[\left|\alpha\circ\mu(w_{l,i_l+j_l})-\alpha\circ\mu(w_{l,i_l}) \right|< L\]
for some fixed $\alpha\in\Delta$. By Lemma \ref{lem: div1}, there is some $\alpha'\in\Delta$ such that for all $l> 0$,
\[\left|\alpha\circ\mu(w_{l,i_l+j_l})-\alpha\circ\mu(w_{l,i_l}) \right|\geq \alpha'(\mu(w_{l,i_l}^{-1}w_{l,i_l+j_l}))-A,\]
which implies that
\[\alpha'(\mu(w_{l,i_l}^{-1}w_{l,i_l+j_l}))< L+A.\]
Since $\displaystyle\lim_{l\to\infty}j_l=\infty$, this contradicts the assumption that $\mathcal{W}$ is regulated.

By specializing to the case when $L=\sqrt{2}C+1$, the previous paragraph implies that there is an integer $M''>0$ such that
\begin{equation}\label{eqn: contradiction}
|\alpha\circ\mu(w_{i+j})-\alpha\circ\mu(w_i) |\geq \sqrt{2}C+1
\end{equation}
for any sequence $(w_i)_{i=0}^\infty$ in $\mathcal{W}$, any integers $i\geq N$ and $j\geq M''$, and any $\alpha\in\Delta$.
It now suffices to prove that
\[\alpha\circ\mu(w_{i+j})-\alpha\circ\mu(w_i)\geq 0\] 
for any sequence $(w_i)_{i=1}^\infty$ in $\mathcal{W}$, any integers $i\geq N$ and $j\geq M''$, and any $\alpha\in\Delta$; indeed, we simply set $M:=\max\{M',M''\}$. 

Suppose for contradiction that this is not the case. Then there exists a sequence $(w_i)_{i=0}^\infty$ in $\mathcal{W}$, integers $i_0\geq N$ and $j_0\geq M''$, and some $\alpha\in\Delta$ such that 
\begin{equation}\label{eqn: >0}
\alpha\circ\mu(w_{i_0+j_0})-\alpha\circ\mu(w_{i_0})<0.\
\end{equation}
By the first paragraph of this proof, there is an integer $K>0$ such that
\begin{align}\label{eqn: lower}
\left|\alpha\circ\mu(w_{i_0+j})-\alpha\circ\mu(w_{i_0}) \right| \geq \alpha\circ\mu(w_{i_0}).
\end{align}
for any integer $j\geq K$. On the other hand, since $\mu(w_{i_0+j})$ also lies in $\mathfrak a^+$, it follows that
\begin{align}\label{eqn: upper}
\alpha\circ\mu(w_{i_0+j})-\alpha\circ\mu(w_{i_0})\geq- \alpha\circ\mu(w_{i_0})
\end{align}
for any integers $j\geq 0$. The inequalities (\ref{eqn: lower}) and (\ref{eqn: upper}) together imply that
\begin{equation}\label{eqn: <0}
\alpha\circ\mu(w_{i_0+j})-\alpha\circ\mu(w_{i_0}) \geq \alpha\circ\mu(w_{i_0})\geq 0
\end{equation}
for any integers $j\geq K$. 

From (\ref{eqn: >0}) and (\ref{eqn: <0}), one then deduces that there exists some integer $j\geq j_0$ such that $\alpha\circ\mu(w_{i_0+j}) -\alpha\circ\mu(w_{i_0})<0$ but $\alpha\circ\mu(w_{i_0+j+1})-\alpha\circ\mu(w_{i_0})\geq 0$. Hence, 
\[0\leq \alpha\circ \mu(w_{i_0+j+1})-\alpha\circ\mu(w_{i_0}) <\alpha\circ \mu(w_{i_0+j+1})-  \alpha\circ\mu(w_{i_0+j}),\] 
which implies that
\[\alpha\circ \mu(w_{i_0+j+1})-\alpha\circ\mu(w_{i_0}) \leq  \|\alpha\| \| \mu(w_{i_0+j+1})-\mu(w_{i_0+j})\|.\]
Since 
\begin{align*}
\| \mu(w_{i_0+j+1}) - \mu (w_{i_0+j}) \| &=\| d_\mathfrak a^+(o, w_{i_0+j+1}\cdot o)- d_\mathfrak a^+(o,w_{i_0+j} \cdot o) \|\\
& \leq d_X(o,o)+d_X(w_{i_0+j+1}\cdot o, w_{i_0+j} \cdot o) \leq C,
\end{align*}
we see that
\[0\leq \alpha\circ \mu(w_{i_0+j+1})-\alpha\circ\mu(w_{i_0}) \leq \sqrt{2}C,\]
which contradicts (\ref{eqn: contradiction}) because $i_o\geq N$ and $j\geq j_0\geq M''$.
\end{proof}

\begin{proof}[Proof of Proposition \ref{prop: other definition}]
Let $N$ be the constant of Lemma \ref{lem: div1}, and let $C>0$ be the constant that $\mathcal W$ is uniformly well-behaved with respect to, see Definition \ref{def: regulated}. By Lemma \ref{lem: mnmmnL}, there is an integer $M>0$ such that 
\begin{equation}\label{eqn: main final 1}
\alpha\circ\mu(w_{i+M})-\alpha\circ\mu(w_i)\geq 1
\end{equation}
for all integers $i\geq N$ and all $\alpha\in\Delta$. Also, for all integers $i,j\geq 0$, we have
\begin{align*}
\| \mu(w_{i+j})-\mu(w_i) \|&\leq\sum_{l=0}^{j-1}\| \mu(w_{i+l+1})-\mu(w_{i+l}) \| =\sum_{l=0}^{j-1}\| d_\mathfrak a^+(o, w_{i+l+1}\cdot o)- d_\mathfrak a^+(o,w_{i+l} \cdot o) \|\\
& \leq \sum_{l=0}^{j-1}(d_X(o,o)+d_X(w_{i+l+1}\cdot o, w_{i+l} \cdot o))\leq jC.
\end{align*}
This implies that
\begin{align}\label{eqn: main final 2}
|\alpha\circ \mu(w_{i+j})-\alpha\circ\mu(w_i)|  \leq \sqrt{2}jC
\end{align}
for all $\alpha\in\Delta$. 

To finish the proof, we will prove that
\begin{align*}
\alpha\circ \mu  (w_{i+j})-\alpha\circ\mu(w_i)\geq \frac{j}{M}-\frac{M+N}{M}- \sqrt{2}C(M+N). 
\end{align*}
The proposition follows by setting $\kappa:=\frac{1}{M}$ and $\kappa':=\frac{M+N}{M}+ \sqrt{2}C(M+N)$. We proceed in three cases; when $i\geq N$, when $i< N$ and $i+j\geq N$, and when $i+j<N$.

When $i\geq N$, let $r\geq 0$ be the largest integer such that $Mr\leq j$. Then $j-Mr\leq M$, so \eqref{eqn: main final 1} and \eqref{eqn: main final 2} imply that
\begin{align*}
\alpha\circ\mu(w_{i+j})  -\alpha\circ\mu(w_i)&= \alpha\left(\mu(w_{i+j})-\mu(w_{i+Mr})+\sum_{p=1}^{r} (\mu(w_{i+Mp})-\mu(w_{i+M(p-1)}))\right) \\
&\geq r - \sqrt{2}(j-Mr)C\geq\frac{j}{M}-1 - \sqrt{2}MC\\
&\geq \frac{j}{M}-\frac{M+N}{M}- \sqrt{2}C(M+N).
\end{align*}
When $0\leq i<N$ and $i+j\geq N$, let $r$ be the largest integer such that $Mr\leq i+j-N$. Then $i+j-N-Mr\leq M$, so \eqref{eqn: main final 1} and \eqref{eqn: main final 2} imply that
\begin{align*}
\alpha\circ \mu & (w_{i+j})-\alpha\circ\mu(w_i)\\
&=\alpha\left(\mu(w_{i+j})-\mu(w_{N+Mr})+\sum_{p=1}^{r} (\mu(w_{N+Mp})-\mu(w_{N+M(p-1)}))+\mu(w_N)-\mu(w_i)\right) \\
&\geq\,\, r - \sqrt{2}(i+j-N-Mr)C-\sqrt{2}(N-i)C \geq \frac{i+j-M-N}{M}-\sqrt{2}C(M+N)\\
&\geq \frac{j}{M}-\frac{M+N}{M}-\sqrt{2}C(M+N). 
\end{align*}
Finally, when $i+j\leq N$, then $j\leq N$, so \eqref{eqn: main final 2} implies
\begin{align*}
\alpha\circ \mu (w_{i+j})-\alpha\circ\mu(w_i)&\geq-\sqrt{2}jC\geq\frac{j-N}{M}-\sqrt{2}jC\\
&\geq\frac{j}{M}-\frac{N}{M}-\sqrt{2}NC\\
&\geq \frac{j}{M}-\frac{M+N}{M}-\sqrt{2}C(M+N).\qedhere
\end{align*}
\end{proof}

\subsection{Proof of Theorem \ref{thm: main 2}}

Using Proposition \ref{prop: other definition}, we prove Theorem \ref{thm: main 2}.

\begin{proof}[Proof of Theorem \ref{thm: main 2}]
First, we show that
\[\mathcal{W}_{R,\Lambda}:=\{(\rho(\eta_i))_{i=0}^\infty:(\eta_i)_{i=0}^\infty\text{ is a }(R,\Lambda)\text{-directed sequence in }\Gamma\}\]
is a regulated and uniformly well-behaved collection of sequences in $\PGL(V)$. 

By the first hypothesis, there is some $C>0$ such that for every $(R,\Lambda)$-directed sequence $(\eta_i)_{i=0}^\infty$ in $\Gamma$, the sequence $(\rho(\eta_i)\cdot o)_{i=0}^\infty$ in $X$ is $C$-bounded from a maximal flat in $X$. By increasing $C$ if necessary, we may assume that 
\[C\geq\max\{d_X(o,\rho(\gamma)\cdot o):\gamma\in R(A)\}.\] 
Then
\[d_X(\rho(\eta_i)\cdot o,\rho(\eta_{i+1})\cdot o)=d_X(o,\rho(\eta_i^{-1}\eta_{i+1})\cdot o)\leq C\]
for any integer $i\geq 0$ and any sequence $(\eta_i)_{i=0}^\infty$ in $\mathcal{W}_{R,\Lambda}$. Thus, $\mathcal{W}_{R,\Lambda}$ is uniformly well-behaved.

Next, we show that $\mathcal{W}_{R,\Lambda}$ is regulated. Suppose for contradiction that it is not. Then there is some $D>0$ with the property that for every integer $i\geq 0$, there is a sequence $(\rho(\eta_{i,l}))_{l=0}^\infty$ in $\mathcal{W}_{R,\Lambda}$ and integers $l_i\geq 0$ and $k_i\geq i$ such that 
\[\alpha\circ\mu(\rho(\eta_{i,l_i}^{-1}\eta_{i,l_i+k_i}))< D\]
for some fixed $\alpha\in\Delta$. At the same time, for each $i\geq 0$, $\eta_{i,l_i}^{-1}\eta_{i,l_i+k_i}$ lies along the $(R,\Lambda)$-directed sequence $(\eta_{i,l_i}^{-1}\eta_{i,l_i+k})_{k=0}^\infty$, which contradicts the second hypothesis of the theorem.

Since $\mathcal{W}_{R,\Lambda}$ is regulated and uniformly well-behaved, Proposition \ref{prop: other definition} implies that there are constants $\kappa,\kappa'>0$ such that 
\[\alpha\circ\mu(\rho(\eta_{i+j}))- \alpha\circ\mu(\rho(\eta_i))\geq \kappa j-\kappa'.\]
for all $(R,\Lambda)$-directed sequences in $\Gamma$, all integers $i,j\geq 0$ and all $\alpha\in\Delta$. Specializing to the case when $i=0$ gives the theorem.
\end{proof}

\section{Positive flags in $\Fc(V)$.} \label{sec: positive}

In this section, we recall the notion of positivity for tuples of flags in $\mathcal{F}(V)$. This allows us to define the open sets of the form \eqref{eqn: open sets}, which we use later to define weakly positive representations. The main goal of this section is to prove Theorem \ref{thm: proj intro}, which is restated below as Theorem \ref{thm: general k}.

In Section \ref{sec: projective geometry} and Section \ref{sec: positivity}, we recall some basic facts from projective geometry and Fock-Goncharov positivity, before restating Theorem \ref{thm: general k} in Section \ref{sec: open sets}. The remainder of the section is the proof of Theorem \ref{thm: general k}.

\subsection{Projective geometry}\label{sec: projective geometry}
Let $\Pb(V)$ denote the \emph{projectivization of $V$}, i.e. $\Pb(V):=(V\setminus\{0\})/\R^\times$. If $W\subset V$ is a non-zero subspace, then $\Pb(W)$ is naturally a subset of $\Pb(V)$. We refer to all such subsets as \emph{projective subspaces} of $V$. In the case when $\Pb(W)$ is $1$-dimensional (resp. $(n-2)$-dimensional), $\Pb(W)$ is a \emph{projective line} (resp. \emph{projective hyperplane}) in $\Pb(V)$. 

Let $(p_1,\dots,p_n)$ be an $n$-tuple of points in $\Pb(V)$ that do not lie in a projective hyperplane. For all $k=1,\dots,n$, there is a unique projective hyperplane $H_k\in\Gr_{n-1}(V)$ that contains $\{p_1,\dots,p_{k-1},p_{k+1},\dots,p_n\}$. Then $\displaystyle p_k=\bigcap_{j\neq k}H_j$, and
\[\displaystyle\Pb(V)\setminus\left(\bigcup_{k=1}^nH_k\right)\] 
is a disconnected open set with $2^{n-1}$ connected components, each of which is called a \emph{simplex with vertices $p_1,\dots,p_n$.} 

For all $k=1,\dots,n-1$, let $\Gr_k(V)$ denote the space of $k$-dimensional linear subspaces of $V$. Then $\Gr_k(V)$ is canonically identified with the set of $(k-1)$-dimensional projective subspaces of $\Pb(V)$, so we may think of a point in $\Gr_k(V)$ both as a linear subspace of $V$ and as a projective subspace of $\Pb(V)$. 

A (complete) \emph{flag} in $V$ is a nested sequence of subspaces in $V$, one of each dimension. Let $\Fc(V)$ denote the space of flags in $V$. For all $k=1,\dots,n-1$, there is an obvious projection $\Theta_k:\Fc(V)\to\Gr_k(V)$. If $F$ is a flag in $\Fc(V)$, we denote $F^{(k)}:=\Theta_k(F)$ when $k=1,\dots,n-1$, $F^{(0)}:=\{0\}$, and $F^{(n)}:=V$. A finite collection of flags $\{F_1,\dots,F_l\}$ in $\Fc(V)$ is in \emph{general position} if $F_1^{(m_1)}+\dots+F_l^{(m_l)}$ is a direct sum for all integers $m_1,\dots,m_l\geq 0$ such that $\displaystyle\sum_{j=1}^lm_j\leq n$. If a pair of flags $\{F_1,F_2\}$ is in general position, we also say that they are \emph{transverse}. 

In this article, we will mainly be interested in simplices whose vertices come from a pair of transverse flags. For convenience, we use the following terminology.

\begin{definition}
For any transverse pair of flags $\{F_1,F_2\}$, let $p_k:=F_1^{(k)}\cap F_2^{(n-k+1)}$ for all $k=1,\dots,n$. We say that a simplex is \emph{associated to $\{F_1,F_2\}$} if its vertices are $p_1,\dots,p_n$. 
\end{definition}

If $g\in\PGL(V)$ is a loxodromic element whose eigenvalues are all of the same sign, then we say that $g$ is \emph{positive loxodromic}. The following observation is immediate.

\begin{observation}\label{obs: lox}
An element $g\in\PGL(V)$ is loxodromic if and only if $g$ has a unique attracting and repelling fixed point in $\Fc(V)$, which we denote by $g_+$ and $g_-$ respectively. In this case, $\{g_+,g_-\}$ is a transverse pair of flags. Furthermore, $g$ is positive loxodromic if and only if $g$ leaves some (equiv. every) simplex associated to $\{g_+,g_-\}$ invariant.
\end{observation}

\subsubsection{Cross ratio}\label{sec: cross ratio}

For any integer $k$ such that $1\leq k\leq n-1$, define 
\[\mathfrak{Q}_k(V):=\{(U_1,U_2,W_1,W_2)\in\Gr_{n-k}(V)^2\times\Gr_k(V)^2:U_i+W_j=V \text{ for all }i,j=1,2\}.\] 

\begin{definition}
The $k$-th \emph{cross ratio} is the function $C_k:\mathfrak{Q}_k(V)\to\R\setminus\{0\}$ defined as follows. For any $(U_1,U_2,W_1,W_2)\in\mathfrak{Q}_k(V)$, choose a basis $\{u_{i,1},\dots,u_{i,n-k}\}$ of $U_i$, and a basis $\{w_{j,1},\dots,w_{j,k}\}$ of $W_j$. Also, choose a linear identification $\displaystyle\Omega:\bigwedge^nV\to\R$. Then
\[C_k(U_1,U_2,W_1,W_2):=\frac{\Omega(u_{2,1},\dots,u_{2,n-k},w_{2,1},\dots,w_{2,k})\Omega(u_{1,1},\dots,u_{1,n-k},w_{1,1},\dots,w_{1,k})}{\Omega(u_{2,1},\dots,u_{2,n-k},w_{1,1},\dots,w_{1,k})\Omega(u_{1,1},\dots,u_{1,n-k},w_{2,1},\dots,w_{2,k})}.\]
\end{definition}

It is straightforward to verify that $C_k$ is well-defined, and does not depend on any of the choices made. Furthermore, $C_k$ is continuous, and the following identities hold:
\begin{itemize}
\item $C_k(U_1,U_2,W_1,W_2)=C_k(U_2,U_1,W_2,W_1)$,
\item $C_k(U_1,U_2,W_1,W_2)=C_{n-k}(W_1,W_2,U_1,U_2)$, and
\item $C_k(U_1,U_2,W_1,W_2)\cdot C_k(U_1,U_2,W_2,W_3)=C_k(U_1,U_2,W_1,W_3)$.
\end{itemize}

The following observations follow easily from the definition of the cross ratio.

\begin{observation}\label{obs: projective coordinates}
Let $U_1,\dots,U_n\in\Gr_{n-1}(V)$ such that $\displaystyle\bigcap_{k=1}^nU_k=\{0\}$, and let $W_1,W_2\in\Gr_1(V)$ such that $U_k\cap W_j=\{0\}$ for all $k=1,\dots,n$ and all $j=1,2$. Then $W_1=W_2$ if and only if $C_1(U_k,U_{k+1},W_1,W_2)=1$ for all $k=1,\dots,n-1$. 
\end{observation}

\begin{observation}\label{obs: cross ratio projection}
Let $(U_1,U_2,W_1,W_2)\in\mathfrak{Q}_k(V)$, and let $W\subset W_1\cap W_2$ be a subspace such that $l:=\dim(W)<k$. If $\pi:V\to V/W$ is the obvious quotient map, then 
\[C_{k-l}(\pi(U_1),\pi(U_2),\pi(W_1),\pi(W_2))=C_k(U_1,U_2,W_1,W_2),\]
where $C_{k-l}$ on the left and $C_k$ on the right are cross ratios on $\mathfrak{Q}_{k-l}(V/W)$ and $\mathfrak{Q}_k(V)$ respectively.
\end{observation}

In the case when $n:=\dim(V)=2$, the cross ratio $C_1$ can also be described as follows. Let $(p_1,p_2,q_1,q_2)$ be a quadruple of points in $\mathbb{P}(V)$ such that $p_i\neq q_j$ for all $i,j=1,2$. Choose an affine chart $\mathbb{A}$ of $\mathbb{P}(V)$ that contains $p_1$, $p_2$, $q_1$, and $q_2$, and choose an affine identification $\mathbb{R}\simeq\mathbb{A}$. Then
\begin{equation}\label{eqn: cross ratio classical}
C_1(p_1,p_2,q_1,q_2)=\frac{(q_2-p_2)(p_1-q_1)}{(q_1-p_2)(p_1-q_2)}.
\end{equation}

A straightforward computation gives the following observation.

\begin{observation}\label{obs: RP1}
Suppose that $n=2$. Let $a,p_1,p_2,q_1,q_2$ be a quintuple of points along the projective line $\mathbb{P}(V)$, such that $a<p_2<q_1<q_2<p_1<a$. Let $\mathbb{A}:=\mathbb{P}(V)\setminus\{a\}$, and choose an affine identification $\mathbb{A}\simeq\mathbb{R}$. If $C_1(p_1,p_2,q_1,q_2)\leq D$ for some $D>1$, then 
\[\left|\frac{q_1-q_2}{p_1-p_2}\right|\leq \frac{\sqrt{D}}{1+\sqrt{D}}.\]
\end{observation}

The cross ratio will mainly be applied to flags in the following two ways. 

\begin{definition} \label{def: cross ratio}
Let $(F_1,F_2,F_3,F_4)$ be a quadruple of flags in $\Fc(V)$.
\begin{enumerate}
\item If $\{F_1,F_2,F_3,F_4\}$ is pairwise transverse, then for all $k=1,\dots,n-1$, define the \emph{$k$-th Labourie cross ratio}
\[B_k(F_1,F_2,F_3,F_4):=C_k(F_1^{(n-k)},F_2^{(n-k)},F_3^{(k)},F_4^{(k)})\]
\item If $\{F_1,F_2,F_3,F_4\}$ is in general position, then for all $k=1,\dots,n-1$, define the \emph{$k$-th edge invariant}
\begin{equation*}
S_k(F_1,F_2,F_3,F_4):=C_1(F_1^{(k)}+F_3^{(n-k-1)},F_1^{(k-1)}+F_3^{(n-k)},F_2^{(1)},F_4^{(1)}).
\end{equation*}
\end{enumerate}
\end{definition}

\subsubsection{Triple ratio}\label{sec: triple ratio}

For $i=1,2,3$, let $U_i$ be a hyperplane in $V$, and let $U_i'\in\Gr_{n-2}(V)$ be a $(n-2)$-dimensional subspace of $U_i$. We say that the sextuple $(U_1,U_2,U_3,U_1',U_2',U_3')$ is \emph{well-positioned} if
\begin{itemize}
\item $W:=U_1\cap U_2\cap U_3$ is a $(n-3)$-dimensional subspace of $V$,
\item $W\subset U_i'$ for all $i=1,2,3$, and
\item $U_i'$ does not lie in $U_j$ for all $i\neq j$.
\end{itemize}
Let $\mathfrak{W}(V)$ denote the set of well-positioned sextuples in $\Gr_{n-1}(V)^3\times\Gr_{n-2}(V)^3$.

\begin{definition}
The \emph{triple ratio} is the function $T:\mathfrak{W}(V)\to\R\setminus\{0\}$ defined as follows. For any ${\bf U}:=(U_1,U_2,U_3,U_1',U_2',U_3')\in\mathfrak{W}(V)$, choose a basis $w_1,\dots,w_{n-3}$ for $W:=U_1\cap U_2\cap U_3$. Also, for $i=1,2,3$, let $u_i'$ be a vector in $U_i'$ that is not in $W$, and let $u_i$ be a vector in $U_i$ that is not in $U_i'$. Finally choose a linear identification $\displaystyle\Omega:\bigwedge^n V\to\R$. Then define
\[T({\bf U}):=\frac{\Omega(w_1,\dots,w_{n-3},u_1',u_1,u_2')\Omega(w_1,\dots,w_{n-3},u_2',u_2,u_3')\Omega(w_1,\dots,w_{n-3},u_3',u_3,u_1')}{\Omega(w_1,\dots,w_{n-3},u_2',u_2,u_1')\Omega(w_1,\dots,w_{n-3},u_3',u_3,u_2')\Omega(w_1,\dots,w_{n-3},u_1',u_1,u_3')}.\]
\end{definition}

As in the case of the cross ratio, one can verify that $T$ is well-defined, and does not depend on any of the choices made. It is clear that $T$ is continuous, and the following identities hold:
\begin{itemize}
\item $T(U_1,U_2,U_3,U_1',U_2',U_3')=T(U_2,U_3,U_1,U_2',U_3',U_1')$, and
\item $T(U_1,U_2,U_3,U_1',U_2',U_3')=T(U_3,U_2,U_1,U_3',U_2',U_1')^{-1}$.
\end{itemize}

We can also apply the triple ratios to flags in the following way. If $\mathbf j:=(j_1,j_2,j_3)$ is a triple of positive integers that sum to $n$, and $\{F_1,F_2,F_3\}$ is a triple of flags in $\Fc(V)$ that are in general position, then we define the \emph{triangle invariant}
\begin{equation}\label{eqn: triangle invariant}
T_{\mathbf j}(F_1,F_2,F_3):=T(U_1,U_2,U_3,U_1',U_2',U_3'),
\end{equation}
where $U_i':=F_{i-1}^{(j_{i-1}-1)}+F_i^{(j_i)}+F_{i+1}^{(j_{i+1}-1)}$ and $U_i:=F_{i-1}^{(j_{i-1}-1)}+F_i^{(j_i+1)}+F_{i+1}^{(j_{i+1}-1)}$ for $i=1,2,3$. Here, arithmetic involving $i$ is done modulo $3$.

\subsection{Positivity of tuples of flags}\label{sec: positivity}
Next, we recall the notion of positive tuple of flags introduced in Fock-Goncharov \cite{Fock-Goncharov}. 

\subsubsection{Total positivity}\label{sec: total positivity}
First, we need the notion of total positivity for a unipotent element in $\PGL(V)$. 

\begin{definition}
A unipotent element $u\in\PGL(V)$ is \emph{totally positive} with respect to a basis $\mathcal{B}:=(e_1,\dots,e_n)$ of $V$ if in this basis, $u$ is represented by an upper-triangular matrix $M_u$ with ones on the diagonal, and all the minors of $M_u$ are positive except for those that are forced to be zero by virtue of $M_u$ being upper triangular. We denote by $U_{>0}(\mathcal{B})$ the set of such elements, and let $U_{\geq 0}(\mathcal{B}):=\overline{U_{>0}(\mathcal{B})}$.
\end{definition}

\begin{remark}
Upper triangular, real valued matrices $M$ with ones along the diagonal, and where all minors are positive except for those that are forced to be zero by virtue of $M$ being upper triangular, are called \emph{totally positive upper triangular matrices}. These were introduced by Lusztig \cite{Lusztig}. For our purposes, we will often be changing the basis with which our unipotent elements in $\PGL(V)$ are upper triangular with respect to. This choice of a basis is equivalent to the choice of a \emph{pinning} as described in Lusztig \cite{Lusztig}. In this paper, it is more convenient to work with bases as opposed to pinnings, so we use the notion of totally positive with respect to a chosen basis.
\end{remark}

Let $(f_1,f_2)$ be any basis of $\mathbb R^2$. Then $(f_1^{n-1-i}\cdot f_2^i)_{i=0}^{n-1}$ is a basis of the $(n-1)$-th symmetric tensor of $\R^2$, denoted $\mathrm{Sym}^{n-1}(\R^2)$. By choosing a basis $\mathcal{B}:=(e_1,\dots,e_n)$ for $V$, we may linearly identify $V$ with $\mathrm{Sym}^{n-1}(\R^2)$ by identifying $e_i$ with $f_1^{n-1-i}\cdot f_2^i$. Observe that the $\GL_2(\R)$-action on $\R^2$ induces a linear $\GL_2(\R)$-action on $\mathrm{Sym}^{n-1}(\R^2)$ given by
\[\bar{g}\cdot(v_1\cdot ...\cdot v_{n-1}):=(\bar{g}\cdot v_1)\cdot\dots\cdot(\bar{g}\cdot v_{n-1}).\]
Thus, we have a linear representation 
\[i_{\mathcal{B}}:\GL_2(\R)\to\GL(\mathrm{Sym}^{n-1}(\R^2))\simeq\GL(V).\] 
Projectivizing this gives a representation $\iota=\iota_{\mathcal{B}}:\PGL_2(\R)\to\PGL(V)$.

\begin{remark}
It is a standard result from the representation theory that up to post-composition with an automorphism of $\PGL(V)$, $\iota$ restricted to $\PSL(2,\R)$ is the unique irreducible representation from $\PSL_2(\R)$ to $\PGL(V)$.
\end{remark}

The following is a well-known proposition, which gives the simplest examples of elements in $U_{>0}(\mathcal{B})$. See \cite[Proposition 5.7]{Fock-Goncharov} for a proof.

\begin{proposition}\label{prop: Veronese}
Let $\mathcal{B}$ be a basis of $V$. Then 
\[\iota_{\mathcal{B}}(U_{>0}(f_1,f_2))\subset U_{>0}(\mathcal{B}).\]
\end{proposition}

We now describe a natural partial order on $U_{\geq 0}(\mathcal{B})$. For any $k=1,\dots,n$, let $i_1,\dots,i_k,j_1,\dots,j_k$ be positive integers such that 
\[1\leq i_1<\dots<i_k\leq n\,\,\,\text{ and }\,\,\,1\leq j_1<\dots<j_k\leq n.\]
Then define the $(i_1,\dots,i_k),(j_1,\dots,j_k)$-\emph{minor map} to be the map
\[\varepsilon^{i_1,\dots,i_k}_{j_1,\dots,j_k}:U_{\geq 0}(\Bc)\to\R\]
that assigns to every $u$ in $U_{\geq 0}(\mathcal{B})$ the minor of $M_u$ corresponding to the $i_1,...,i_k$ rows and the $j_1,...,j_k$ columns. Since $M_u$ is an upper triangular matrix with $1$'s along the diagonal, observe that
\begin{itemize}
\item if $j_s<i_s$ for some $s=1,\dots,k$, then the image of $\varepsilon^{i_1,\dots,i_k}_{j_1,\dots,j_k}$ is $0$, and
\item if $i_s=j_s$ for all $s=1,\dots,k$, then the image of $\varepsilon^{i_1,\dots,i_k}_{j_1,\dots,j_k}$ is $1$.
\end{itemize}
As such, we say that the $(i_1,\dots,i_k),(j_1,\dots,j_k)$-\emph{minor map} is
\begin{itemize}
\item \emph{trivial} if $j_s<i_s$ for some $s=1,\dots,k$, or $i_s=j_s$ for all $s=1,\dots,k$.
\item \emph{non-trivial} if $j_s\geq i_s$ for all $s=1,\dots,k$, and $j_s>i_s$ for some $s=1,\dots,k$.
\end{itemize}

\begin{definition}
If $u,v\in U_{\geq 0}(\mathcal{B})$, we say that $u$ \emph{precedes} $v$, denoted $u\prec v$, if $\varepsilon(u)<\varepsilon(v)$ for every non-trivial minor map $\varepsilon$. 
\end{definition}

The following observation is an easy consequence of the definition of $\prec$ and the continuity of the minors.

\begin{observation}\label{obs: prec}
Let $\mathcal{B}$ be a basis of $V$ and let $(u_i)_{i=1}^\infty$ be a sequence in $U_{\geq 0}(\mathcal{B})$.
\begin{enumerate}
\item If there is some $v\in U_{>0}(\mathcal{B})$ such that $u_i\prec v$ for all integers $i>0$, then there is a subsequence of $(u_i)_{i=1}^\infty$ that converges in $U_{\geq 0}(\mathcal{B})$. Furthermore, if there is some $v,w\in U_{>0}(\mathcal{B})$ such that $w\prec u_i\prec v$ for all integers $i>0$, then there is a subsequence of $(u_i)_{i=1}^\infty$ that converges in $U_{>0}(\mathcal{B})$.
\item If there is some $v\in U_{>0}(\mathcal{B})$ such that $u_i\prec u_{i+1}\prec v$ for all integers $i>0$, then $(u_i)_{i=1}^\infty$ converges in $U_{>0}(\mathcal{B})$.
\end{enumerate}
\end{observation}

For all $k=1,\dots,n-1$, there is a natural $\GL(V)$-action on $\displaystyle\bigwedge^kV$ defined by \
\[g\cdot (v_1\wedge\dots\wedge v_k)=(g\cdot v_1)\wedge\dots\wedge(g\cdot v_k).\] Furthermore, the basis $\mathcal{B}=(e_1,\dots,e_n)$ induces a basis $\{e_{i_1}\wedge\dots\wedge e_{i_k}:1\leq i_1<\dots<i_k\leq n\}$ of $\displaystyle\bigwedge^kV$. A straightforward calculation proves that for all basis elements $e_{j_1}\wedge\dots\wedge e_{j_k}$ and for all $u\in U_{>0}(\mathcal{B})$, we have
\[u\cdot (e_{j_1}\wedge\dots\wedge e_{j_k})=\sum_{1\leq i_1<\dots<i_k\leq n}\varepsilon^{i_1,\dots,i_k}_{j_1,\dots,j_k}(u)e_{i_1}\wedge\dots\wedge e_{i_k},\]
where $\varepsilon^{i_1,\dots,i_k}_{j_1,\dots,j_k}:U_{>0}(\mathcal{B})\to\R$ is the $(i_1,\dots,i_k),(j_1,\dots,j_k)$-minor map. It follows that
\[\varepsilon^{i_1,\dots,i_k}_{j_1,\dots,j_k}(uv)=\sum_{1\leq l_1<\dots<l_k\leq n}\varepsilon^{i_1,\dots,i_k}_{l_1,\dots,l_k}(u)\varepsilon^{l_1,\dots,l_k}_{j_1,\dots,j_k}(v).\]
From this, we deduce the following observation.

\begin{observation}\label{obs: increasing minors}
Let $\mathcal{B}$ be a basis of $V$. If $u\in U_{>0}(\mathcal{B})$ and $v\in U_{\geq 0}(\mathcal{B})$, then $uv,vu\in U_{>0}(\mathcal{B})$, and $u,v\prec uv,vu$.
 \end{observation}

\subsubsection{Positive tuples of flags} 

\begin{definition}\label{def: positive flags}
Let $l\geq 3$, and let $(F_1,\dots,F_l)$ be an $l$-tuple of flags in $\Fc(V)$ such that $F_1$ and $F_l$ are transverse. We say that $(F_1,\dots,F_l)$ is \emph{positive} if there is
\begin{itemize}
\item an ordered basis $\mathcal{B}:=(e_1,\dots,e_n)$ of $V$ such that $e_k\in F_1^{(n-k+1)}\cap F_l^{(k)}$ for all $k=1,\dots,n$, 
\item a (necessarily unique) $(l-2)$-tuple of elements $(u_1,\dots,u_{l-2})$ in $U_{>0}(\mathcal{B})$, 
\end{itemize}
such that $(F_1,\dots,F_l)=(F_1,u_1\cdot F_1,u_1u_2\cdot F_1,\dots,u_1u_2\cdots u_{l-2}\cdot F_1,F_l)$.
\end{definition} 

Observe that when $n=2$, $\Fc(V)=\mathbb{P}(V)$ has two natural cyclic orders, which are reverses of each other. Then a tuple of flags $(F_1,\dots,F_k)$ in $\Fc(V)$ is positive if and only $F_1<\dots<F_k<F_1$ in either of the cyclic orders on $\mathbb{P}(V)$. For general $n$, it is well-known that every positive tuple of flags is in general position; a proof can be found in \cite[Appendix A]{SunWienhardZhang}.

The following is an example of a positive tuples of flags in $\Fc(V)$ for general $V$.

\begin{example}\label{eq: Veronese}
Recall that in Section \ref{sec: total positivity}, we defined, using the basis $\mathcal{B}$, a homomorphism $\iota=\iota_{\mathcal B}:\PGL_2(\R)\to\PGL(V)$. Let $F_\pm\in\Fc(V)$ be the flags defined by $F_+^{(i)}=\Span_\R(e_1,\dots,e_i)$ and $F_-^{(i)}=\Span_\R(e_{n-i+1},\dots,e_n)$ for all $i=1,\dots,n-1$, and let $\nu:\mathbb P(\R^2)\to\mathbb P(V)$ be the map given by 
\begin{align*}
\nu:\begin{bmatrix}1\\0\end{bmatrix}\mapsto&\,\,F_+,\\
\nu:\begin{bmatrix}x\\1\end{bmatrix}\mapsto&\,\,\iota\left(\begin{bmatrix}1&x\\0&1\end{bmatrix}\right)\cdot F_-.
\end{align*}
It is straightforward to check that $\nu$ is $\iota$-equivariant. Thus it follows from Proposition~\ref{prop: Veronese} that if $x_1<\dots<x_k<x_1$ is a $k$-tuple of points in the cyclic order on $\mathbb P(\R^2)$  or its reverse, then $(\nu(x_1),\dots,\nu(x_k))$ is a positive tuple of flags. 
\end{example}

Using Observation \ref{obs: prec}, we may deduce the following.

\begin{observation}\label{obs: positive converge}
Let $(F_i)_{i=1}^\infty$ be a sequence of flags in $\Fc(V)$. 
\begin{enumerate}
\item If $(F_i)_{i=1}^\infty$ converges to some flag $F_\infty\in\Fc(V)$ and there is some $F,H,K,G\in\Fc(V)$ such that $(F,H,F_i,K,G)$ is a positive tuple of flags for all integers $i\geq 1$, then $(F,F_\infty,G)$ is positive.
\item If there is some $F\in\Fc(V)$ such that $(F_1,\dots,F_i,F)$ is positive for all integers $i>1$, then $(F_i)_{i=1}^\infty$ converges to some flag $F_\infty$ in $\Fc(V)$. 
\end{enumerate}
\end{observation}



The following is a well-known theorem due to Fock-Goncharov \cite{Fock-Goncharov}. It gives a coordinate-free description of a positive tuple of flags in terms of the edge and triangle invariants defined in Section \ref{sec: cross ratio} and Section \ref{sec: triple ratio} respectively. 

\begin{theorem}\cite[Theorem 9.1(a)]{Fock-Goncharov}\label{thm: Fock-Goncharov}
Let $(F_1,\dots,F_l)$ be an $l$-tuple of flags. Let $M$ be a convex planar polygon with $l$ vertices, $v_1<\dots<v_l<v_1$ in this cyclic order along the boundary of $M$, and let $F_{v_l}:=F_l$. Choose a triangulation $\Tmc$ of $M$, where the vertices of each triangle of $\Tmc$ is a vertex of $M$. For each triangle $S$ of $\Tmc$, let $(v_{S,1},v_{S,2},v_{S,3})$ be the vertices of $S$. Also, for each interior edge $e$ of $\Tmc$, let $v_{e,1}$ and $v_{e,2}$ be the endpoints of $e$, and let $u_{e,1}$ and $u_{e,2}$ be the vertices of $M$ such that both $(u_{e,1},v_{e,1},v_{e,2})$ and $(u_{e,2},v_{e,1},v_{e,2})$ are vertices of some triangle of $\Tmc$. (See Figure \ref{fig: FG}.) Then $(F_1,\dots,F_l)$ is positive if and only if both of the following statements hold:
\begin{enumerate}
\item For all triples of positive integers $\mathbf j:=(j_1, j_2, j_3)$ that sum to $n$, and all triangles $S$ of $\Tmc$,
\[T_\mathbf j(F_{v_{S,1}},F_{v_{S,2}},F_{v_{S,3}})>0.\]
\item For all $k=1,\dots,n-1$ and all edges $e$ of $\Tmc$,
\[S_k(F_{v_{e,1}},F_{u_{e,1}},F_{v_{e,2}},F_{u_{e,2}})<0.\]
\end{enumerate}
\end{theorem}

\begin{figure}[h]
    \centering
    \includegraphics[width=0.5\textwidth]{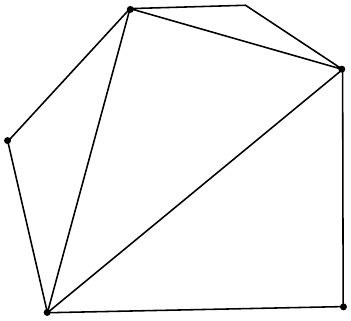}
    \small
\put (-65,55){$S$}
\put (-149,100){$e$}
\tiny
\put (-10,3){$v_{S,1}$}
\put (-214,2){$v_{S,2}=v_{e,1}$}
\put (-2,145){$v_{S,3}=u_{e,2}$}
\put (-140,186){$v_{e,2}$}
\put (-216,102){$u_{e,1}$}
    \caption{Triangulation $\mathcal T$ in Theorem \ref{thm: Fock-Goncharov}.}
    \label{fig: FG}
\end{figure}
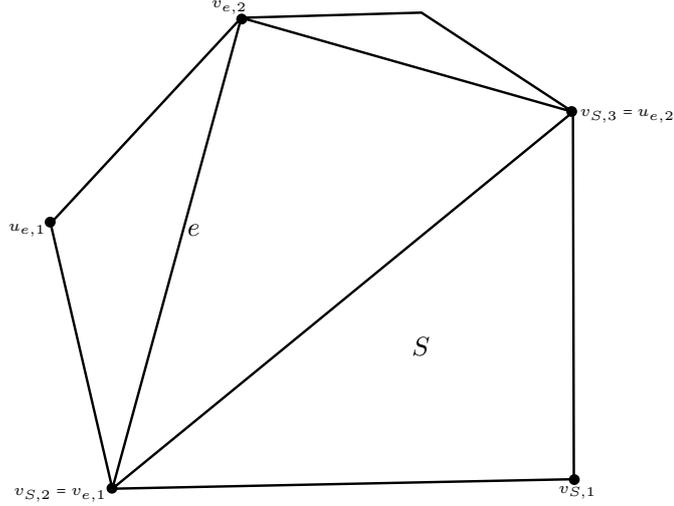

By the symmetries of the triple ratio described in Section \ref{sec: triple ratio}, the positivity of $T_\mathbf j(F_{v_{S,1}},F_{v_{S,2}},F_{v_{S,3}})$ does not depend on how we label the vertices of $S$. As an immediate consequence of Theorem \ref{thm: Fock-Goncharov}, we have the following observation.

\begin{observation}\label{obs: basic flag}
The following are equivalent
\begin{enumerate}
\item $(F_1,F_2,\dots,F_l)$ is a positive $l$-tuple of flags.
\item $(g\cdot F_1,g\cdot F_2,\dots,g\cdot F_l)$ is a positive $l$-tuple of flags for some $g\in\PGL(V)$.
\item $(F_l,\dots,F_2,F_1)$ is a positive $l$-tuple of flags.
\item $(F_2,F_3,\dots,F_l,F_1)$ is a positive $l$-tuple of flags.
\item $(F_1,\dots,F_{l-1})$ is a positive $(l-1)$-tuple of flags and $(F_1,F_i,F_{l-1},F_l)$ is a positive quadruple of flags for some, or equivalently, all $i=2,\dots,l-2$.
\item $(F_1,g\cdot F_2,F_3,\dots,F_l)$ is a positive $l$-tuple of flags for all positive loxodromic $g\in\PGL(V)$ that fixes $F_1$ and $F_3$.
\end{enumerate}
\end{observation}

\subsection{Collapsing sequence of domains} \label{sec: open sets}

Recall that in \eqref{eqn: open sets}, we defined for any positive triple of flags $(F,H,K)$, the open set $\mathfrak U(F,H,K)$. Recall also that for all $k=1,\dots,n-1$, $B_k$ denotes the $k$-th Labourie cross ratio as defined in Definition \ref{def: cross ratio}. With this, we can state the main theorem of this section.

\begin{theorem}\label{thm: general k}
Let $(F_i)_{i=1}^\infty$ and $(H_i)_{i=1}^\infty$ be sequences of flags in $\Fc(V)$ and $K\in\Fc(V)$ such that $(F_1,\dots,F_l,H_l,\dots,H_1,K)$ is a positive tuple of flags for all positive integers $l$. If there is some $D>1$ such that $B_k(H_i,F_i,F_{i+1},H_{i+1})\leq D$ for all integers $i>0$ and all $k=1,\dots,n-1$, then $\displaystyle\lim_{i\to\infty}\overline{\mathfrak U(F_i,H_i,K)}$ is a point. 
\end{theorem}

In the case when $n=2$, $\mathfrak{U}(F,H,K)$ is the open subinterval of $\Pb(V)$ with endpoints $F$ and $H$ that does not contain $K$, and the proof of Theorem \ref{thm: general k} is a straightforward cross ratio computation. However, for general $V$, the proof of Theorem \ref{thm: general k} is more involved; it requires understanding properties of $B_k$ evaluated along positive tuples of flags, and understanding how positivity is preserved under taking certain quotients. The remainder of this section is the proof of Theorem \ref{thm: general k}.

\subsection{Preliminary properties of $\mathfrak U(F,H,K)$.}
As a first step, we prove some basic properties of $\mathfrak{U}(F,H,K)$ given in the following proposition. Observe from the definition that $\mathfrak{U}(F,H,K)\subset\Fc(V)$ is connected and open, $G\in\mathfrak{U}(F,H,K)$ if and only if $K\in\mathfrak{U}(F,H,G)$.\\

\begin{proposition} \label{prop: easy positive}
\begin{enumerate}
\item The subset 
\[\mathfrak{U}(F,H,K)\subset\{G\in\Fc(V):(F,G,H)\text{ is in general position}\}\] 
is a connected component. In particular, $\mathfrak{U}(F,H,G)=\mathfrak{U}(F,H,G')$ for all $G,G'\in\mathfrak{U}(F,H,K)$.
\item If $\mathfrak{U}(F_2,H_2,K_2)\subset\mathfrak{U}(F_1,H_1,K_1)$ and $(F_1,F_2,H_2,H_1)$ is positive, then $(F_1,F_2,H_2,H_1,K_1)$ is positive and $\mathfrak{U}(F_2,H_2,K_1)=\mathfrak{U}(F_2,H_2,K_2)$. In particular, if we choose $G_i\in\mathfrak U(F_i,H_i,K_i)$ for $i=1,2$, then $\mathfrak{U}(F_1,H_1,G_1)\subset\mathfrak{U}(F_2,H_2,G_2)$.
\item If $\mathfrak{U}(F_2,H_2,K_2)\subset\mathfrak{U}(F_1,H_1,K_1)$ and $(F_1,F_2,H_2,H_1)$ is positive, then $\overline{\mathfrak U(F_2,H_2,K_2)}\subset\mathfrak U(F_1,H_1,K_1)$.
\end{enumerate}
\end{proposition}


\begin{proof}
Proof of (1). Since $\mathfrak{U}(F,H,K)$ is open and connected, to prove the first statement, it is sufficient to prove that 
\[\mathfrak{U}(F,H,K)\subset\{G\in\Fc(V):(F,G,H)\text{ is in general position}\}\] 
is closed. Consider a sequence of flags $(G_i)_{i=1}^\infty$ in $\mathfrak{U}(F,H,K)$ that converges to some $G\in\Fc(V)$, such that $(F,G,H,K)$ is in general position. We need to show that $G\in\mathfrak{U}(F,H,K)$.

For all integers $i>0$, Theorem \ref{thm: Fock-Goncharov} implies that $T_\mathbf j(F,G_i,K)>0$ and $T_\mathbf j(G_i,H,K)>0$ for all triples $\mathbf j$ of positive integers that sum to $n$, and $S_k(G_i,F,K,H)<0$ for all $k=1,\dots,n-1$. At the same time, since $(F,G,H,K)$ is in general position, the quantities $T_\mathbf j(F,G,K)$ and $T_\mathbf j(G,H,K)$ for all triples $\mathbf j$ of positive integers that sum to $n$, and $S_k(G,H,K,F)$ for all $k=1,\dots,n-1$, are well-defined and non-zero. Thus, by the continuity of the triangle and edge invariants, $T_\mathbf j(F,G,K)>0$ and $T_\mathbf j(G,H,K)>0$ for all triples $\mathbf j$ of positive integers that sum to $n$, and $S_k(G,F,K,H)<0$ for all $k=1,\dots,n-1$. Apply Theorem \ref{thm: Fock-Goncharov} again to see that $(F,G,H,K)$ is positive. This proves the first statement.

 If $G,G'\in\mathfrak{U}(F,H,K)$, then $K\in \mathfrak{U}(F,H,G)\cap \mathfrak{U}(F,H,G')$, so the fact that $\mathfrak{U}(F,H,G)$ and $\mathfrak{U}(F,H,G')$ are connected components of implies that $\mathfrak{U}(F,H,G)=\mathfrak{U}(F,H,G')$.

 Proof of (2). Since $\mathfrak{U}(F_2,H_2,K_2)\subset\mathfrak{U}(F_1,H_1,K_1)$, we have $F_2,H_2\in\overline{\mathfrak{U}(F_1,H_1,K_1)}$. At the same time, $(F_1,F_2,H_2,H_1)$ is positive, so $(F_1,F_2,H_1)$ is in general position, and (1) implies $F_2\in\mathfrak{U}(F_1,H_1,K_1)$. Then by Observation \ref{obs: basic flag}(5), $(F_1,F_2,H_2,H_1,K_1)$ is positive. 
 
 Suppose for contradiction that $\mathfrak{U}(F_2,H_2,K_1)\neq \mathfrak{U}(F_2,H_2,K_2)$. It follows from Observation \ref{obs: basic flag}(6) that for $i=1,2$, 
 \[\tau_i:=\{G^{(1)}:G\in\mathfrak{U}(F_2,H_2,K_i)\}\]
 is a simplex associated to $(F_2,H_2)$, and $\tau_1\neq\tau_2$. At the same time, since $\mathfrak U(F_2,H_2,K_i)\subset\mathfrak U(F_1,H_1,K_1)$ for $i=1,2$, so both $\tau_1$ and $\tau_2$ lie in a simplex $\Delta$ associated to $(F_1,H_1)$. This is impossible because $\tau_1\cup\tau_2$ contain an entire projective line in its closure, while $\Delta$ does not.  
 
 To prove the last claim, observe that $G_2\in \mathfrak U(F_2,H_2,K_2)=\mathfrak U(F_2,H_2,K_1)$, so since $(F_1,F_2,H_2,H_1,K_1)$ is positive, Observation \ref{obs: basic flag}(5) implies that $(F_1,F_2,G_2,H_2,H_1,K_1)$ is positive. It now follows that $\mathfrak U(F_1,H_1,G_2)\subset\mathfrak U(F_2,H_2,G_2)$ and $G_2\in\mathfrak U(F_1,H_1,K_1)$. Since $G_1\in\mathfrak U(F_1,H_1,K_1)$, (1) implies that $\mathfrak U(F_1,H_1,G_1)\subset\mathfrak U(F_2,H_2,G_2)$.

 Proof of (3). By (2), we may assume that $K_1=K_2=:K$. Let $G$ be a flag in $\overline{\mathfrak U(F_2,H_2,K)}$, and let $(G_i)_{i=1}^\infty$ be a sequence of flags in $\mathfrak U(F_2,H_2,K)$ that converges to $G$. Since $(F_1,F_2,H_2,H_1,K)$ is positive, there is a basis $\Bc:=(e_1,\dots,e_n)$ of $V$ with the property that
\begin{itemize}
\item $e_j\in K^{(j)}\cap F_1^{(n-j+1)}$ for all $j=1,\dots,n$, 
\item there exists $w_1,w_2,u_i',u_i,v_2\in U_{>0}(\Bc)$ such that $F_2=v_2\cdot F_1$, $G_i=v_2u_i\cdot F_1$, $H_2=v_2w_2\cdot F_1$, $H_1=v_2w_2w_1\cdot F_1$, and $w_2=u_iu_i'$. 
\end{itemize}
By Observation \ref{obs: increasing minors}, $\id\prec u_i,u_i'\prec w_2$, so Observation \ref{obs: prec}(1) implies that by taking subsequences, we may assume that the sequences $(u_i)_{i=1}^\infty$, $(u_i')_{i=1}^\infty$ converge to some $u_\infty,u_\infty'\in U_{\geq 0}(\Bc)$. Then by Observation \ref{obs: increasing minors}, $v_2u_\infty,u_\infty'w_1\in U_{> 0}(\Bc)$. Furthermore,
\[G=\lim_{i\to\infty}G_i=\lim_{i\to\infty}v_2u_i\cdot F_1=v_2u_\infty \cdot F_1\]
and
\[H_1=v_2u_iu_i'w_1\cdot F_1=v_2 u_\infty u_\infty'w_1 \cdot F_1,\]
so $(F_1,G,H_1,K)$ is positive. 
\end{proof}

In light of Proposition \ref{prop: easy positive}(1), we may define opposites for open sets of the form $\mathfrak U=\mathfrak{U}(F,H,K)$.

\begin{definition}
For every positive open set $\mathfrak U=\mathfrak{U}(F,H,K)$, its \emph{opposite} $\mathfrak U^{\rm opp}:=\mathfrak{U}(F,H,G)$ for some (any) $G\in\mathfrak U(F,H,K)$.
\end{definition}

Note that $(\mathfrak U^{\rm opp})^{\rm opp}=\mathfrak U$, and Proposition \ref{prop: easy positive}(2) implies that if $\mathfrak U_1=\mathfrak U(F_1,H_1,K_1)$, $\mathfrak U_2=\mathfrak U(F_2,H_2,K_2)$, and $(F_1,F_2,H_2,H_1)$ is positive, then $\mathfrak U_1\subset\mathfrak U_2$ if and only if  $\mathfrak U_2^{\rm opp}\subset\mathfrak U_1^{\rm opp}$.

By Observation \ref{obs: positive converge}(2), we see that if $(F_i)_{i=1}^\infty$ and $(H_i)_{i=1}^\infty$ are sequences of flags in $\Fc(V)$ and $K\in\Fc(V)$ such that $(F_1,\dots,F_l,H_l,\dots,H_1)$ is positive for all integers $l\geq 2$, then $\displaystyle\lim_{i\to\infty}F_i$ and $\displaystyle\lim_{i\to\infty}H_i$ exist. The next proposition tells us that to prove the conclusion of Theorem \ref{thm: general k}, it is sufficient to prove that $\displaystyle\lim_{i\to\infty}F_i=\lim_{i\to\infty}H_i$. 

\begin{proposition}\label{prop: shrink}
Let $(F_i)_{i=1}^\infty$ and $(H_i)_{i=1}^\infty$ be sequences of flags in $\Fc(V)$ and $K\in\Fc(V)$ such that $(F_1,\dots,F_l,H_l,\dots,H_1,K)$ is a positive tuple of flags for all positive integers $l$. If $\displaystyle\lim_{i\to\infty}F_i=\lim_{i\to\infty}H_i=:G$, then $\displaystyle\lim_{i\to\infty}\overline{\mathfrak U(F_i,H_i,K)}=\{G\}$.
\end{proposition}

\begin{proof}
Let $G_i\in\overline{\mathfrak U(F_i,H_i,K)}$ for all integers $i>0$. We will show that $\displaystyle\lim_{i\to\infty}G_i=G$. Let $\Bc:=(e_1,\dots,e_n)$ be the basis of $V$ with the property that
\begin{itemize}
\item $e_j\in K^{(j)}\cap F_1^{(n-j+1)}$ for all $j=1,\dots,n$, 
\item for all integers $i\geq 2$, there are unipotent elements $u_i,u_i',u_i''\in U_{>0}(\Bc)$ such that $F_i=u_i\cdot F_1$, $H_i=u_iu_i'\cdot F_1$, $H_1=u_iu_i'u_i''\cdot F_1$.
\end{itemize}
Note that $u:=u_iu_i'u_i''$ does not depend on $i$. By Proposition \ref{prop: easy positive}(3), $G_i=w_i\cdot F_1$, where $u_{i-1}\prec w_i\prec u_{i-1}u_{i-1}'$ for all integers $i>1$. Also, Proposition \ref{prop: easy positive}(3) implies that $(F_1,G,H_1,K)$ is positive, so $G$ is transverse to $K$. Hence, there is a unique unipotent element $w\in\PGL(V)$ that fixes $K$ such that $w\cdot F_1=G$. Then $\displaystyle\lim_{i\to\infty}u_i=w=\lim_{i\to\infty}u_iu_i'$, which implies that $\displaystyle\lim_{i\to\infty}w_i=w$, so $\displaystyle\lim_{i\to\infty}G_i=\lim_{i\to\infty}w_i\cdot F_1=w\cdot F_1=G$.
\end{proof}

\subsection{Positive flags and the Labourie cross ratio} \label{sec: lab}
To prove properties of $B_k$ evaluated on positive tuples of flags, it is often convenient to use the notion of a Frenet curve, which we now recall. Let $S^1$ denote the topological circle. For any map $\xi:S^1\to\Fc(V)$ and for any integer $k=1,\dots,n-1$, let $\xi^{(k)}:S^1\to\Gr_k(V)$ be the map defined by $\xi^{(k)}(x)=\xi(x)^{(k)}$.

\begin{definition}
A map $\xi:S^1\to\Fc(V)$ is \emph{Frenet} if  the following hold:
\begin{itemize}
\item Let $k>0$ be an integer. If $(x_1,\dots,x_k)$ is a pairwise distinct $k$-tuple of points in $S^1$, then the $k$-tuple of flags $\{\xi(x_1),\dots,\xi(x_k)\}$ is in general position.
\item Let $(n_1,\dots,n_k)$ be a $k$-tuple of positive integers such that $n_1+\dots+n_k=m\leq n$. If $x\in S^1$, and $((x_{i,1},\dots,x_{i,k}))_{i=1}^\infty$ is a sequence of $k$-tuples of pairwise distinct points in $S^1$ such that $\displaystyle\lim_{i\to\infty}x_{i,l}=x$ for all $l=1,\dots,k$, then
\[\lim_{i\to\infty}\xi^{(n_1)}(x_{i,1})+\dots+\xi^{(n_k)}(x_{i,k})=\xi^{(m)}(x).\]
\end{itemize}
\end{definition}

Choose one of the two natural cyclic orderings on $S^1$. The following theorem gives a relationship between Frenet curves and positive tuples of flags. 
It follows easily from the main theorems in Labourie \cite[Theorem 1.4]{Labourie}, Guichard \cite[Theorem 1]{Guichard}, and Bonahon-Dreyer \cite[Theorem 17]{BD}.


\begin{theorem}\label{thm: positive and Frenet}
Let $(F_1,\dots,F_k)$ be a positive $k$-tuple of flags. Then there is 
\begin{itemize}
\item a Frenet curve $\xi:S^1\to\Fc(V)$, and 
\item a $k$-tuple of points $x_1<\dots<x_k<x_1$ that lie in $S^1$ in this cyclic order,
\end{itemize}
such that $\xi(x_i)=F_i$ for all $i=1,\dots,k$.
\end{theorem}

More informally, one can think of Frenet curves as an extension of positive tuples to a map of the circle into $\Fc(V)$ with ``strong continuity properties". Our next goal is to prove, as applications of Theorem \ref{thm: positive and Frenet}, the pair of inequalities stated below as Proposition \ref{prop: positive flags and positive cross ratios} and Proposition \ref{prop: smaller tent}. 

\begin{proposition}\label{prop: positive flags and positive cross ratios}
If $(F_1,F_2,F_3,F_4)$ is a positive quadruple of flags in $\Fc(V)$, then 
\[B_k(F_1,F_2,F_3,F_4)>1\] 
for any integer $k=1,\dots, n-1$. In particular, if $(F_1,F_2,F_3,F_4,F_5)$ is a positive quintuple of flags in $\Fc(V)$, then
\[B_k(F_1,F_2,F_3,F_4)<B_k(F_1,F_2,F_3,F_5)\,\text{ and }\,B_k(F_1,F_2,F_4,F_5)<B_k(F_1,F_2,F_3,F_5)\]
for all $k=1,\dots,n-1$.
\end{proposition}

To prove Proposition \ref{prop: positive flags and positive cross ratios}, it is useful to consider more general classes of projective invariants, which we now define.

\begin{definition}
Choose a linear identification $\displaystyle\Omega:\bigwedge^nV\to\R$. For any quadruple of flags $(F_1,F_2,F_3,F_4)$ in $\Fc(V)$, choose bases $(u_{i,1},\dots,u_{i,n})$ of $V$ such that $F_i^{(k)}=\Span_{\mathbb{R}}(u_{i,1},\dots,u_{i,k})$ for all $k=1$, $\dots$, $n$ and $i=1,\dots,4$. Then let $U_{i,k}$ be the $k$-tuple of vectors $(u_{i,1},\dots,u_{i,k})$.
\begin{enumerate}
\item Suppose that $\{F_1,F_2,F_3\}$ and $\{F_1,F_2,F_4\}$ are in general position. For any triple of non-negative integers $\mathbf k:=(k_1,k_2,k_3)$ such that $k_3>0$ and $k_1+k_2+k_3=n-1$, define
\[D_{\mathbf k}(F_1,F_2,F_3,F_4):=\frac{\Omega(U_{1,k_1+1},U_{2,k_2},U_{3,k_3})\Omega(U_{1,k_1},U_{2,k_2+1},U_{4,k_3})}{\Omega(U_{1,k_1+1},U_{2,k_2},U_{4,k_3})\Omega(U_{1,k_1},U_{2,k_2+1},U_{3,k_3})},\]
\item Suppose that $\{F_1,F_2,F_3,F_4\}$ are in general position. For any quadruple of non-negative integers $\mathbf j:=(j_1,j_2,j_3,j_4)$ such that $j_1+j_2+j_3+j_4=n-2$, define
\[A_{\mathbf j}(F_1,F_2,F_3,F_4):=\frac{\Omega(U_{1,j_1+1},U_{2,j_2},U_{3,j_3+1},U_{4,j_4})\Omega(U_{1,j_1},U_{2,j_2+1},U_{3,j_3},U_{4,j_4+1})}{\Omega(U_{1,j_1+1},U_{2,j_2},U_{3,j_3},U_{4,j_4+1})\Omega(U_{1,j_1},U_{2,j_2+1},U_{3,j_3+1},U_{4,j_4})}.\]
\end{enumerate}
\end{definition}

One can verify that the quantities $D_{\mathbf k}(F_1,F_2,F_3,F_4)$ and $A_{\mathbf j}(F_1,F_2,F_3,F_4)$ are well-defined and do not depend on any of the choices made.

The functions $A_{\mathbf j}$ were studied by the third author, who proved the following.

\begin{proposition}\cite[Proposition 2.12(1)]{Zhang}\label{prop: Zhang thesis}
Let $\xi:S^1\to\Fc(V)$ be a Frenet curve, and let $x_1<x_2<x_3<x_4<x_1$ lie in $S^1$ in this cyclic order. Then 
\[A_{\mathbf j}(\xi(x_1),\xi(x_2),\xi(x_3),\xi(x_4))>1\]
for any quadruple of non-negative integers $\mathbf j:=(j_1,j_2,j_3,j_4)$ that sum to $n-2$.
\end{proposition} 

The following lemma was previously observed in Martone-Zhang \cite[Lemma 3.6]{MZ}; its proof is a straightforward computation that we omit.

\begin{lemma}\label{lem: Martone-Zhang}
Let $(F_1,F_2,F_3,F_4)$ be a generic quadruple of flags in $\Fc(V)$. 
\begin{enumerate}
\item For all $k=1,\dots,n-1$, 
\[B_k(F_1,F_2,F_3,F_4)=\prod_{\mathbf k\in\mathcal{A}_k}D_{\mathbf k}(F_1,F_2,F_3,F_4),\]
where $\mathcal{A}_k:=\{(k_1,k_2,k_3):k_3=k,\,k_1\geq 0,\,k_2\geq 0,\,\text{and }\,k_1+k_2+k_3=n-1\}$.
\item For all $\mathbf k:=(k_1,k_2,k_3)$ such that $k_3>0$ and $k_1+k_2+k_3=n-1$, 
\[D_{\mathbf k}(F_1,F_2,F_3,F_4)=\prod_{\mathbf j\in\mathcal{B}_{\mathbf k}}A_{\mathbf j}(F_1,F_2,F_3,F_4),\]
where $\mathcal{B}_{\mathbf k}:=\{(j_1,j_2,j_3,j_4):j_1=k_1,\,j_2=k_2,\,j_3\geq 0,\,j_4\geq 0,\,\text{and }\,j_3+j_4=k_3-1\}$.
\end{enumerate}
In particular, for all $k=1,\dots,n-1$, we have
\[B_k(F_1,F_2,F_3,F_4)=\prod_{\mathbf j\in\mathcal{C}_k}A_{\mathbf j}(F_1,F_2,F_3,F_4),\]
where $\mathcal{C}_k:=\{(j_1,j_2,j_3,j_4):j_1+j_2=n-k-1\,\text{ and }\,j_3+j_4=k-1\}$.
\end{lemma}


\begin{proof}[Proof of Proposition \ref{prop: positive flags and positive cross ratios}]
The first statement is an immediate consequence of Theorem \ref{thm: positive and Frenet}, Proposition \ref{prop: Zhang thesis}, and Lemma \ref{lem: Martone-Zhang}. The second follows from the first and the identity
\[B_k(F_1,F_2,F_3,F_4)\cdot B_k(F_1,F_2,F_4,F_5)=B_k(F_1,F_2,F_3,F_5).\qedhere\]
\end{proof}

\begin{proposition}\label{prop: smaller tent}
Suppose that $(F_1,F_2,G,H_2,H_1)$ is a positive quintuple of flags in $\Fc(V)$. Fix $h=1,\dots,n-1$, and let $G_F$ and $G_H$ be the flags in $\Fc(V)$ defined by
\[G_F^{(l)}=\left\{\begin{array}{ll}
G^{(l)}&\text{if }l\leq h;\\
G^{(h)}+F_2^{(l-h)}&\text{if }l>h,
\end{array}\right.\,\,\,\,\,\text{ and }\,\,\,\,\,G_H^{(l)}=\left\{\begin{array}{ll}
G^{(l)}&\text{if }l\leq h;\\
G^{(h)}+H_2^{(l-h)}&\text{if }l>h.
\end{array}\right.\]
Then $B_k(H_1,F_1,G_F,G_H)\leq B_k(H_1,F_1,F_2,H_2)$ for all $k=1,\dots,n-1$.
\end{proposition}

To prove Proposition \ref{prop: smaller tent}, we use the following lemma.

\begin{lemma}\label{lem: Frenet trick}
Let $(F_1,F_2,F_3,F_4,F_5)$ be a positive quintuple of flags in $\Fc(V)$, and fix $h=1,\dots,n-1$. If $G\in\Fc(V)$ is the flag defined by
\[G^{(l)}=\left\{\begin{array}{ll}
F_3^{(l)}&\text{if }l\leq h;\\
F_3^{(h)}+F_2^{(l-h)}&\text{if }l>h,
\end{array}\right.\]
then $(F_1,G,F_4,F_5)$ is a positive quadruple of flags.
\end{lemma}

\begin{proof}
By Theorem \ref{thm: positive and Frenet}, there are points $x_1<x_2<x_3<x_4<x_5<x_1$ along $S^1$ in this cyclic order, and a Frenet curve $\xi:S^1\to\Fc(V)$, such that $\xi(x_i)=F_i$ for $i=1,\dots,5$. For any $t\in S^1$ such that $x_2\leq t<x_3$, let $G(t)\in\Fc(V)$ be the flag defined by
\[G(t)^{(l)}=\left\{\begin{array}{ll}
\xi^{(l)}(x_3)&\text{if }l\leq h;\\
\xi^{(h)}(x_3)+\xi^{(l-h)}(t)&\text{if }l>h,
\end{array}\right.\]
and let $G(x_3):=\xi(x_3)$. The Frenet property of $\xi$ implies that $t\mapsto G(t)$ is continuous, and that $\{\xi(x_1),G(t),\xi(x_4),\xi(x_5)\}$ is in general position for all $x_2\leq t\leq x_3$. Since $(F_1,F_3,F_4,F_5)=(\xi(x_1),G(x_3),\xi(x_4),\xi(x_5))$ is positive, Proposition \ref{prop: easy positive}(1) implies that $(F_1,G,F_4,F_5)=(\xi(x_1),G(x_2),\xi(x_4),\xi(x_5))$ is positive.
\end{proof}

\begin{proof}[Proof of Proposition \ref{prop: smaller tent}]
By Theorem \ref{thm: positive and Frenet}, there are points $x_1<x_2<z<y_2<y_1<x_1$ along $S^1$ in this cyclic order, and a Frenet curve $\xi:S^1\to\Fc(V)$, such that $\xi(z)=G$, and $\xi(x_j)=F_j$ and $\xi(y_j)=H_j$ for $j=1,2$. Let $(a_i)_{i=1}^\infty$, $(b_i)_{i=1}^\infty$, $(c_i)_{i=1}^\infty$, and $(d_i)_{i=1}^\infty$ be sequences of points in $S^1$ such that $\displaystyle\lim_{i\to\infty}a_i=x_2$, $\displaystyle\lim_{i\to\infty}b_i=z=\lim_{i\to\infty}c_i$, $\displaystyle\lim_{i\to\infty}d_i=y_2$, and 
\[x_1<a_1<\dots<a_i<x_2<b_1<\dots<b_i<z<c_i<\dots<c_1<y_2<d_i<\dots<d_1<y_1<x_1\]
for all integers $i>0$. Then let $A_i$ and $D_i$ be the flags defined by
\[A_i^{(l)}=\left\{\begin{array}{ll}
\xi^{(l)}(b_i)&\text{if }l\leq h;\\
\xi^{(h)}(b_i)+\xi^{(l-h)}(a_i)&\text{if }l>h,
\end{array}\right.\,\,\,\,\,\text{ and }\,\,\,\,\,D_i^{(l)}=\left\{\begin{array}{ll}
\xi^{(l)}(c_i)&\text{if }l\leq h;\\
\xi^{(h)}(c_i)+\xi^{(l-h)}(d_i)&\text{if }l>h.
\end{array}\right.\]

By Lemma \ref{lem: Frenet trick}, the tuple $(\xi(x_1),\xi(a_{i-1}),A_i,G,D_i,\xi(d_{i-1}),\xi(y_1))$ is positive for all integers $i\geq 2$. Then Proposition \ref{prop: positive flags and positive cross ratios} implies that 
\[B_k(\xi(y_1),\xi(x_1),A_i,D_i)< B_k(\xi(y_1),\xi(x_1),\xi(a_{i-1}),\xi(d_{i-1}))\]
for all $k=1,\dots,n-1$. Since $\xi$ is Frenet, $(A_i)_{i=1}^\infty$ and $(D_i)_{i=1}^\infty$ converge to $G_F$ and $G_H$ respectively, and $(\xi(a_i))_{i=1}^\infty$ and $(\xi(a_i))_{i=1}^\infty$ converge to $F_2$ and $H_2$ respectively. Thus,
\begin{align*}
B_k(H_1,F_1,G_F,G_H)&=\lim_{i\to\infty}B_k(\xi(y_1),\xi(x_1),A_i,D_i)\\
&\leq\lim_{i\to\infty}B_k(\xi(y_1),\xi(x_1),\xi(a_{i-1}),\xi(d_{i-1}))\\
&= B_k(H_1,F_1,F_2,H_2).\qedhere
\end{align*}
\end{proof}

\subsection{Quotients of positive tuples of flags} \label{sec: projection}

Another ingredient needed in the proof of Theorem \ref{thm: general k} is understanding how positivity behaves under taking certain quotients. The two results in this direction that we need are stated as Proposition \ref{prop: positivity projection} and Proposition \ref{prop: converge} below.

Let $W\subset V$ be a $k$-dimensional subspace for any $k=1,\dots,n-1$, and let $\pi_W:V\to V/W$ be the obvious quotient map. We abuse notation by also denoting by $\pi_W$ the induced map 
\[\pi_W:\Fc(V)\to\Fc(V/W)\]
that sends the flag $F$ in $\Fc(V)$ to the flag $F'$ in $\Fc(V/W)$ defined as follows. For any $j=1,\dots,n-k-1$, let $l_j$ be the integer such that $F^{(l_j)}\cap W$ has dimension $l_j-j$. In other words, if we write $F^{(l_j)}$ as the direct sum 
\[F^{(l_j)}=\left(F^{(l_j)}\cap W\right)+U\] 
for some $U\subset F^{(l_j)}$, then $\dim(U)=j$. Then $F'$ is the flag defined by $F'^{(j)}=\pi(F^{(l_j)})$ for all $j=1,\dots,n-k-1$. 

Observe that if $F$ has the property that $F^{(j)}\cap W=\{0\}$ for all $j=1,\dots,n-k-1$, then $\pi(F)$ is the flag defined by $\pi(F)^{(j)}=\pi(F^{(j)})=(F^{(j)}+W)/W$ for all $j=1,\dots,n-k-1$. On the other hand, if $W=F^{(k)}$, then $\pi(F)$ is the flag defined by $\pi(F)^{(j)}=\pi(F^{(j+k)})=F^{(j+k)}/W$ for all $j=1,\dots,n-k-1$.

\begin{proposition}\label{prop: positivity projection}
Let $(F_1,F_2,\dots,F_l,G,H_l,\dots,H_2,H_1)$ be a positive tuple of flags in $\Fc(V)$, let $k=1,\dots,n-1$, and let $W:=G^{(k)}$. If $\pi_W=\pi:V\to V/W$ denotes the quotient map, then 
\[(\pi(F_1),\pi(F_2),\dots,\pi(F_l),\pi(G),\pi(H_l),\dots,\pi(H_2),\pi(H_1))\] 
is a positive tuple of flags in $\Fc(V/W)$ for all $l\geq 2$.
\end{proposition}

To prove Proposition \ref{prop: positivity projection}, we use the following lemma.

\begin{lemma}\label{lem: total positivity projection}
Let $\Bc:=(e_1,\dots,e_n)$ be a basis of $V$. Fix $k=1,\dots,n-1$, and set $\Dc:=(e_1,\dots,e_k)$ and $\Cc:=(e_{k+1},\dots,e_n)$. Then let $W':=\Span_{\mathbb{R}}(\Cc)$, $W:=\Span_{\mathbb{R}}(\Dc)$, and let $P=P_W:V\to W'$ be the projection with kernel $W$. If $u\in U_{>0}(\Bc)$, then $u':=P\circ u\in\PGL(W')$ satisfies $u'\in U_{>0}(\Cc)$.
\end{lemma}

\begin{proof}
Let $M_u$ be the matrix representing $u$ in the basis $\Bc$, then $M_u$ is an $n\times n$, totally positive, unipotent, upper triangular matrix. If $M_{u'}$ is the matrix representing $u'$ in the basis $\Cc$, then $M_{u'}$ is the submatrix of $M_u$ corresponding to the last $k+1$ rows and the last $k+1$ columns. It follows that $M_{u'}$ is also a totally positive, unipotent upper triangular matrix.
\end{proof}

\begin{proof}[Proof of Proposition \ref{prop: positivity projection}]
Let $K$ be a flag in $\Fc(V)$ such that 
\[(K,F_1,F_2,\dots,F_l,G,H_l,\dots,H_2,H_1)\] 
is a positive tuple of flags. Then there is some basis $\Bc:=(e_1,\dots,e_n)$ of $V$ such that $e_i\in G^{(i)}\cap K^{(n-i+1)}$ for all $i=1,\dots,n$, and some $u_1,u_2,\dots,u_l\in U_{>0}(\Bc)$ such that $F_i=u_1\dots u_i\cdot K$ for all $i=1,\dots,l$. Similarly, there is some basis $\Bc':=(e_1',\dots,e_n')$ of $V$ such that $e_i'\in G^{(i)}\cap K^{(n-i+1)}$ for all $i=1,\dots,n$, and some $v_1,v_2,\dots,v_l\in U_{>0}(\Bc')$ such that $H_i=v_1\dots v_i\cdot K$ for all $i=1,\dots,l$. 

Let $\Cc:=(e_{k+1},\dots,e_n)$ and $\Cc':=(e_{k+1}',\dots,e_n')$, and let $W':=\Span_\mathbb{R}(\Cc)=\Span_\mathbb{R}(\Cc')$. Since $V=W+W'$, we may identify $V/W$ with $W'$. Via this identification, the quotient map $\pi:V\to V/W$ can viewed as a projection $V\to W'$ whose kernel is $W$. Then 
\[\pi(G)^{(j)}=\Span_\mathbb{R}(e_{k+1},\dots,e_{k+j})\,\text{ and }\,\pi(K)^{(j)}=\Span_\mathbb{R}(e_{n-j+1},\dots,e_n)\] 
for all $j=1,\dots,n-k-1$. For all $i=1,\dots,l$, let $u_i',v_i'\in\PGL(V/W)$ be defined by $u_i'=\pi\circ u_i$ and $v_i'=\pi\circ v_i$ respectively. By Lemma \ref{lem: total positivity projection}, $u_i'\in U_{>0}(\Cc)$ and $v_i'\in U_{>0}(\Cc')$. Furthermore, as elements in $\PGL(W')$, $\pi\circ u_1\circ\dots\circ u_i=u_1'\circ\dots\circ u_i'$ for all $i=1,\dots,l$. This implies that for all $j=1,\dots,n-k$, we have
\[\pi(F_i)^{(j)}=\pi(F_i^{(j)}+W)=\pi(u_1\dots u_i\cdot(K^{(j)}+W))=u_1'\dots u_i'\cdot \pi(K)^{(j)},\] 
so $\pi(F_i)=u_1'\dots u_i'\cdot \pi(K)$. As such, $(\pi(K),\pi(F_1),\dots,\pi(F_l),\pi(G))$ is a positive tuple of flags. Similarly, the tuple $(\pi(K),\pi(H_1),\dots,\pi(H_l),\pi(G))$ is also positive.

To prove that $(\pi(K),\pi(F_1),\dots,\pi(F_l),\pi(G),\pi(H_l),\dots,\pi(H_1))$ is positive, it is sufficient to prove that $(\pi(K),\pi(F_1),\pi(G),\pi(H_1))$ is positive and use Observation~\ref{obs: basic flag}(5). Since $(\pi(K),\pi(F_1),\pi(G))$ and $(\pi(K),\pi(H_1),\pi(G))$ are positive, by Theorem~\ref{thm: Fock-Goncharov}, it suffices to show that for all $j=1,\dots,n-k-1$, $S_j(\pi(K),\pi(F_1),\pi(G),\pi(H_1))<0$. By Observation \ref{obs: cross ratio projection},
\begin{eqnarray*}
&&S_{j+k}(G,F_1,K,H_1)\\
&=&C_1\left(G^{(j+k-1)}+K^{(n-j-k)},G^{(j+k)}+K^{(n-j-k-1)},F_1^{(1)},H_1^{(1)}\right)\\
&=&C_1\left(\pi(G)^{(j-1)}+\pi(K)^{(n-j-k)},\pi(G)^{(j)}+\pi(K)^{(n-j-k-1)},\pi(F_1)^{(1)},\pi(H_1)^{(1)}\right)\\
&=&S_j(\pi(G),\pi(F_1),\pi(K),\pi(H_1))
\end{eqnarray*}
Since $(G,F_1,K,H_1)$ is positive, $S_{j+k}(G,F_1,K,H_1)<0$.
\end{proof}

\begin{proposition}\label{prop: converge}
Let $(F_i)_{i=1}^\infty$ be a sequence of flags in $\Fc(V)$ such that $(F_1,\dots,F_i,F_\infty)$ is a positive tuple of flags for all integers $i\geq 2$. Fix $k=1,\dots,n-1$, set $W:=F_\infty^{(k)}$, and let $\pi=\pi_W:V\to V/W$ to be the quotient map. If $(F_i)_{i=1}^\infty$ converges to $F_\infty$, then $(\pi(F_i))_{i=1}^\infty$ converges to $\pi(F_\infty)$.
\end{proposition}

\begin{proof}
For all integers $i>0$, let $F_{i,\infty}$ be the flag in $\Fc(V)$ defined by 
\[F_{i,\infty}^{(l)}=\left\{\begin{array}{ll}
F_\infty^{(l)}&\text{if }l\leq k;\\
F_\infty^{(k)}+F_i^{(l-k)}&\text{if }l> k.\\
\end{array}\right.\]
Observe that $\pi(F_i)=\pi(F_{i,\infty})$. Also, let $K\in\Fc(V)$ and let $(H_i)_{i=1}^\infty$ be a sequence in $\Fc(V)$ such that $(F_1,\dots,F_i,F_\infty,H_i,\dots,H_1,K)$ is positive, and $\displaystyle\lim_{i\to\infty}H_i=F_\infty$. 

By Proposition \ref{prop: shrink}, 
\[\displaystyle \{F_\infty\}=\lim_{i\to\infty}\overline{\mathfrak{U}(F_i,H_i,K)}=\bigcap_{i=1}^\infty \overline{\mathfrak{U}(F_i,H_i,K)}.\] 
Lemma \ref{lem: Frenet trick} implies that $F_{i,\infty}\in\mathfrak U(F_{i-1},H_{i-1},K)$ for all integers $i\geq 2$, so 
\[\displaystyle F_{i,\infty}\in\bigcap_{j=1}^{i-1} \overline{\mathfrak U(F_{j},H_{j},K)}.\] 
Thus, $\displaystyle\lim_{i\to\infty}F_{i,\infty}=F_\infty$, which implies that 
\[\displaystyle\lim_{i\to\infty}\pi(F_i)=\lim_{i\to\infty}\pi(F_{i,\infty})=\pi(F_\infty).\qedhere\]
\end{proof}

\subsection{Proof of Theorem \ref{thm: general k}}\label{sec: collapse}

Recall that by Observation \ref{obs: positive converge} part (2) that if $(F_i)_{i=1}^\infty$ and $(H_i)_{i=1}^\infty$ are sequences of flags in $\Fc(V)$ such that $(F_1,\dots,F_l,H_l,\dots,H_1)$ is positive for all integers $l\geq 2$, then $\displaystyle\lim_{i\to\infty}F_i$ and $\displaystyle\lim_{i\to\infty}H_i$ exist. First, we prove a weaker version of Theorem \ref{thm: general k}.

\begin{proposition}\label{prop: k=1'}
Let $(F_i)_{i=1}^\infty$ and $(H_i)_{i=1}^\infty$ be sequences of flags in $\Fc(V)$ such that for all integers $l\geq 2$, $(F_1,\dots,F_l,H_l,\dots,H_1)$ is a positive tuple of flags. Let $\displaystyle F_\infty:=\lim_{i\to\infty}F_i$ and $\displaystyle H_\infty:=\lim_{i\to\infty}H_i$. If there is some $D>1$ such that $B_1(H_i,F_i,F_{i+1},H_{i+1})\leq D$ for all integers $i>0$, then $F_\infty^{(1)}=H_\infty^{(1)}$. 
\end{proposition}

The proof of Proposition \ref{prop: k=1'} requires the following lemma.

\begin{lemma}\label{lem: arrgghh'}
Let $(F_1,F_2,H_2,H_1)$ be a positive quadruple of flags in $\Fc(V)$. 
\begin{enumerate}
\item  Let $K\in\Fc(V)$ be a flag such that $(F_1,F_2,H_2,H_1,K)$ is positive. For $i,j=1,2$,
\[\tau_{i,j}:=\{G^{(1)}\in\mathbb{P}(V):G\in\mathfrak U(F_i,H_j,K)\}\]
is a simplex associated to $\{F_i,H_j\}$.
\item For $i=1,2$, let $\tau_{F_i}$ and $\tau_{H_i}$ be the closed faces of the simplex $\tau_{i,i}$ that lie in $F_i^{(n-1)}$ and $H_i^{(n-1)}$ respectively. If $P\in\Gr_2(V)$ denotes the subspace containing $F_2^{(1)}$ and $H_2^{(1)}$, then $P\cap F_1^{(n-1)}\in\tau_{F_1}$ and $P\cap H_1^{(n-1)}\in\tau_{H_1}$.
\item $C_1\left(H_1^{(n-1)},F_1^{(n-1)},x_1,x_2\right)\leq B_1\left(H_1,F_1,F_2,H_2\right)$ for all $x_1$ and $x_2$ in $\overline{\tau_{2,2}}$.
\end{enumerate}
\end{lemma}

\begin{proof}
Proof of (1). This is immediate from Observation \ref{obs: basic flag}(6). 

Proof of (2). Since positive tuples are in general position, $P\cap F_1^{(n-1)}$ and $P\cap F_2^{(n-1)}$ are points in $\Pb(V)$. By Observation \ref{obs: basic flag}(5) and part (1), $\tau_{2,2}\subset \tau_{1,2}\subset\tau_{1,1}$. Since $H_2^{(1)}$ is a vertex of $\tau_{1,2}$ and $F_2^{(1)}$ lies in $\tau_{1,2}$, this implies that $P\cap F_1^{(n-1)}$ lies in $\tau_{F_1}$, see Figure~\ref{fig: face}. The same argument, switching the roles of $F_1$ and $H_1$ with $F_2$ and $H_2$ respectively, proves that $P\cap H_1^{(n-1)}$ lies in $\tau_{H_1}$. 

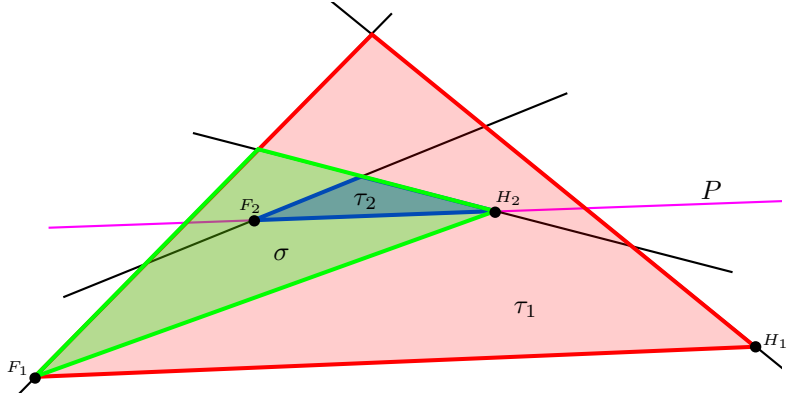
\begin{figure}[h]
    \centering
    \includegraphics[width=0.7\textwidth]{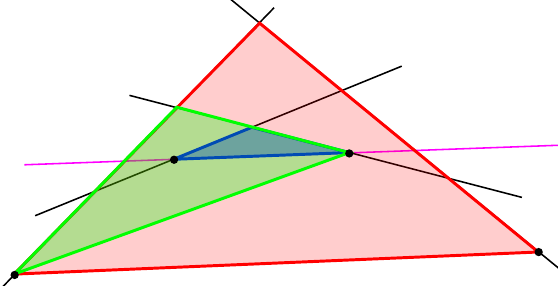}
    \small
 \put (-30,73){$P$}   
  \put (-100,30){$\tau_1$}   
  \put (-190,50){$\sigma$}   
  \put (-160,71){$\tau_2$}   
\tiny
\put (-290,8){$F_1$}
\put (-7,18){$H_1$}
\put (-203,69){$F_2$}
\put (-107,72){$H_2$}
    \caption{The simplices $\tau_2\subset\sigma\subset\tau_1$.}
    \label{fig: face}
\end{figure}

Proof of (3). By Proposition \ref{prop: easy positive}(3), $\overline{\mathfrak U(F_2,H_2,K)}\subset\mathfrak U(F_1,H_1,K)$. Thus, by part (1), for $i=1,2$, there is a flag $G_i\in\overline{\mathfrak U(F_2,H_2,K)}\subset\mathfrak U(F_1,H_1,K)$ such that $G_i^{(1)}=x_i$. By Proposition \ref{prop: positive flags and positive cross ratios}, $B_1(H_1,F_1,F_2,G_i),B_1(H_1,F_1,G_i,H_2)\ge1$. Also, since $G_1^{(1)}$ and $G_2^{(1)}$ lie in the same connected component of $\mathbb{P}(V)\setminus(F_1^{(n-1)}\cup H_1^{(n-1)})$, one verifies that $B_1(H_1,F_1,G_1,G_2)>0$. Thus,
\begin{align*}
B_1\left(H_1,F_1,F_2,H_2\right)&=B_1(H_1,F_1,F_2,G_1)\cdot B_1(H_1,F_1,G_1,G_2)\cdot B_1(H_1,F_1,G_2,H_2)\\
&\ge B_1(H_1,F_1,G_1,G_2)=C_1\left(H_1^{(n-1)},F_1^{(n-1)},x_1,x_2\right).\qedhere
\end{align*}
\end{proof}

\begin{proof}[Proof of Proposition \ref{prop: k=1'}]
For any integer $j>0$, let $P_j:=H_j^{(1)}+F_j^{(1)}\in\Gr_2(V)$. Since the quadruple $(F_i,F_j,H_j,H_i)$ is positive for all integers $i,j>0$ such that $i< j$, we see that $P_j$ does not lie in $F_i^{(n-1)}$ or $H_i^{(n-1)}$. Thus, we may define the points $p_{i,j}:=F_i^{(n-1)}\cap P_j$ and $q_{i,j}:=H_i^{(n-1)}\cap P_j$ in $\mathbb{P}(V)$. Let $K\in\Fc(V)$ be a flag such that $(F_1,F_2,H_2,H_1,K)$ is positive, and let
\[\tau_i:=\{G^{(1)}\in\mathbb{P}(V):G\in\mathfrak U(F_i,H_i,K)\}.\] 
By Lemma \ref{lem: arrgghh'}(1), $\tau_i$ is a simplex associated to $\{F_i,H_i\}$. Let $\tau_{F_i}$ and $\tau_{H_i}$ be the closed faces of $\tau_i$ that lie in $F_i^{(n-1)}$ and $H_i^{(n-1)}$ respectively. By Proposition \ref{prop: easy positive}(3), $\overline{\tau_{i+1}}\subset\tau_i$ for all integers $i>0$. Since Lemma \ref{lem: arrgghh'}(2) implies that $p_{i,j}\in \tau_{F_i}$ and $q_{i,j}\in\tau_{H_i}$ for all integers $0<i<j$, it follows that 
\begin{align}
\label{eqn: order of points}
p_{1,j}<p_{2,j}<\dots<p_{j,j}<q_{j,j}<\dots<q_{2,j}<q_{1,j}<p_{1,j},
\end{align}
see Figure \ref{fig: nested}.

\begin{figure}[h]
    \centering
    \includegraphics[width=0.9\textwidth]{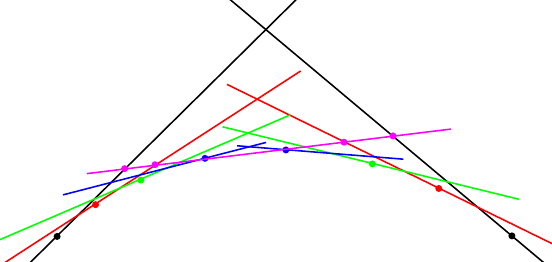}
\tiny
\put (-330,10){$F_1$}
\put (-33,10){$H_1$}
\put (-303,33){$F_2$}
\put (-87,45){$H_2$}
\put (-235,62){$F_j$}
\put (-180,67){$H_j$}
\put (-300,66){$p_{1,j}$}
\put (-273,70){$p_{2,j}$}
\put (-108,89){$q_{1,j}$}
\put (-141,85){$q_{2,j}$}
    \caption{Triangulation $\mathcal T$ in Theorem \ref{thm: Fock-Goncharov}.}
    \label{fig: nested}
\end{figure}
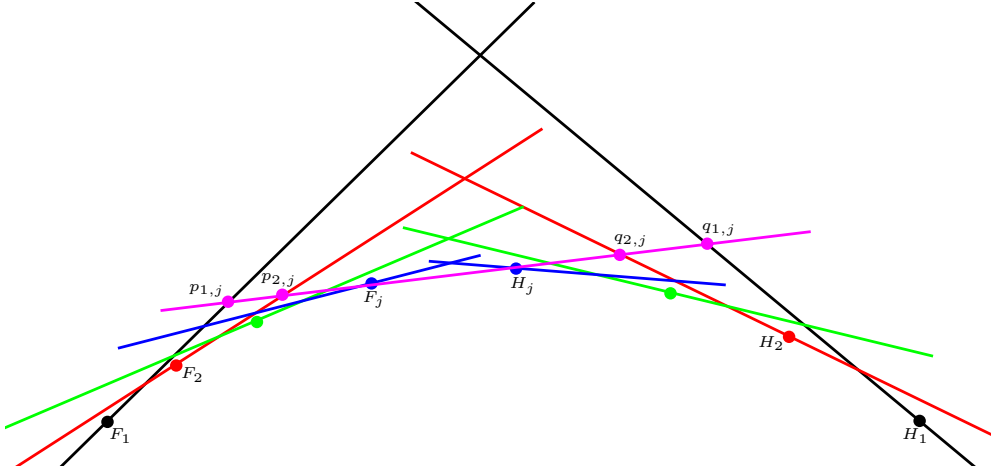

First, we prove that for all integers $i,j>0$ satisfying $i<j$, we have
\begin{equation}\label{eqn: 1D'}1<C_1(q_{i,j},p_{i,j},p_{i+1,j},q_{i+1,j})\leq D.\end{equation}
Here, $C_1$ is the cross ratio on $\mathfrak{Q}_1(P_j)$. It is straightforward to verify that 
\[C_1(q_{i,j},p_{i,j},p_{i+1,j},q_{i+1,j})=C_1\left(H_i^{(n-1)},F_i^{(n-1)},p_{i+1,j},q_{i+1,j}\right),\]
where $C_1$ on the right is a cross ratio on $\mathfrak{Q}_1(V)$. Also, Lemma \ref{lem: arrgghh'}(2) implies that $p_{i+1,j}\in\tau_{F_{i+1}}$ and $q_{i+1,j}\in\tau_{H_{i+1}}$ for all $j>i$, so we may apply Lemma \ref{lem: arrgghh'}(3) to deduce that
\[C_1\left(H_i^{(n-1)},F_i^{(n-1)},p_{i+1,j},q_{i+1,j}\right)\leq B_1\left(H_i,F_i,F_{i+1},H_{i+1}\right).\]
Since $B_1\left(H_i,F_i,F_{i+1},H_{i+1}\right)\leq D$ by hypothesis, this proves the required upper bound. By (\ref{eqn: cross ratio classical}), to prove the required lower bound, it is sufficient to show that $q_{i,j}<p_{i,j}<p_{i+1,j}<q_{i+1,j}<q_{i,j}$ lies in $P_j$ in this cyclic order. This follows from (\ref{eqn: order of points}).

Next, choose an affine chart $\mathbb{A}$ of $\mathbb{P}(V)$ that contains $\overline{\tau_1}$, and equip $\mathbb{A}$ with an Euclidean metric $d$, i.e. $d$ is invariant under translations in $\mathbb{A}$. Let $\mathbb{A}_i:=P_i\cap\mathbb{A}$ be the induced affine chart on the projective line $P_i$, and choose an affine isometry $\mathbb{A}_i\simeq\mathbb{R}$. By Lemma \ref{lem: arrgghh'}(2), $p_{1,i}$ and $q_{1,i}$ lie in $\overline{\tau_1}$ for all integers $i>0$. Since $\overline{\tau_1}$ is compact, there is a constant $A>0$ such that $d(p_{1,i},q_{1,i})\leq A$ for all integers $i>0$. Then (\ref{eqn: order of points}), (\ref{eqn: 1D'}) and Observation \ref{obs: RP1} together give
\[d(p_{i,i},q_{i,i})\leq\left(\frac{\sqrt{D}}{1+\sqrt{D}}\right)d(p_{i-1,i},q_{i-1,i})\leq\dots\leq\left(\frac{\sqrt{D}}{1+\sqrt{D}}\right)^{i-1}d(p_{1,i},q_{1,i})\leq\left(\frac{\sqrt{D}}{1+\sqrt{D}}\right)^{i-1}A.\]
Since $p_{i,i}=F_i^{(1)}$ and $q_{i,i}=H_i^{(1)}$, the sequences $(p_{i,i})_{i=1}^\infty$ and $(q_{i,i})_{i=1}^\infty$ converge to $F_\infty^{(1)}$ and $H_{\infty}^{(1)}$ respectively. Thus,
\[d(F_{\infty}^{(1)},H_{\infty}^{(1)})=\lim_{i\to\infty}d(p_{i,i},q_{i,i})\leq\lim_{i\to\infty}\left(\frac{\sqrt{D}}{1+\sqrt{D}}\right)^{i-1}A=0,\]
which means that $F_\infty^{(1)}=H_\infty^{(1)}$.
\end{proof}

Using Proposition \ref{prop: k=1'}, we now prove Theorem \ref{thm: general k}.

\begin{proof}[Proof of Theorem \ref{thm: general k}]
Suppose for contradiction that $F_\infty\neq H_\infty$. Let $k$ be the smallest positive integer such that $F_\infty^{(k)}\neq H_\infty^{(k)}$, and let $W:=F_\infty^{(k-1)}=H_\infty^{(k-1)}=F_\infty^{(k)}\cap H_\infty^{(k)}$ ($W=\{0\}$ if $k=1$). Let $\pi:V\to V/W$ be the quotient map. As before, we abuse notation and denote by $\pi:\Fc(V)\to\Fc(V/W)$ the induced map defined in Section \ref{sec: projection}. 

For all integers $i>0$, let $F_{i,\infty}$ and $H_{i,\infty}$ be the flags in $\Fc(V)$ defined by
\[F_{i,\infty}^{(j)}=\left\{\begin{array}{ll}
F_\infty^{(j)}&\text{if }j\leq k-1;\\
W+F_i^{(j-k)}&\text{if }j>k-1,
\end{array}\right.\,\,\,\,\,\text{ and }\,\,\,\,\,H_{i,\infty}^{(j)}=\left\{\begin{array}{ll}
H_\infty^{(j)}&\text{if }j\leq k-1;\\
W+H_i^{(j-k)}&\text{if }j>k-1.
\end{array}\right.\]
By Proposition \ref{prop: smaller tent}, 
\begin{equation}\label{eqn: ineq 1'}
B_k(H_i,F_i,F_{i+1,\infty},H_{i+1,\infty})\leq B_k(H_i,F_i,F_{i+1},H_{i+1})
\end{equation}
for all $k=1,\dots,n-1$. Also, since $F_{i,\infty}^{(k)}=W+F_i^{(1)}$ and $H_{i,\infty}^{(k)}=W+H_i^{(1)}$, observe that $\pi(F_{i,\infty}^{(k)})=\pi(F_i^{(1)})$ and $\pi(H_{i,\infty}^{(k)})=\pi(H_i^{(1)})$. This implies that
\begin{align}\label{eqn: ineq 2'}
B_1(\pi(H_i),\pi(F_i),\pi(F_{i+1}),\pi(H_{i+1}))&= B_1(\pi(H_i),\pi(F_i),\pi(F_{i+1,\infty}),\pi(H_{i+1,\infty}))\nonumber\\
&= B_k(H_i,F_i,F_{i+1,\infty},H_{i+1,\infty}),
\end{align}
where the second inequality follows from Observation \ref{obs: cross ratio projection}. Together, (\ref{eqn: ineq 1'}) and (\ref{eqn: ineq 2'}) imply that for all integers $i>0$,
\begin{equation}\label{eqn: ineq 5'}
B_1(\pi(H_i),\pi(F_i),\pi(F_{i+1}),\pi(H_{i+1}))\leq D.
\end{equation}

By Proposition \ref{prop: positivity projection}, the tuple
\[(\pi(F_1),\pi(F_2),\dots,\pi(F_i),\pi(H_i),\dots,\pi(H_2),\pi(H_1))\]
is positive for any positive integer $i$. Also, Proposition \ref{prop: converge} implies that $(\pi(F_i))_{i=1}^\infty$ and $(\pi(H_i))_{i=1}^\infty$ converge to $\pi(F_\infty)$ and $\pi(H_\infty)$ respectively. Since (\ref{eqn: ineq 5'}) holds for all integers $i>0$, we may then apply Proposition \ref{prop: k=1'} to deduce that $\pi(F_\infty)^{(1)}=\pi(H_\infty)^{(1)}$. This implies that $F_\infty^{(k)}=H_\infty^{(k)}$, which is a contradiction. 
\end{proof}

\section{Weakly positive representations} \label{sec:admissible}
In this section, we prove our main result Theorem \ref{thm: main intro}, which we restate here. 

\begin{theorem}\label{thm: weakly positive is directed Anosov}
Suppose that $\Lambda$ is path-symmetric. If $\rho:\Gamma\to\PGL(V)$ is $(R,\Lambda)$-weakly positive, then $\rho$ is $(R,\Lambda)$-directed Anosov.
\end{theorem} 

Since $(R,\Lambda)$-weakly positive representations are easy to construct, this theorem provides a process to construct many $(R,\Lambda)$-directed Anosov representations. Several such examples are given in Section \ref{sec: 6}.

\subsection{Relation to purely hyperbolic Schottky representations}\label{sec: adm}

As a special case of Theorem \ref{thm: weakly positive is directed Anosov}, we recover a theorem of Burelle-Treib \cite{BT} about purely hyperbolic Schottky representations from a free group $F_d$ (of any rank $d$) to $\PGL(V)$. In our language, such representations can be defined as follows. Choose a continuous embedding of the Gromov boundary $\partial F_d$ of $F_d$ into $\mathbb S^1$ (for instance by identifying $F_d$ with the fundamental group of a convex cocompact hyperbolic surface).

\begin{definition}\label{def: purely hyperbolic Schottky}
Let $\Lambda$ be the finite directed graph with $A=\{a_1,\dots,a_d,b_1,\dots,b_d\}$ as its set of vertices and $B=A^2-\{(a_1,b_1),\dots,(a_d,b_d),(b_1,a_1),\dots,(b_d,a_d)\}$ as its set of edges. Let $R:A\to F_d$ be a map such that $R(\{a_1,\dots,a_d\})$ is a generating set of $F_d$, and $R(b_i)=R(a_i)^{-1}$ for all $i=1,\dots,m$. A representation $\rho:F_d\to \PGL(V)$ is a \emph{purely hyperbolic Schottky representation} if there is a compatible system of forward domains $\{\mathfrak{U}_a=\mathfrak{U}(F_a,H_a,K_a):a\in A\}$ for $(\rho\circ R,\Lambda)$ such that the following holds: If we enumerate $R(A)=\{c_1,\dots,c_{2m}\}$ so that $R(c_1)_+<R(c_2)_+<\dots<R(c_{2m})_+$ according to the cyclic order on $\mathbb{S}^1$, then 
\[(G_{c_1},G_{c_2},\dots,G_{c_{2m}})\]
is positive for all $G_{c_i}\in\overline{\mathfrak U_{c_i}}$.
 \end{definition}

\begin{corollary}\cite[Theorem 1.3]{BT} If $\rho:F_d\to\PGL(V)$ is a purely hyperbolic Schottky representation, then $\rho$ is $\Delta$-Anosov.
\end{corollary}

\begin{proof}
Let $\Lambda$ and $R$ be as specified in Definition \ref{def: purely hyperbolic Schottky}. Since $\Lambda$ is path-symmetric, Theorem \ref{thm: weakly positive is directed Anosov} implies that $\rho$ is $(R,\Lambda)$-directed Anosov. The corollary now follows from the observation that all geodesic rays in $F_d$ (with generating set $R(A)$) are $(R,\Lambda)$-directed geodesic rays.
\end{proof}

In fact, our proof of Theorem \ref{thm: weakly positive is directed Anosov}, when specialized to the case when $\rho$ is a purely hyperbolic Schottky representation, becomes a ping-pong argument that is very similar to the argument used by Burelle and Treib. A key difference between the two arguments however, is that we use the Labourie cross ratio to show that the intersection of certain nested sequence of open sets is a point (see Theorem \ref{thm: general k}), while Burelle and Treib used what they call the interval distance in place of the Labourie cross ratio (see \cite[Proposition 3.25]{BT}). The fact that the Labourie cross ratio is invariant under the $\PGL(V)$-action, while the interval distance is not, allows us to prove the same result under the weaker hypothesis of being $(R,\Lambda)$-weakly positive instead of being purely hyperbolic Schottky.

\subsection{Admissible maps and collapsing domains in $\Fc(V)$}
As an intermediate step to prove Theorem \ref{thm: weakly positive is directed Anosov}, we prove the following proposition.

\begin{proposition}\label{prop: main 1}
Let $f:A\to\PGL(V)$ be a $\Lambda$-admissible map, and let 
\[\{\mathfrak{U}_a=\mathfrak{U}(F_a,H_a,K_a):a\in A\}\] 
be a compatible system of forward domains for $f$ (see Section \ref{sec: intro main}). Fix a directed ray $(p_i)_{i=1}^\infty$ in $\Lambda$, and let $v_i:=f(p_1)\dots f(p_i)$ for all $i\geq 0$ ($v_0:=\id$).  Then the following hold:
\begin{enumerate}
\item For all integers $i\geq 1$, $v_i\cdot \mathfrak U_{p_{i+1}}=\mathfrak{U}(v_i\cdot F_{p_{i+1}},v_i\cdot H_{p_{i+1}},K_{p_1})$ and the tuple
\[(v_0\cdot F_{p_1},v_1\cdot F_{p_2},\dots,v_i\cdot F_{p_{i+1}},v_i\cdot H_{p_{i+1}},\dots,v_1\cdot H_{p_2},v_0\cdot H_{p_1},K_{p_1})\]
is positive up to switching $F_{p_{j+1}}$ and $H_{p_{j+1}}$ for some of the $j\in\{1,\dots,i\}$. 
\item The intersection $\displaystyle\bigcap_{i=0}^\infty \overline{v_i\cdot \mathfrak U_{p_{i+1}}}$ is a point.
\end{enumerate}
\end{proposition}

\begin{proof}
Proof of (1). We prove this by induction on $i$. First, observe that the $\Lambda$-admissibility of $f$ and Proposition~\ref{prop: easy positive}(2) imply that for all positive integers $j$, 
\begin{align}\label{eqn: nested step 1}
(F_{p_j},f(p_j)\cdot F_{p_{j+1}},f(p_j)\cdot H_{p_{j+1}},H_{p_j},K_{p_j})
\end{align}
is positive up to switching $F_{p_{j+1}}$ and $H_{p_{j+1}}$, and 
\begin{align}\label{eqn: nested step 2}
f(p_j)\cdot\mathfrak U_{p_{j+1}}=\mathfrak U(f(p_j)\cdot F_{p_{j+1}},f(p_j)\cdot H_{p_{j+1}},K_{p_j}).
\end{align} 
The base case $i=1$ is then simply the specialization of this observation to the case when $j=1$.

Next, we prove the inductive step. Suppose that $i>1$. Since $f$ is $\Lambda$-admissible,
\[v_i\cdot\mathfrak U_{p_{i+1}}\subset v_{i-1}\cdot\mathfrak U_{p_i}=\mathfrak{U}(v_{i-1}\cdot F_{p_i},v_{i-1}\cdot H_{p_i},K_{p_1}),\]
where the equality holds by the inductive hypothesis, and \eqref{eqn: nested step 1} implies that
\[(v_{i-1}\cdot F_{p_i},v_i\cdot F_{p_{i+1}},v_i\cdot H_{p_{i+1}},v_{i-1}\cdot H_{p_i})\] 
is positive up to switching $F_{p_{i+1}}$ and $H_{p_{i+1}}$. Then Proposition~\ref{prop: easy positive}(2) implies that 
\[v_i\cdot\mathfrak U_{p_{i+1}}=\mathfrak{U}(v_i\cdot F_{p_{i+1}},v_i\cdot H_{p_{i+1}},K_{p_1})\]
and 
\[(v_{i-1}\cdot F_{p_i},v_i\cdot F_{p_{i+1}},v_i\cdot H_{p_{i+1}},v_{i-1}\cdot H_{p_i},K_{p_1})\]
is positive. The inductive hypothesis also implies that
\[(v_0\cdot F_{p_1},v_1\cdot F_{p_2},\dots,v_{i-1}\cdot F_{p_i},v_{i-1}\cdot H_{p_i},\dots,v_1\cdot H_{p_2},v_0\cdot H_{p_1},K_{p_1})\] 
is positive up to switching $F_{p_{j+1}}$ and $H_{p_{j+1}}$ for some of the $j$ in $\{1,\dots,i-1\}$, so we may apply Proposition \ref{obs: basic flag}(5) to finish the proof.

Proof of (2). For each $i>1$, set
\[(F_i,H_i):=\left\{\begin{array}{ll}
(F_{p_i},H_{p_i})&\text{if }(F_{p_1},v_{i-1}\cdot F_{p_i},v_{i-1}\cdot H_{p_i},H_{p_1})\text{ is positive};\\
(H_{p_i},F_{p_i})&\text{if }(F_{p_1},v_{i-1}\cdot H_{p_i},v_{i-1}\cdot F_{p_i},H_{p_1})\text{ is positive},
\end{array}
\right.\]
and set $(F_1,H_1):=(F_{p_1},H_{p_1})$. By (1), 
\[(F_1,v_1\cdot F_2,\dots,v_i\cdot F_{i+1},v_i\cdot H_{i+1},\dots,v_1\cdot H_2,H_1,K_{p_1})\]
is positive for all positive integers $i$. Thus, by Theorem \ref{thm: general k}, it is sufficient to show that there is some $D>1$ such that $B_k(v_{i-1}\cdot H_i,v_{i-1}\cdot F_i,v_i\cdot F_{i+1},v_i\cdot H_{i+1})\leq D$ for all $k=1,\dots,n-1$ and all integers $i>0$. 

Let $D:=\max\{D',D''\}$, where
\[D':=\max\{B_k(H_{a_1},F_{a_1},f(a_1)\cdot H_{a_2},f(a_1)\cdot F_{a_2}):(a_1,a_2)\in B,k=1,\dots,n-1\}\]
and
\[D'':=\max\{B_k(H_{a_1},F_{a_1},f(a_1)\cdot F_{a_2},f(a_1)\cdot H_{a_2}):(a_1,a_2)\in B,k=1,\dots,n-1\}.\]
Then for any integer $i>0$ and any $k=1,\dots,n-1$. 
\begin{align*}
B_k(v_{i-1}\cdot H_i,v_{i-1}\cdot F_i,v_i\cdot F_{i+1},v_i\cdot H_{i+1})&=B_k(H_i,F_i,f(p_i)\cdot F_{i+1},f(p_i)\cdot H_{i+1})\leq D.\qedhere
\end{align*}

\end{proof}

Proposition \ref{prop: main 1} has the following consequences. For any integer $m>0$, we say that a directed ray $(p_i)_{i=1}^\infty$ in $\Lambda$ is \emph{$m$-periodic} if $p_{i+m}=p_i$ for all integers $i>0$.

\begin{corollary}\label{cor: positive loxodromic}
Let $f:A\to\PGL(V)$ be a $\Lambda$-admissible map, and let 
\[\{\mathfrak{U}_a:=\mathfrak U(F_a,H_a,K_a):a\in A\}\] 
be a compatible system of forward domains for $(f,\Lambda)$. If $(p_i)_{i=1}^\infty$ is a $m$-periodic directed ray in $\Lambda$, then  $w:=f(p_1)\dots f(p_m)$ is loxodromic, and for all integers $i$, the attracting (respectively, repelling) fixed point $w_+$ (respectively, $w_-$) in $\mathcal F(V)$ of $w$ lies in $w^i\cdot \mathfrak U_{p_1}$ (respectively, $w^i\cdot \mathfrak U_{p_1}^{\rm opp}$).
\end{corollary}

\begin{proof}
By Proposition \ref{prop: main 1}(2), $(\overline{w^i\cdot \mathfrak U_{p_1}})_{i=0}^\infty$ is a nested sequence of compact sets and $\displaystyle\bigcap_{i=0}^\infty\overline{w^i\cdot \mathfrak{U}_{p_1}}$ is a point. Since $\mathfrak U_{p_1}$ is open, this point is necessarily the attracting fixed point $w_+\in\Fc(V)$ of $w$, and
\[\mathfrak{U}_{p_1}^{\rm opp}=\mathfrak{U}(F_{p_1},H_{p_1},w_+).\]

Also, Proposition \ref{prop: main 1}(1) implies that for all positive integers $i$,
\[(F_{p_1},w\cdot F_{p_1},\dots,w^i\cdot F_{p_1},w_+,w^i\cdot H_{p_1},\dots,w\cdot H_{p_1},H_{p_1})\]
is positive up to switching $w^j\cdot F_{p_1}$ and $w^j\cdot H_{p_1}$ for some of the $j\in\{1,\dots,i\}$, so the same is true for 
\[(F_{p_1},w^{-1}\cdot F_{p_1}\dots,w^{-i}\cdot F_{p_1},w^{-i}\cdot H_{p_1},\dots,w^{-1}\cdot H_{p_1},H_{p_1},w_+).\]
Since $B_k(w^{-j}\cdot H_{p_1},w^{-j}\cdot F_{p_1},w^{-j-1}\cdot F_{p_1},w^{-j-1}\cdot H_{p_1})=B_k(H_{p_1},F_{p_1},w^{-1}\cdot F_{p_1},w^{-1}\cdot H_{p_1})$ for all $k\in\{1,\dots,n-1\}$ and all integers $j\geq 0$, Theorem \ref{thm: general k} implies that $\displaystyle\bigcap_{i=0}^\infty\overline{w^{-i}\cdot \mathfrak{U}_{p_1}^{\rm opp}}$ is a point. Since $\mathfrak U_{p_1}^{\rm opp}$ is open, this point is necessarily the repelling fixed point $w_-\in\Fc(V)$ of $w$. Apply Observation \ref{obs: lox}.
\end{proof}

\begin{corollary}\label{cor: main 2}
Let $f:A\to\PGL(V)$ be a $\Lambda$-admissible map, and let 
\[\{\mathfrak{U}_a:=\mathfrak U(F_a,H_a,K_a):a\in A\}\] 
be a compatible system of forward domains for $(f,\Lambda)$. Let $(w_i)_{i=1}^\infty$ be a sequence in $\PGL(V)$ so that for each $i$, there is a positive integer $l_i$ such that $\lim_{i\to\infty}l_i=\infty$ and a directed ray $(p_{i,j})_{j=1}^\infty$ in $\Lambda$ such that
\[w_i=f(p_{i,1})f(p_{i,2})\dots f(p_{i,l_i}).\]
Then up to taking a subsequence, there is some $\bar{a}\in A$ such that $(\overline{w_i\cdot \mathfrak U_{\bar{a}}})_{i=1}^\infty$ is a nested sequence of compact sets whose intersection is a point.
\end{corollary}

\begin{proof}

By taking a subsequence, we may assume that $l_i\geq i$ for all $i>0$, $p_{i,l_i}=p_{j,l_j}=:\bar{p}$ for all $i,j>0$, and $p_{i,l}=p_{j,l}$ for all $l>0$ and all $i,j\geq l$. In other words, if we set $w_i'=f(p_{i,{i+1}})\dots f(p_{i,l_i})$ ($w_i'=\id$ if $l_i=i$), then one of the following cases must hold:
\begin{enumerate}
\item[(i)] There is a tuple $(m_1,\dots,m_{k-1})$ of positive integers and a tuple $(a_1,\dots,a_k)$ of points in $A$ such that $a_l\neq a_{l+1}$ for all $0<l<k$, and 
\[w_i=f(a_1)^{m_1}f(a_2)^{m_2}\dots f(a_{s_i-1})^{m_{s_i-1}} f(a_{s_i})^{i-M_i}w_i'\]
for all integers $i>0$, where $0<s_i\leq k$ is the integer such that $\displaystyle M_i:=\displaystyle\sum_{l=1}^{s_i-1}m_l\leq i$.
\item[(ii)] There is a sequence $(m_l)_{l=1}^\infty$ of positive integers and a sequence $(a_l)_{l=1}^\infty$ in $A$ such that $a_l\neq a_{l+1}$ for all $l>0$, and  
\[w_i=f(a_1)^{m_1}f(a_2)^{m_2}\dots f(a_{s_i-1})^{m_{s_i-1}} f(a_{s_i})^{i-M_i}w_i'\]
for all integers $i>0$, where $s_i>0$ is the integer such that $\displaystyle M_i:=\displaystyle\sum_{l=1}^{s_i-1}m_l\leq i$.
\end{enumerate}
Let $\bar{a}\in A$ be any point such that $(\bar{p},\bar{a})\in B$. We will prove that in either case, $\displaystyle\lim_{i\to\infty}w_i\cdot F_{\bar{a}}=\lim_{i\to\infty}w_i\cdot H_{\bar{a}}$ for some $\bar{a}\in A$.

First, suppose that (i) holds. By removing the first $M_k$ terms of the sequence, we may assume that $w_i=w f(a_k)^iw_i'$, where $w:=f(a_1)^{m_1}\dots f(a_{k-1})^{m_{k-1}}$ ($w=\id$ if $k=1$). By Proposition \ref{prop: main 1}(2), 
\[\displaystyle\lim_{i\to\infty} \overline{f(a_k)^i\cdot\mathfrak U_{p_{i,i+1}}}\subset\lim_{i\to\infty} \overline{f(a_k)^{i-1}\cdot\mathfrak U_{a_k}}=\bigcap_{i=1}^\infty\overline{f(a_k)^{i-1}\cdot\mathfrak U_{a_k}}=\{K\}\]
for some flag $K\in\Fc(V)$. Also, the $\Lambda$-admissibility of $f$ implies that $\overline{w_i'\cdot\mathfrak U_{\bar{a}}}\subset\mathfrak U_{p_{i,i+1}}$ for all integers $i>0$, so
\[\lim_{i\to\infty}w_i\cdot F_{\bar{a}}= w\lim_{i\to\infty}f(a_k)^iw_i'\cdot F_{\bar{a}}\in w\lim_{i\to\infty}\overline{f(a_k)^i\cdot\mathfrak U_{p_{i,i+1}}}=w\cdot K=\lim_{i\to\infty}w_i\cdot H_{\bar{a}}.\]

Now suppose that (ii) holds. By taking a further subsequence, we may assume that $w_i=v_i f(a_i) w_i'$ for all integers $i>0$, where $v_i:=f(a_1)^{m_1}\dots f(a_i)^{m_i-1}$. By Proposition~\ref{prop: main 1}(2), 
\[\displaystyle\lim_{i\to\infty} \overline{v_if(a_i)\cdot\mathfrak U_{p_{i,i+1}}}\subset\lim_{i\to\infty} \overline{v_i\cdot\mathfrak U_{a_i}}=\bigcap_{i=1}^\infty \overline{v_i\cdot\mathfrak U_{a_i}}=\{K\}\]
for some flag $K\in\Fc(V)$, and the $\Lambda$-admissibility of $f$ implies that $\overline{w_i'\cdot\mathfrak U_{\bar{a}}}\subset \mathfrak U_{p_{i,i+1}}$ for all integers $i>0$, so
\[\lim_{i\to\infty}w_i\cdot F_{\bar{a}}= \lim_{i\to\infty}v_if(a_i)w_i'\cdot F_{\bar{a}}\in \lim_{i\to\infty}v_if(a_i)\cdot\overline{\mathfrak U_{p_{i,i+1}}}=K=\lim_{i\to\infty}w_i\cdot H_{\bar{a}}.\]

Since $\displaystyle\lim_{i\to\infty}w_i\cdot F_{\bar{a}}=\lim_{i\to\infty}w_i\cdot H_{\bar{a}}$ in either case, Proposition \ref{prop: shrink} implies that $(\overline{w_i\cdot \mathfrak U_{\bar{a}}})_{i=1}^\infty$ is a nested sequence of compact sets whose intersection is a point.
\end{proof}

\subsection{Proof of Theorem \ref{thm: weakly positive is directed Anosov}}
Recall that $X$ denotes the $\PGL(V)$-Riemannian symmetric space, and $\Delta$ denotes the set of simple roots of $\PGL(V)$. Fix a base point $o\in X$ with which we define a Cartan projection $\mu:\GL(V)\to\R^n$. By Theorem \ref{thm: main 2}, to prove Theorem \ref{thm: weakly positive is directed Anosov}, it is sufficient to prove the following pair of statements.

\begin{proposition}\label{prop: ppmfb}
Suppose that $\Lambda$ is path-symmetric, and let $\rho:\Gamma\to\PGL(V)$ be an $(R,\Lambda)$-weakly positive representation. 
\begin{enumerate}
\item There is some $C>0$ such that for every $(R,\Lambda)$-directed sequence $(\eta_i)_{i=0}^\infty$ in $\Gamma$, the sequence $(\rho(\eta_i)\cdot o)_{i=0}^\infty$ in $X$ is $C$-bounded from a maximal flat in $X$.
\item $\displaystyle\lim_{i\to\infty}\alpha\circ\mu(\rho(\gamma_i))=\infty$ for every $\alpha\in\Delta$ if for each $i$, there is an integer $l_i$ such that $\lim_{i\to\infty}l_i=\infty$ and a 
$(R,\Lambda)$-directed sequence $(\eta_{i,l})_{l=0}^\infty$ such that $\gamma_i=\eta_{i,l_i}$.
\end{enumerate}
\end{proposition}

Given a transverse pair of flags $\{F_1,F_2\}$ in $\Fc(V)$, there is a unique maximal, diagonalizable, connected, abelian subgroup $A\subset\PGL(V)$ that stabilizes both $F_1$ and $F_2$. In this case, we denote $\mathbf F(F_1,F_2):=\mathbf F_A$ (where $\mathbf F_A$ is the maximal flat defined in Section \ref{sec: flats}), and we say that $\mathbf F(F_1,F_2)$ is the flat \emph{asymptotic} to the transverse pair of flags $\{F_1,F_2\}$. The proof of Proposition \ref{prop: ppmfb}(1) relies on the following lemma. 

\begin{lemma}\label{lem: C exists}
Let $(F_1,H_1,H_2,F_2)$ be a positive quadruple of flags in $\Fc(V)$, let 
\[\mathfrak U_1:=\mathfrak U(F_1,H_1,H_2)=\mathfrak U(F_1,H_1,F_2),\text{ and let }\mathfrak U_2:=\mathfrak U(F_2,H_2,H_1)=\mathfrak U(F_2,H_2,F_1).\] 
Then there is some $D=D(F_1,H_1,H_2,F_2)>0$ such that for all $G_i\in\overline{\mathfrak U_i}$,
\[d_X(o,\mathbf F(G_1,G_2))\leq D.\] 
\end{lemma}

\begin{proof}
Let $\mathfrak{T}$ denote the set of transverse pairs of flags in $\Fc(V)$, and let $\phi:\mathfrak{T}\to\mathbb{R}$ be the continuous function given by $\phi:(G_1,G_2)\mapsto d_X(o,\mathbf F(G_1,G_2))$. Observe that if $K_1,K_2\in\Fc(V)$ are flags such that $(F_2,K_1,F_1,H_1,K_2,H_2)$ is positive, then by Proposition \ref{prop: easy positive}(3) $(G_1,K_1,G_2,K_2)$ is positive for all $G_i\in\overline{\mathfrak U_i}$. Since positive tuples of flags are in general position, it follows that $\overline{\mathfrak U_1}\times \overline{\mathfrak U_2}\subset\mathfrak T$.  The compactness of $\overline{\mathfrak{U}_1}\times\overline{\mathfrak{U}_2}$ then implies that $\phi({\overline{\mathfrak{U}_1}\times\overline{\mathfrak{U}_2}})$ is bounded, and thus has an upper bound $D$.
\end{proof}

On the other hand, the proof of Proposition \ref{prop: ppmfb}(2) uses the following lemma, which is a standard consequence of the singular value decomposition.

\begin{lemma}\label{lem: singular values Cartan projection}
Let $(g_i)_{i=0}^\infty$ be a sequence in $\PGL(V)$. If there is an open set $\mathfrak U\subset\Fc(V)$ such that $g_i\cdot \mathfrak U$ converges to a point as $i$ goes to $\infty$, then
\[\lim_{i\to\infty}\alpha\circ\mu(g_i)=\infty\]
for all $\alpha\in\Delta$. 
\end{lemma}
\begin{proof}[Proof of Proposition \ref{prop: ppmfb}]
Proof of (1). Let $f:=\rho\circ R$, let $\{\mathfrak{U}_a:=\mathfrak U(F_a,H_a,K_a):a\in A\}$ be a compatible system of forward domains for $(f,\Lambda)$, and let 
\[C:=\max\{D(F_{a_1},H_{a_1},f(a_1)\cdot H_{a_2},f(a_1)\cdot F_{a_2}):(a_1,a_2)\in B\}>0,\]
where $D(F_1,H_1,H_2,F_2)$ is the constant associated to a positive quadruple $(F_1,H_1,H_2,F_2)$ given in Lemma \ref{lem: C exists}. Let $(\eta_i)_{i=0}^\infty$ be an $(R,\Lambda)$-directed sequence in $\Gamma$, and let $(p_i)_{i=1}^\infty$ be the directed ray in $\Lambda$ such that $\eta_i=R(p_1)\dots R(p_i)$ for all integers $i> 0$.

First, we prove (1) in the special case when $(\eta_i)_{i=0}^\infty$ (equivalently, $(p_i)_{i=1}^\infty$) is $m$-periodic for some integer $m>0$. Then for all $i\geq 0$, observe that 
\[\eta_i^{-1}\eta_m\eta_i=R(p_{i-km+1})\dots R(p_m)R(p_1)\dots R(p_{i-km}),\]
where $k\geq 0$ is the largest integer such that $km\leq i$. Thus, by Corollary \ref{cor: positive loxodromic}, 
\[v_i:=\rho(\eta_i^{-1}\eta_m\eta_i)\] 
is loxodromic, $(v_i)_+\in v_i\cdot \mathfrak{U}_{p_{i-km+1}}\subset f(p_{i-km+1})\cdot \mathfrak{U}_{p_{i-km+2}}$, and $(v_i)_-\in \mathfrak{U}_{p_{i-km+1}}^{\rm opp}$. Furthermore, if we set $\mathbf F_i:=\mathbf F((v_i)_-,(v_i)_+)$, then $\rho(\eta_i^{-1})\cdot \mathbf F_0=\mathbf F_i$. Hence, we may apply Lemma \ref{lem: C exists} to deduce that 
\begin{align*}
d_X(\rho(\eta_i)\cdot o, \mathbf F_0)&=d_X(o,\rho(\eta_i^{-1})\cdot \mathbf F_0)=d_X(o,\mathbf F_i)\\
&\leq D(F_{p_{i-km+1}},H_{p_{i-km+1}},f(p_{i-km+1})\cdot H_{p_{i-km+2}},f(p_{i-km+1})\cdot F_{p_{i-km+2}})\\
&\leq C.
\end{align*}


Using this, we prove the general case where $(\eta_i)_{i=0}^\infty$ is an arbitrary $(R,\Lambda)$-directed sequence in $\Gamma$.  Since $\Lambda$ is path-symmetric, for each integer $i>0$, there is a $m_i$-periodic directed ray $(q_{i,j})_{j=0}^\infty$ such that $q_{i,j}=p_j$ for all $j=0,\dots,i$. Let $\gamma_{i,j}:=R(q_{i,1})\dots R(q_{i,j})$ and $w_{i,j}:=\rho(\gamma_{i,j})$ for all integers $i,j>0$. If we denote $\mathbf F_i:=\mathbf F((w_{i,m_i})_+,(w_{i,m_i})_-)$, then by the special case we proved above,
\[d_X(\rho(\eta_j)\cdot o,\mathbf F_i)=d_X(w_{i,j}\cdot o,\mathbf F_i)\leq D\]
for all integers $i>0$ and all $j=1,\dots,i$. Furthermore, by Corollary \ref{cor: positive loxodromic}, 
\[(w_{i,m_i})_+\in w_{i,m_i}\cdot \mathfrak{U}_{p_{1}}\subset f(p_{1})\cdot \mathfrak{U}_{p_{2}},\text{ and }(w_{i,m_i})_-\in \mathfrak{U}_{p_1}^{\rm opp}\] 
for all integers $i>1$. The compactness of $\overline{f(p_1)\cdot \mathfrak U_{p_2}}\times\overline{\mathfrak U_{p_1}^{\rm opp}}$ then ensures that up to taking a subsequence, the pair $((w_{i,m_i})_+,(w_{i,m_i})_-)$ converges to a transverse pair of flags $(F_1,F_2)$ in $\overline{f(p_1)\cdot \mathfrak U_{p_2}}\times\overline{\mathfrak U_{p_1}^{\rm opp}}$. Thus, for all integers $j\geq 0$,
\[d_X(\rho(\eta_j)\cdot o,\mathbf F(F_1,F_2))=\lim_{i\to\infty}d_X(\rho(\eta_j)\cdot o,\mathbf F_i)\leq D.\] 

Proof of (2). Corollary \ref{cor: main 2} implies that by taking subsequences, there is some $\bar{a}\in A$ such that $(\overline{\rho(\gamma_i)\cdot \mathfrak U_{\bar{a}}})_{i=1}^\infty$ is a nested sequence of compact sets whose intersection is a point. Apply Lemma \ref{lem: singular values Cartan projection}.
\end{proof}

\section{Applications to primitive stability}\label{sec: 6}

The goal of this section is to give an application of Theorem \ref{thm: weakly positive is directed Anosov} in the setting of primitive stable representations. 

\subsection{Primitive stable representations and weak positivity}
We first define the notion of a $\theta$-primitive stable representation. 

\begin{definition}\label{def: primitive stable}
Let $\theta\subset\Delta$ be a non-empty subset. Equip $F_d$ with a word metric. A representation $\rho :F_d \to\PGL(V)$ is \emph{$\theta$-primitive stable} if there exists constants $\kappa, \kappa'>0$ such that 
\[\alpha\circ\mu(\rho(\eta_i))\geq \kappa i-\kappa'\]
for all rooted, primitive geodesic rays $(\eta_i)_{i=0}^\infty$ in $F_d$, all integers $i\geq 0$ and all $\alpha\in\theta$.\end{definition}


Changing the word metric on $F_d$ might result in different constants $\kappa$ and $\kappa'$, but their existence does not depend on this choice. Since we will focus exclusively on $\Delta$-primitive stable representations, henceforth, we refer to $\Delta$-primitive stable representations simply as primitive stable representations. This recovers Definition \ref{def: primitive stable intro}.

For our application of Theorem \ref{thm: weakly positive is directed Anosov}, we focus on the case where $d=2$. The reason for this is the following theorem in Cohen-Metzler-Zimmermann \cite{CMZ81}.

\begin{theorem}[\cite{CMZ81}]\label{thm: CMZ}
Fix a pair of generators $\{\gamma_1,\gamma_2\}$ of $F_2$. Up to replacing $\gamma_1$ with $\gamma_1^{-1}$, replacing $\gamma_2$ with $\gamma_2^{-1}$, and switching the roles of $\gamma_1$ and $\gamma_2$, every primitive element in $F_2$ is conjugate to $\gamma_1$ or an element of the form 
\[\gamma_1\gamma_2^{m_1}\gamma_1\gamma_2^{m_2}\dots\gamma_1\gamma_2^{m_k}\]
where $m_i\in\{l,l+1\}$ for all $i=1,\dots,k$, and $k,l>0$ are integers.
\end{theorem}

Recall that in the introduction, we defined, for all positive integers $i$, the graph 
$$\Lambda_i=(A_i,B_i):=\mathsf{K}(v_1,w_1)\cup \mathsf{K}(v_1', w_1')\cup \cdots \cup \mathsf{K}(v_i,w_i)\cup \mathsf{K}(v_i',w_i').$$ 
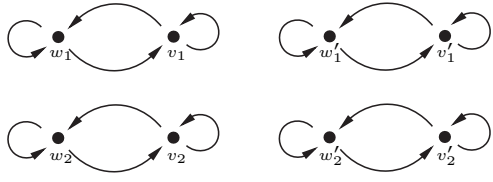
\begin{figure}[h]
    \centering
    \includegraphics[width=0.45\textwidth]{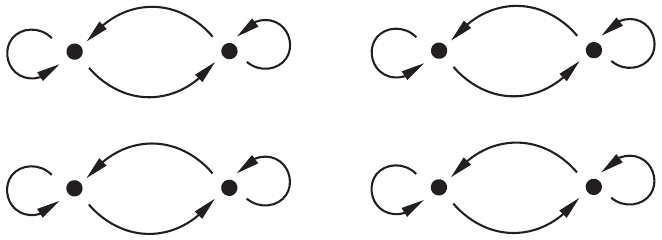}
    \tiny
\put (-123,45){$v_1$}
\put (-167,45){$w_1$}
\put (-22,45){$v_1'$}
\put (-66,45){$w_1'$}
\put (-123,6){$v_2$}
\put (-167,6){$w_2$}
\put (-22,6){$v_2'$}
\put (-66,6){$w_2'$}
    \caption{Picture of $\Lambda_4$.}
    \label{fig: complete directed graph}
\end{figure}
Also, we defined, for $i=2,3$, what it means for a map $R_i:A_i\to\Gamma$ to be defined by a pair of elements $(\gamma_1,\gamma_2)$ in $\Gamma$. 
In the case when $\{\gamma_1,\gamma_2\}$ is a pair of generators of $F_2$, the triple $\{\gamma_1,\gamma_2,\gamma_3\}\subset F_2$ is a \emph{superbasis} of $F_2$, i.e. $\{\gamma_1^{-1},\gamma_2\}$, $\{\gamma_2^{-1},\gamma_3\}$, and $\{\gamma_3^{-1},\gamma_1\}$ are generating sets of $F_2$, and $\gamma_1\gamma_2\gamma_3=\id$. 

Using Theorem \ref{thm: weakly positive is directed Anosov} and Theorem \ref{thm: CMZ}, we prove the following. 

\begin{proposition}\label{prop: forward primitive}
Let $(\gamma_1,\gamma_2)$ be a pair of generators of $F_2$, and for $i=2,3$, let $R_i:A_i\to F_2$ be defined by $(\gamma_1,\gamma_2)$. If $\rho:F_2\to\PGL(V)$ is $(R_2,\Lambda_2)$-directed Anosov or $(R_3,\Lambda_3)$-directed Anosov, then it is primitive stable. In particular, if $\rho:F_2\to\PGL(V)$ is $(R_2,\Lambda_2)$-weakly positive or $(R_3,\Lambda_3)$-weakly positive, then it is primitive stable.
\end{proposition}

\begin{proof}
Let $Y\subset F_2$ denote the set of elements that, up to replacing $\gamma_1$ with $\gamma_1^{-1}$, replacing $\gamma_2$ with $\gamma_2^{-1}$, and switching the roles of $\gamma_1$ and $\gamma_2$, are of the form $\gamma_1$ or $\gamma_1\gamma_2^{m_1}\gamma_1\gamma_2^{m_2}\dots\gamma_1\gamma_2^{m_k}$, where $m_i\in\{l,l+1\}$ for all $i=1,\dots,k$, and $k,l>0$ are integers. 

Suppose first that $\rho$ is $(R_2,\Lambda_2)$-directed Anosov. Let $(\eta_i)_{i=0}^\infty$ be a rooted primitive geodesic ray in $F_2$ in the word metric associated to $\{\gamma_1,\gamma_2,\gamma_1^{-1},\gamma_2^{-1}\}$, and let $\eta\in F_2$ be a primitive element whose axis contains $(\eta_i)_{i=0}^\infty$. Then $\eta$ is cyclically reduced, so by Theorem \ref{thm: CMZ}, we may assume that $\eta\in Y$, which implies that $(\eta_i)_{i=0}^\infty$ is $(R_2,\Lambda_2)$-directed. In particular, every rooted primitive geodesic ray in $\Gamma$ is $(R_2,\Lambda_2)$-directed, so  $(R_2,\Lambda_2)$-directed Anosov representations are primitive stable.

On the other hand, suppose that $\rho$ is $(R_3,\Lambda_3)$-directed Anosov.  Let $(\eta_i)_{i=0}^\infty$ be a rooted primitive geodesic ray in $F_2$ in the word metric associated to $\{\gamma_1,\gamma_1^{-1},\gamma_2,\gamma_2^{-1},\gamma_3,\gamma_3^{-1}\}$, and let $\eta\in F_2$ be a primitive element whose axis contains $(\eta_i)_{i=0}^\infty$. As before, by Theorem \ref{thm: CMZ}, we may assume that $\eta\in Y$. If $\eta$ is neither a product of elements in $\{\gamma_1^{-1},\gamma_2\}$, nor a product of elements in $\{\gamma_1,\gamma_2^{-1}\}$, then because $\eta$ lies in $Y$, there are four other possibilities:
\begin{enumerate}
\item If $\eta=\gamma_1^{-1}\gamma_2^{-m_1}\gamma_1^{-1}\gamma_2^{-m_2}\dots\gamma_1^{-1}\gamma_2^{-m_k}$, then $\gamma_2^{-1}\eta\gamma_2=\gamma_3\gamma_2^{-m_1+1}\gamma_3\gamma_2^{-m_2+1}\dots\gamma_3\gamma_2^{-m_k+1}$,
\item If $\eta=\gamma_1\gamma_2^{m_1}\gamma_1\gamma_2^{m_2}\dots\gamma_1\gamma_2^{m_k}$, then $\eta=\gamma_3^{-1}\gamma_2^{m_1-1}\gamma_3^{-1}\gamma_2^{m_2-1}\dots\gamma_3^{-1}\gamma_2^{m_k-1}$,
\item If $\eta=\gamma_2\gamma_1^{m_1}\gamma_2\gamma_1^{m_2}\dots\gamma_2\gamma_1^{m_k}$, then $\gamma_1\eta\gamma_1^{-1}=\gamma_3^{-1}\gamma_1^{m_1-1}\gamma_3^{-1}\gamma_1^{m_2-1}\dots\gamma_3^{-1}\gamma_1^{m_k-1}$,
\item If $\eta=\gamma_2^{-1}\gamma_1^{-m_1}\gamma_2^{-1}\gamma_1^{-m_2}\dots\gamma_2^{-1}\gamma_1^{-m_k}$, then $\eta=\gamma_3\gamma_1^{-m_1+1}\gamma_3\gamma_1^{-m_2+1}\dots\gamma_3\gamma_1^{-m_k+1}$.
\end{enumerate}
Thus, $\eta$ is conjugate to a product of elements in $R_3(\{v_i,w_i\})$ or $R_3(\{ v_i', w_i'\})$ for some $i=2,3$. This implies that $(\eta_i)_{i=0}^\infty$ is $(R_3,\Lambda_3)$-directed. We have thus shown that every primitive geodesic ray in $F_2$ is $(R_3,\Lambda_3)$-directed. The assumption that $\rho$ is $(R_3,\Lambda_3)$-directed Anosov implies that $\rho$ is primitive stable.

The second claim of the proposition follows immediately from Theorem \ref{thm: weakly positive is directed Anosov}. 
\end{proof}

\begin{remark}
When $n=2$, the converse to Proposition \ref{prop: forward primitive} holds, see Appendix \ref{sec: converse}.
\end{remark}

Observe that unlike primitive stability, being $(R_2,\Lambda_2)$ or $(R_3,\Lambda_3)$-weakly positive is a finite collection of conditions that can be explicitly verified. Thus, Proposition \ref{prop: forward primitive} gives us a way to construct new and interesting examples of primitive stable representations. In particular, we can prove the following theorem.

\begin{theorem}\label{thm: general n}
Let $b\in\PGL(V)$ be positive loxodromic, and let $a\in\PGL(V)$ be loxodromic.
Let $f$ be a map from the vertex set $A_2$ of $\Lambda_2$ to $\PGL(V)$ defined by 
$$(f(v_1),f(w_1))=(a,b), \,\,(f(v_2),f(w_2))=(a^{-1},b),$$
and $f(v_i')=f(v_i)^{-1}, f(w_i')=f(w_i)^{-1}$ for $i=1,2$. If $(b_-,a\cdot b_-,a_+,a\cdot b_+,b_+,a_-)$ is positive up to switching $a\cdot b_-$ and $a\cdot b_+$, then $f$ is $\Lambda_2$-admissible. In particular, if $\{\gamma_1,\gamma_2\}$ is a generating pair for $F_2$ and $\rho:F_2\to\PGL(V)$ is the representation defined by $\rho(\gamma_1)=a$ and $\rho(\gamma_2)=b$, then $\rho$ is $(R_2,\Lambda_2)$-weakly positive and thus primitive stable.
\end{theorem}




When $n=3$, we can further weaken the hypothesis of Theorem \ref{thm: general n} to obtain the following theorem.

\begin{theorem}\label{thm: positive quadruple}
Suppose that $n=3$. If $a,b\in\PGL(V)$ are positive loxodromic elements such that $(b_-,a_+,b_+,a_-)$ is a positive quadruple of flags, then the map $f$ defined in Theorem \ref{thm: general n} is $\Lambda_2$-admissible.
\end{theorem}

When $n=2$, the analog of Theorem \ref{thm: positive quadruple} is a consequence of a result of Goldman \cite[Section 3.2 and Lemma 3.4.5]{Go}. We do not know if Theorem \ref{thm: positive quadruple} holds for $n\geq 4$. 



Using Theorem \ref{thm: general n} and Theorem \ref{thm: positive quadruple}, we construct some explicit examples of primitive stable representations in Section \ref{sec: irred Fuch}. These examples include non-discrete representations and non-faithful representations when $n=3$, and non-positive representations for all $n$.



\subsection{Proof of Theorem \ref{thm: general n}}

As a preliminary step to prove Theorem \ref{thm: general n}, we use the following lemma.

\begin{lemma} \label{lem: minor facts}
Let $a\in\PGL(V)$ be positive loxodromic.
\begin{enumerate}
\item There is a flag $F\in\Fc(V)$ such that $(a_-,F,a\cdot F,a_+)$ is positive.
\item If $F, G, H\in\Fc(V)$ are flags such that $(a_-,F,a_+,H)$ and $(a_-,G,a_+,H)$ are positive, then there is some integer $N>0$ such that $(a_-,F,a^i\cdot G,a_+,H)$ is positive for all $i>N$.
\end{enumerate}
\end{lemma}

\begin{proof}
Proof of (1). Pick a point $x\in\Pb(V)$ that does not lie in the hyperplane $a_-^{(i)}+a_+^{(n-i-1)}$ for all $i\in\{0,\dots,n-1\}$. Then one can verify that the curve $C$ in $\Pb(V)$ parameterized by 
\[\mathbb{R}\to \Pb(V),\,\,t\mapsto a^t\cdot x\] 
is a smooth curve that is convex, i.e. any projective hyperplane intersects $C$ at most $n-1$ times. Let $F$ be the osculating flag to $C$ at $x$, i.e. $F$ is defined by 
\[F^{(1)}=x\,\text{ and }\,F^{(i)}=\displaystyle\lim_{t\to 0}(F^{(i-1)}+a^t\cdot x)\,\text{ for all }i\in\{1,\dots,n-1\}.\]
Then $a\cdot F$ is the osculating flag to $C$ at $a\cdot x$. By \cite[Theorem 1.3]{Fock-Goncharov}, $(a^-,F,a\cdot F,a^+)$ is a positive quadruple of flags.

Proof of (2). Since positivity is an open condition on $\Fc(V)^4$ and $\displaystyle\lim_{i\to\infty}a^i\cdot G=a_+$, there is an integer $N>0$ such that $(a_-,F,a^i\cdot G,H)$ is positive for all integers $i\geq N$. At the same time, since $a$ is positive loxodromic, $(a_-,a^i\cdot G,a_+,H)$ is positive for all integers $i$. Thus, $(a_-,F,a^i\cdot G,a_+,H)$ is positive for all $i\geq N$.
\end{proof}

The following lemma is the main geometric input needed to prove Theorem \ref{thm: general n}.

\begin{lemma}\label{lem: K existence}
Let $b\in\PGL(V)$ be positive loxodromic and $a\in\PGL(V)$ be loxodromic. 
\begin{enumerate}
\item If $(a_-,a^{-1}\cdot b_-,b_-,a\cdot b_-,a_+,b_+)$ is positive, then for every neighborhood $\mathfrak U\subset\Fc(V)$ of $b_-$, there is a flag $K\in \mathfrak U$ such that the tuples $(a_-,a^{-1}\cdot K,K,a\cdot K,a_+,b_+)$ and $(a_-,b_-,b^{-1}\cdot K,K,b\cdot K,b_+)$ are positive.
\item If $(a_-,b_-,a_+,a\cdot b_-,b_+,a^{-1}\cdot b_-)$ is positive, then for every neighborhood $\mathfrak U\subset\Fc(V)$ of $b_-$, there is a flag $K\in\Fc(V)$ such that the tuples $(a_-,K,a_+,a\cdot K,b_+,a^{-1}\cdot K)$ and $(a_-,b_-,b^{-1}\cdot K,K,b\cdot K,b_+)$ are positive.
\end{enumerate}
\end{lemma}

\begin{proof}
Let $K'$ be a flag such that $(a_-,b_-,b^{-1}\cdot K',K',b_+)$ is positive. This exists by Lemma \ref{lem: minor facts}(1). Fix a neighborhood $\mathfrak U$ of $b_-$. By Lemma \ref{lem: minor facts}(2), there is an integer $N>0$ such that $b^{-i-1}\cdot K',b^{-i}\cdot K',b^{-i+1}\cdot K'\in\mathfrak U$ and
\[(a_-,b_-,b^{-i-1}\cdot K',b^{-i}\cdot K',b^{-i+1}\cdot K',a_+,b_+)\] 
is positive for all $i\geq N$. 

Proof of (1). Since $(a_-,a^{-1}\cdot b_-,b_-,a\cdot b_-,a_+,b_+)$ is positive and $\displaystyle\lim_{i\to\infty}b^{-i}\cdot K'=b_-$, there is an integer $N'>0$ such that 
\[(a_-,a^{-1}b^{-i}\cdot K',b^{-i}\cdot K',ab^{-i}\cdot K',a_+,b_+)\] 
is positive for all $i\geq N'$. Let $k:=\max\{N,N'\}$, and set $K:=b^{-k}\cdot K'$.

Proof of (2). Since $(a_-,b_-,a_+,a\cdot b_-,b_+,a^{-1}\cdot b_-)$ is positive and $\displaystyle\lim_{i\to\infty}b^{-i}\cdot K'=b_-$, there is an integer $N'>0$ such that 
\[(a_-,b^{-i}\cdot K',a_+,ab^{-i}\cdot K',b_+,a^{-1}b^{-i}\cdot K')\] 
is positive for all $i\geq N'$. Let $k:=\max\{N,N'\}$, and set $K:=b^{-k}\cdot K'$.
\end{proof}

The next lemma deduces the conclusions of Theorem \ref{thm: general n} from the conclusions of Lemma \ref{lem: K existence}.

\begin{lemma}\label{lem: general n}
Let $b\in\PGL(V)$ be positive loxodromic and $a\in\PGL(V)$ be loxodromic. 
Let $f$ be a map from the vertex set $\{v, w, v', w'\}$ of $\mathsf{K}(v,w)\cup \mathsf{K}(v',w')$ into $\PGL(V)$ defined by 
$(f(v),f(w))=(a,b),  \ (f(v'),f(w'))=(a^{-1},b^{-1}).$
\begin{enumerate}
\item Suppose that there are flags $F,H\in\PGL(V)$ such that
\begin{itemize}
\item[(i)] $(a_-,a^{-1}\cdot F,F,a\cdot F,a_+,b_+)$ is positive,
\item[(ii)] $(a_-,b_-,b^{-1}\cdot F,F,b\cdot F,b_+)$ is positive,
\item[(iii)] $(a_+,a\cdot H,H,a^{-1}\cdot H,a_-,b_-)$ is positive, and
\item[(iv)] $(a_+,b_+,b\cdot H,H,b^{-1}\cdot H,b_-)$ is positive.
\end{itemize}
Then $f$ is $(\mathsf{K}(v,w)\cup \mathsf{K}(v',w'))$-admissible.
\item Suppose that there are flags $F,H\in\PGL(V)$ such that
\begin{itemize}
\item[(I)] $(a_-,a^{-1}\cdot H,F,a\cdot H,a_+,b_+)$ is positive,
\item[(II)] $(a_-,b_-,b^{-1}\cdot F,F,b\cdot F,b_+)$ is positive,
\item[(III)] $(a_+,a\cdot F,H,a^{-1}\cdot F,a_-,b_-)$ is positive, and
\item[(IV)] $(a_+,b_+,b\cdot H,H,b^{-1}\cdot H,b_-)$ is positive.
\end{itemize}
Then $f$ is $(\mathsf{K}(v,w)\cup \mathsf{K}(v',w'))$-admissible.
\end{enumerate}
\end{lemma}

\begin{proof}
We only give the proof of (1); the proof of (2) is very similar.

Proof of (1). Note that (i) and (ii) imply that $(a_-,b_-,F,a_+,b_+)$ is positive, and (iii) and (iv) imply that $(a_+,b_+,H,a_-,b_-)$ is positive. Thus, $(b_-,F,a_+,b_+,H,a_-)$ is positive, and in particular,
\begin{equation}\label{eqn: obvious1}
(F,a_+,H,a_-)
\end{equation}
and 
\begin{equation}\label{eqn: obvious2}
(b_-,F,b_+,H)
\end{equation}
are positive. Then (i), (iii), and \eqref{eqn: obvious1} imply that
\[(a^{-1}\cdot F,F,a\cdot F,a_+,a\cdot H,H,a^{-1}\cdot H,a_-)\]
is positive, and (ii), (iv), and \eqref{eqn: obvious2} imply that
\[(b_-,b^{-1}\cdot F,F,b\cdot F,b_+,b\cdot H,H,b^{-1}\cdot H)\]
is positive. From this, we see that the assignment $(F,H,a^-)$ to $v$, $(F,H,b^-)$ to $w$, $(F,H,a^+)$ to $v'$, and $(F,H,b^+)$ to $w'$ gives rise to a compatible system of forward domains for $(f,\mathsf{K}(v,w)\cup \mathsf{K}(v',w'))$ and thus $f$ is $(\mathsf{K}(v,w)\cup \mathsf{K}(v',w'))$-admissible.
\end{proof}

\begin{proof}[Proof of Theorem \ref{thm: general n}]
We prove the two cases of this theorem separately.

{\bf Case 1: $(b_-,a\cdot b_-,a_+,a\cdot b_+,b_+,a_-)$ is positive.} This implies that
\begin{align}\label{eqn: case 1}
(a^{-1}\cdot b_-,b_-,a\cdot b_-,a_+,a\cdot b_+,b_+,a^{-1}\cdot b_+,a_-),
\end{align}
and hence $(a^{-1}\cdot b_-,b_-,a\cdot b_-,a_+,b_+,a_-)$, is positive. Then Lemma \ref{lem: K existence}(1) implies that there is a flag $F\in\Fc(V)$ such that (i) and (ii) in the statement of Lemma \ref{lem: general n}(1) hold. Also, since $(b_-,a_+,a\cdot b_+,b_+,a^{-1}\cdot b_+,a_-)$ is positive, Lemma \ref{lem: K existence}(1) applied to $a^{-1}$ and $b^{-1}$ (in place of $a$ and $b$), ensures that there is a flag $H\in\Fc(V)$ such that (iii) and (iv) in the statement of Lemma \ref{lem: general n}(1) holds. By Lemma \ref{lem: general n}(1), the restriction map of $f$ to $\{v_1, w_1,v_1', w_1'\}$ is $(\mathsf{K}(v_1,w_1)\cup \mathsf{K}(v_1',w_1'))$-admissible. Finally, note that \eqref{eqn: case 1} implies $(a\cdot b_-,b_-,a^{-1}\cdot b_-,a_-,b_+,a_+)$ is also positive, so the same argument as above, with $a^{-1}$ in place of $a$, implies that the restriction map of $f$ to $\{v_2,w_2, v_2', w_2'\}$ is $(\mathsf{K}(v_2,w_2)\cup \mathsf{K}(v_2',w_2'))$-admissible. Therefore $f$ is $\Lambda_2$-admissible.


{\bf Case 2: $(b_-,a\cdot b_+,a_+,a\cdot b_-,b_+,a_-)$ is positive.} This implies that
\begin{align}\label{eqn: case 2}
(a^{-1}\cdot b_+,b_-,a\cdot b_+,a_+,a\cdot b_-,b_+,a^{-1}\cdot b_-,a_-),
\end{align}
and hence $(b_-,a_+,a\cdot b_-,b_+,a^{-1}\cdot b_-,a_-)$ is positive. Then Lemma \ref{lem: K existence}(2) implies that for any neighborhood $\mathfrak U\subset\Fc(V)$ of $b_-$, there is a flag $F\in\mathfrak U$ such that 
\begin{equation}\label{eqn: run1}
(a_-,F,a_+,a\cdot F,b_+,a^{-1}\cdot F)
\end{equation}
and
\begin{equation}\label{eqn: run2}
(a_-,b_-,b^{-1}\cdot F,F,b\cdot F,b_+)
\end{equation}
are positive. Also, note that \eqref{eqn: case 2} implies that $(a^{-1}\cdot b_+,b_-,a\cdot b_+,a_+,b_+,a_-)$ is positive, so Lemma \ref{lem: K existence}(2) applied to $a^{-1}$ and $b^{-1}$ (in place of $a$ and $b$) implies that for any neighborhood $\mathfrak V\subset\Fc(V)$ of $b_+$, there is a flag $H\in\mathfrak V$ such that 
\begin{equation}\label{eqn: run3}
(a_+,H,a_-,a^{-1}\cdot H,b_-,a\cdot H)
\end{equation} 
and 
\begin{equation}\label{eqn: run4}
(a_+,b_+,b\cdot H,H,b^{-1}\cdot H,b_-)
\end{equation}
are positive.

It follows from the positivity of (\ref{eqn: run1}) and (\ref{eqn: run2}) that $(a_-,b_-,F,a_+,b_+)$ is positive. Similarly, the positivity of (\ref{eqn: run3}) and (\ref{eqn: run4}) imply that $(a_+,b_+,H,a_-,b_-)$ is positive. Together, these imply that $(b_-,F,a_+,b_+,H,a_-)$ is positive. Hence, by choosing $\mathfrak U$ and $\mathfrak V$ to be sufficiently small, the positivity of (\ref{eqn: run1}) and (\ref{eqn: run3}) imply respectively that (III) and (I) in the statement of Lemma \ref{lem: general n}(2) holds. Note also that the positivity of (\ref{eqn: run2}) and (\ref{eqn: run4}) are exactly (II) and (IV) in the statement of Lemma \ref{lem: general n}(2). Thus, Lemma \ref{lem: general n}(2) implies that the restriction map of $f$ to $\{v_1,w_1,v_1', w_1'\}$ is $(\mathsf{K}(v_1,w_1)\cup \mathsf{K}(v_1', w_1'))$-admissible. Finally, \eqref{eqn: case 2} implies that $(b_-,a_-,a^{-1}\cdot b_-,b_+,a\cdot b_-,a_+)$ is also positive, so the same argument as above with $a^{-1}$ in place of $a$, proves that the restriction map of $f$ to $\{v_2, w_2,v_2', w_2'\}$ is $(\mathsf{K}(v_2,w_2)\cup \mathsf{K}(v_2',w_2'))$-admissible and hence $f$ is $\Lambda_2$-admissible.

Noting that $\rho\circ R_2$ is $\Lambda_2$-admissible, the second statement follows from the first and Proposition \ref{prop: forward primitive}.
\end{proof}

\subsection{Proof of Theorem \ref{thm: positive quadruple}}

Next, we prove Theorem \ref{thm: positive quadruple}. Observe that given the proof of Theorem \ref{thm: general n}, it is sufficient to prove the following lemma, which is a strengthening of Lemma \ref{lem: K existence}(1) in the case when $n=3$.

\begin{lemma}\label{lem: separator lemma}
Suppose that $n=3$. Let $a,b\in\PGL(V)$ be positive loxodromic elements such that $(b_-,a_+,b_+,a_-)$ is positive. Then there is a flag $K\in\Fc(V)$ such that the tuples $(a_-,a^{-1}\cdot K,K,a\cdot K,a_+,b_+)$ and $(a_-,b_-,b^{-1}\cdot K,K,b\cdot K,b_+)$ are positive.
\end{lemma}

Indeed, if we prove Lemma \ref{lem: separator lemma}, then the same proof used to prove Theorem \ref{thm: general n}, but with Lemma \ref{lem: separator lemma} used in place of Lemma \ref{lem: K existence}(1), will also prove Theorem \ref{thm: positive quadruple}.

The following is the key geometric lemma needed to prove Lemma \ref{lem: separator lemma}. Suppose that $n:=\dim(V)=3$. Let $g \in\PGL(V)$ be positive loxodromic. For any point $p\in\mathbb{P}(V)$ such that $p\notin g_+^{(i)}+g_-^{(2-i)}$ for $i=0,1,2$, the \emph{$g$-invariant osculating flag map through $p$} is the map 
\[\xi_{g,p}:\R\to\Fc(V)\] 
defined by $\xi_{g,p}^{(1)}(t):=g^t\cdot p$ and $\xi_{g,p}^{(2)}(t):=\displaystyle\lim_{(s,s')\to(t,t)}\xi_{g,p}^{(1)}(s)+\xi_{g,p}^{(1)}(s')$, where the limit is taken over all distinct pairs of real numbers $(s,s')$. By an explicit computation, one can verify that $\xi_{g,p}^{(1)}$ is differentiable map, so $\xi_{g,p}^{(2)}$ is well-defined.

\begin{lemma}\label{lem: basic existence}
Suppose that $n=3$. Let $g\in\PGL(V)$ be positive loxodromic.
\begin{enumerate}
\item If $p\notin g_+^{(i)}+g_-^{(2-i)}$ for $i=0,1,2$, then the quadruple $(g_-,\xi_{g,p}(s),\xi_{g,p}(t),g_+)$ is positive for all $s<t$. 
\item Let $\tau$ be a simplex associated to $\{g_+,g_-\}$, and let $L$ be a projective line through $g_+^{(1)}$ such that $L\cap\tau$ is non-empty. Then there is a sequence $(p_i)_{i=1}^\infty$ in $\tau$ and a sequence $(t_i)_{i=1}^\infty\in\R$ such that the sequence $(\xi_{g,p_i}(t_i))_{i=1}^\infty$ in $\Fc(V)$ converges to the flag $F$ defined by $F^{(1)}=g_+^{(1)}$ and $F^{(2)}=L$.
\item Let $\tau$ be a simplex associated to $\{g_+,g_-\}$, and let $q\in\mathbb{P}(V)$ be a point in the interior of $g_+^{(2)}\cap\overline{\tau}$. Then there is a sequence $(p_i)_{i=1}^\infty$ in $\tau$ and a sequence $(t_i)_{i=1}^\infty\in\R$ such that the sequence $(\xi_{g,p_i}(t_i))_{i=1}^\infty$ in $\Fc(V)$ converges to the flag $G$ defined by $G^{(1)}=q$ and $G^{(2)}=g_+^{(2)}$.
\end{enumerate}
\end{lemma}

\begin{proof}
Proof of (1). Choose a basis $(e_1,e_2,e_3)$ of $V$ such that $e_i\in g_+^{(i)}\cap g_-^{(4-i)}$ for $i=1,2,3$, and let $p=\begin{bmatrix}1\\1\\1\end{bmatrix}\in\mathbb{P}(V)$ (when written as a column vector in this basis). One can compute that 
\[\xi_{g,p}^{(1)}(t)=\begin{bmatrix}\frac{\lambda_1}{\lambda_3}(g^t)\\
\frac{\lambda_2}{\lambda_3}(g^t)\\
1
\end{bmatrix}\,\text{ and }\xi_{g,p}^{(2)}(t)=\left[1:-\frac{\log\frac{\lambda_1}{\lambda_3}(g)}{\log\frac{\lambda_2}{\lambda_3}(g)}\frac{\lambda_1}{\lambda_2}(g^t):\frac{\log\frac{\lambda_1}{\lambda_2}(g)}{\log\frac{\lambda_2}{\lambda_3}(g)}\frac{\lambda_1}{\lambda_3}(g^t)\right].\]
With this, the first statement follows from an easy computation (using Theorem \ref{thm: Fock-Goncharov}). 

Proof of (2). The assumptions on $L$ implies that $L$ is neither of the projective lines $\xi^{(2)}(g_+)$ and $\xi^{(1)}(g_+)+\xi^{(1)}(g_-)$. Thus, we may choose a basis $(e_1,e_2,e_3)$ of $V$ such that $e_i\in g_+^{(i)}\cap g_-^{(4-i)}$ for $i=1,2,3$, and $L=[0:-1:1]\in\mathbb{P}(V^*)$. By replacing $e_1$ with $-e_1$ if necessary, we may assume that
\[\tau=\left\{\begin{bmatrix}x\\y\\1\end{bmatrix}:x,y>0\right\}.\] 
Let $p_i:=\begin{bmatrix}1\\\frac{1}{i}\\1\end{bmatrix}\in\mathbb{P}(V)$ for any integer $i>0$, and observe that 
\[\xi_{g,p_i}^{(1)}(t)=\begin{bmatrix}\frac{\lambda_1}{\lambda_3}(g^t)\\
\frac{1}{i}\frac{\lambda_2}{\lambda_3}(g^t)\\
1
\end{bmatrix}\,\text{ and }\xi_{g,p_i}^{(2)}(t)=\left[1:-i\frac{\log\frac{\lambda_1}{\lambda_3}(g)}{\log\frac{\lambda_2}{\lambda_3}(g)}\frac{\lambda_1}{\lambda_2}(g^t):\frac{\log\frac{\lambda_1}{\lambda_2}(g)}{\log\frac{\lambda_2}{\lambda_3}(g)}\frac{\lambda_1}{\lambda_3}(g^t)\right].\]
Thus, if we let $t_i:=\frac{1}{\log\frac{\lambda_2}{\lambda_3}(g)}\log\left(i\frac{\log\frac{\lambda_1}{\lambda_3}(g)}{\log\frac{\lambda_1}{\lambda_2}(g)}\right)$ for all integers $i>0$, then 
\[\lim_{i\to\infty}\xi_{g,p_i}^{(1)}(t_i)=\lim_{i\to\infty}\begin{bmatrix}\frac{\lambda_1}{\lambda_2}(g^{t_i})\\
\frac{1}{i}\\
\frac{\lambda_3}{\lambda_2}(g^{t_i})
\end{bmatrix}=\begin{bmatrix}1\\
0\\
0
\end{bmatrix}.\]
because $\displaystyle\lim_{i\to\infty}t_i=\infty$. Also, by a straightforward computation, 
\[\frac{-i\frac{\log\frac{\lambda_1}{\lambda_3}(g)}{\log\frac{\lambda_2}{\lambda_3}(g)}\frac{\lambda_1}{\lambda_2}(g^{t_i})}{\frac{\log\frac{\lambda_1}{\lambda_2}(g)}{\log\frac{\lambda_2}{\lambda_3}(g)}\frac{\lambda_1}{\lambda_3}(g^{t_i})}=-1,\]
and
\[\lim_{i\to\infty}\frac{\log\frac{\lambda_1}{\lambda_2}(g)}{\log\frac{\lambda_2}{\lambda_3}(g)}\frac{\lambda_1}{\lambda_3}(g^{t_i})=\infty\]
because $\displaystyle\lim_{i\to\infty}t_i=\infty$. Thus, $\displaystyle\lim_{i\to\infty}\xi_{g,p_i}^{(2)}(t_i)=L$.

Proof of (3). The assumptions on $q$ implies that $q$ is neither of the points $\xi^{(1)}(g_+)$ and $\xi^{(2)}(g_+)\cap\xi^{(2)}(g_-)$. Thus, we may choose a basis $(e_1,e_2,e_3)$ of $V$ such that $e_i\in g_+^{(i)}\cap g_-^{(4-i)}$ for $i=1,2,3$, and $q=\begin{bmatrix}1\\1\\0\end{bmatrix}\in\mathbb{P}(V)$. By replacing $e_3$ with $-e_3$ if necessary, we may assume that
\[\tau=\left\{\begin{bmatrix}x\\y\\1\end{bmatrix}:x,y>0\right\}.\] 
Let $p_i:=\begin{bmatrix}1\\i\\1\end{bmatrix}\in\mathbb{P}(V)$ for any integer $i>0$. Then observe that 
\[\xi_{g,p_i}^{(1)}(t)=\begin{bmatrix}\frac{\lambda_1}{\lambda_3}(g^t)\\
i\frac{\lambda_2}{\lambda_3}(g^t)\\
1
\end{bmatrix}\,\text{ and }\xi_{g,p_i}^{(2)}(t)=\left[1:-\frac{1}{i}\frac{\log\frac{\lambda_1}{\lambda_3}(g)}{\log\frac{\lambda_2}{\lambda_3}(g)}\frac{\lambda_1}{\lambda_2}(g^t):\frac{\log\frac{\lambda_1}{\lambda_2}(g)}{\log\frac{\lambda_2}{\lambda_3}(g)}\frac{\lambda_1}{\lambda_3}(g^t)\right].\]
Thus, if we let $t_i:=\frac{\log i}{\log\frac{\lambda_1}{\lambda_2}(g)}$ for all integers $i>0$, then one computes that
\[\frac{\frac{\lambda_1}{\lambda_3}(g^{t_i})}{i\frac{\lambda_2}{\lambda_3}(g^{t_i})}=1.\]
Since $\displaystyle\lim_{i\to\infty}\log\frac{\lambda_1}{\lambda_3}(g^{t_i})=\infty$, this implies that $\displaystyle\lim_{i\to\infty}\xi^{(1)}_{g,p_i}(t_i)=q$.
At the same time,
\[\lim_{i\to\infty}\frac{-\frac{1}{i}\frac{\log\frac{\lambda_1}{\lambda_3}(g)}{\log\frac{\lambda_2}{\lambda_3}(g)}\frac{\lambda_1}{\lambda_2}(g^{t_i})}{\frac{\log\frac{\lambda_1}{\lambda_2}(g)}{\log\frac{\lambda_2}{\lambda_3}(g)}\frac{\lambda_1}{\lambda_3}(g^{t_i})}=-\frac{\log\frac{\lambda_1}{\lambda_3}(g)}{\log\frac{\lambda_1}{\lambda_2}(g)}\lim_{i\to\infty}\frac{1}{i}\frac{\lambda_3}{\lambda_2}(g^{t_i})=0\]
and
\[\lim_{i\to\infty}\frac{\log\frac{\lambda_1}{\lambda_2}(g)}{\log\frac{\lambda_2}{\lambda_3}(g)}\frac{\lambda_1}{\lambda_3}(g^{t_i})=\infty,\]
so $\displaystyle\lim_{i\to\infty}\xi_{g,p_i}^{(2)}(t_i)=[0:0:1]=g_+^{(2)}$.
\end{proof}

\begin{proof}[Proof of Lemma \ref{lem: separator lemma}]
To simplify notation, we will denote the point $a_+^{(1)}\in\mathbb{P}(V)$ simply by $x$. The proof proceeds in three different cases. 

{\bf Case 0: The image of $\xi_{b,x}^{(1)}$ is tangent to $a_+^{(2)}$.} In this case, $\xi_{b,x}(0)=a_+$, and there is some $t>0$ such that $\xi_{b,x}(-t)=b^{-1}\cdot a_+$ and $\xi_{b,x}(t)=b\cdot a_+$, see Figure \ref{fig: n=3}(i). Thus, Lemma \ref{lem: basic existence}(1) implies that $(b_-,b^{-1}\cdot a_+,a_+,b\cdot a_+,b_+,a_-)$ is positive. Also, Lemma \ref{lem: minor facts}(1) implies that there is a flag $K'\in\Fc(V)$ such that $(a_-,a^{-1}\cdot K',K',a\cdot K',a_+,b_+)$ is positive. Thus, 
\[(a_-,a^{-1}\cdot(a^i\cdot K'), (a^i\cdot K'),a\cdot (a^i\cdot K'),a_+,b_+)\] 
is positive for all integers $i$. Since positivity of a tuple of flags is an open condition and $\displaystyle\lim_{i\to\infty}a^i\cdot K'=a_+$,
\[(b_-,b^{-1}\cdot (a^i\cdot K'),a^i\cdot K',b\cdot (a^i\cdot K'),b_+,a_-)\] 
is also positive for sufficiently large integers $i$. Set $K:=a^i\cdot K'$.

\begin{figure}[h]
    \centering
    \includegraphics[width=0.9\textwidth]{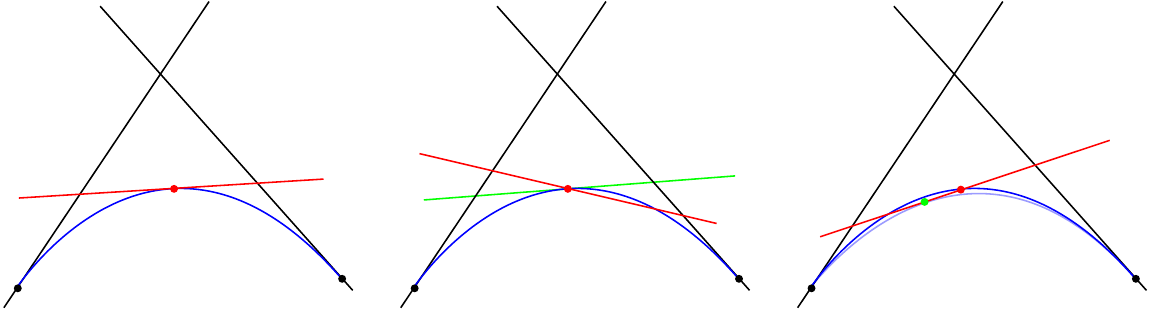}
    \tiny
\put (-372, 7){$b_-$}
\put (-267, 8){$b_+$}
\put (-316, 42){$a_+$}
\put (-244, 7){$b_-$}
\put (-139, 8){$b_+$}
\put (-188, 42){$a_+$}
\put (-150, 47){$\xi_{b,x}^{(2)}(0)$}
\put (-117, 7){$b_-$}
\put (-12, 8){$b_+$}
\put (-64, 42){$a_+$}
\put (-73, 30){$q$}
\small
\put (-320, -15){(i)}
\put (-193, -15){(ii)}
\put (-66, -15){(iii)}
    \caption{Proof of Lemma \ref{lem: separator lemma}.}
    \label{fig: n=3}
\end{figure}

{\bf Case 1: The image of $\xi_{b,x}^{(1)}$ intersects $a_+^{(2)}$ at some $t>0$.} Let $\tau$ be the simplex associated to $\{a_+,a_-\}$ that contains $b_-^{(1)}$. Since $\xi_{b,x}^{(1)}$ intersects $a_+^{(2)}$ at some $t>0$, observe that the line $\xi_{b,x}^{(2)}(0)$ passes through $\tau$, see Figure \ref{fig: n=3}(ii). Thus, Lemma \ref{lem: basic existence}(2) implies that there is a sequence $(p_i)_{i=1}^\infty$ in $\tau$ and a sequence $(t_i)_{i=1}^\infty\in\R$ such that 
\[\lim_{i\to\infty}\xi_{a,p_i}(t_i)=\xi_{b,x}(0).\]
By Lemma \ref{lem: basic existence}(1), $(b_-,b^{-1}\cdot \xi_{b,x}(0),\xi_{b,x}(0),b\cdot \xi_{b,x}(0),b_+,a_-)$ is positive, so 
\[(b_-,b^{-1}\cdot \xi_{a,p_i}(t_i),\xi_{a,p_i}(t_i),b\cdot \xi_{a,p_i}(t_i),b_+,a_-)\]
is positive for sufficiently large integers $i$. Also, Lemma \ref{lem: basic existence}(1) implies that
\[(a_-,a^{-1}\cdot \xi_{a,p_i}(t_i),\xi_{a,p_i}(t_i),a\cdot \xi_{a,p_i}(t_i),a_+,b_-)\]
is positive. Set $K:=\xi_{a,p_i}(t_i)$.

{\bf Case 2: The image of $\xi_{b,x}^{(1)}$ intersects $a_+^{(2)}$ at some $t<0$.} Let $\tau$ be the simplex associated to $\{a_+,a_-\}$ that contains $b_-$. Since the image of $\xi_{b,x}^{(1)}$ intersects $a_+^{(2)}$ at some $t<0$, observe that there is some point $q$ in the interior of $\overline{\tau}\cap a_+^{(2)}$ such that the image of $\xi^{(1)}_{b,q}(0)=q$ and $\xi^{(2)}_{b,q}(0)=a_+^{(2)}$ at $q$, see Figure \ref{fig: n=3}(iii). Thus, Lemma \ref{lem: basic existence}(3) implies that there is a sequence $(p_i)_{i=1}^\infty$ in $\tau$ and a sequence $(t_i)_{i=1}^\infty\in\R$ such that 
\[\lim_{i\to\infty}\xi_{a,p_i}(t_i)=\xi_{b,q}(0).\]
By Lemma \ref{lem: basic existence}(1), $(b_-,b^{-1}\cdot \xi_{b,q}(0),\xi_{b,q}(0),b\cdot \xi_{b,q}(0),b_+,a_-)$ is positive, so 
\[(b_-,b^{-1}\cdot \xi_{a,p_i}(t_i),\xi_{a,p_i}(t_i),b\cdot \xi_{a,p_i}(t_i),b_+,a_-)\]
is positive for sufficiently large integers $i$. Also, Lemma \ref{lem: basic existence}(1) implies that
\[(a_-,a^{-1}\cdot \xi_{a,p_i}(t_i),\xi_{a,p_i}(t_i),a\cdot \xi_{a,p_i}(t_i),a_+,b_-)\]
is positive. Set $K:=\xi_{a,p_i}(t_i)$.
\end{proof}

\subsection{Explicit examples of primitive stable representations}\label{sec: irred Fuch}

In this section, we use weak positivity to construct explicit examples of primitive stable representations. 

\subsubsection{Non-positive examples in all dimensions.} \label{sec: example1}
As our first example, we construct a non-positive, primitive stable representation from $F_2$ to $\PGL(V)$ whose image does not lie in $\iota(\PGL_2(\R))$ for any irreducible representation $\iota:\PGL_2(\R)\to\PGL(V)$. 

Let $U_n$ be the $n\times n$ upper triangular matrix whose entries are given by
\begin{equation}\label{eqn: Un}
(U_n)_{i,j}:=\left\{\begin{array}{ll}
0&\text{if }i>j;\\
{j-1 \choose i-1}&\text{if }i\leq j,\\
\end{array}\right.
\end{equation}
and let $W_n$ be the $n\times n$ upper triangular matrix whose entries are given by
\begin{equation}\label{eqn: Wn}(W_n)_{i,j}:=\left\{\begin{array}{ll}
0&\text{if }i>j;\\
(-1)^{j+i}{j-1 \choose i-1}&\text{if }i\leq j.\\
\end{array}\right.
\end{equation}
The matrix $U_n$ is usually called the \emph{$n$-th upper triangular Pascal matrix}.

Choose a basis $(e_1,\dots,e_n)$ of $V$. For all $t>1$, let $a_t\in\PGL(V)$ be represented by a diagonal matrix whose diagonal entries are $2^{\frac{n-1}{2}},2^{\frac{n-3}{2}},\dots,2^{\frac{3-n}{2}},\frac{2^{\frac{3-n}{2}}}{t}$. Then let $b\in\PGL(V)$ be the positive loxodromic element given by the following conditions:
\begin{itemize}
\item The eigenvalues of $b$ (up to scaling by a non-zero number) are $2^{\frac{n-1}{2}}$, $2^{\frac{n-3}{2}}$, $\dots$, $2^{\frac{3-n}{2}}$, $2^{\frac{1-n}{2}}$.
\item For all $i=1,\dots,n-1$, $b_+^{(i)}$ is spanned by the last $i$ columns of $U_n$. 
\item For all $i=1,\dots,n-1$, $b_-^{(i)}$ is spanned by the last $i$ columns of $W_n$. 
\end{itemize}
Here, the columns of $U_n$ and $W_n$ are viewed as vectors in $V$ via the chosen basis.

Let $\{\gamma_1,\gamma_2\}$ be a generating pair for $F_2$, and let $\rho_t:F_2\to\PGL(V)$ be the representation defined by $\rho_t(\gamma_1):=a_t$ and $\rho_t(\gamma_2):=b$. 

\begin{proposition}\label{prop: example 1}
\begin{enumerate}
\item For all $t>1$, $\rho_t$ is primitive stable.
\item For all $1<t<\frac{3}{2}$, the triple $(b_+, a_t\cdot b_+,(a_t)_+)$ is not positive. In particular, for any identification $\pi_1(\Sigma)\simeq F_2$, $\rho_t:\pi_1(\Sigma)\to\PGL(V)$ is not a positive representation.
\item If $t\neq 2$, then $\rho_t(F_2)$ does not lie in $\iota(\PGL_2(\R))$ for any irreducible representation $\iota:\PGL_2(\R)\to\PGL(V)$.
\end{enumerate}
\end{proposition}

\begin{proof}
Proof of (1). Recall that in Section \ref{sec: total positivity}, we defined, using the basis $\Bc:=(e_1,\dots,e_n)$, a linear representation $i=i_{\Bc}:\GL_2(\R)\to\GL(V)$. This projectivizes to an irreducible representation $\iota:\PGL_2(\R)\to\PGL(V)$. Let $\nu:\mathbb P(\R^2)\to\mathcal F(V)$ be the map given by 
\begin{align*}
\nu:\begin{bmatrix}1\\0\end{bmatrix}\mapsto&\,\,(a_t)_+,\\
\nu:\begin{bmatrix}x\\1\end{bmatrix}\mapsto&\,\,\iota\left(\begin{bmatrix}1&x\\0&1\end{bmatrix}\right)\cdot (a_t)_-.
\end{align*}
As observed in Example \ref{eq: Veronese}, $\nu$ is $\iota$-equivariant and positive. 

Let $b'\in\PGL(\R^2)$ be the element whose eigenvalues are $\sqrt{2}$ and $\frac{1}{\sqrt{2}}$, and whose attracting and repelling fixed points in $\mathbb P(\R^2)$ are $\begin{bmatrix}1\\1\end{bmatrix}$ and $\begin{bmatrix}-1\\1\end{bmatrix}$ respectively. Then $b'$ preserves both cyclic orderings on $\mathbb P(\R^2)$, so 
\[\left(\begin{bmatrix}-1\\1\end{bmatrix},\begin{bmatrix}1\\0\end{bmatrix},b'\cdot \begin{bmatrix}1\\0\end{bmatrix},\begin{bmatrix}1\\1\end{bmatrix},b'\cdot \begin{bmatrix}0\\1\end{bmatrix},\begin{bmatrix}0\\1\end{bmatrix}\right)\]
is positive. It is straightforward to check that 
\[\nu\left(\begin{bmatrix}-1\\1\end{bmatrix}\right)=b_-,\,\,\nu\left(\begin{bmatrix}1\\1\end{bmatrix}\right)=b_+,\,\,\nu\left(\begin{bmatrix}1\\0\end{bmatrix}\right)=(a_t)_+,\,\text{ and }\,\nu\left(\begin{bmatrix}0\\1\end{bmatrix}\right)=(a_t)_-.\]
Furthermore, $\iota(b')=b$. Thus, the fact that $\nu$ is positive and $\iota$-equivariant then implies that $(b_-,(a_t)_+,b\cdot (a_t)_+,b_+,b\cdot (a_t)_-,(a_t)_-)$ is positive. The fact that $\rho_t$ is primitive stable now follows from Theorem \ref{thm: general n}.

Proof of (2). A straightforward computation (in the basis $(e_1,\dots,e_n)$) yields
\[T_{(1,1,n-2)}(b_+,a_t\cdot b_+,(a_t)_+)=\frac{1}{2t-3}.\]
(The triple ratio $T_{\bf j}$ was defined in Section \ref{sec: triple ratio}.) Thus, $T_{(1,1,n-2)}(b_+,a_t\cdot b_+,(a_t)_+)<0$ for all $1<t<\frac{3}{2}$, so Theorem \ref{thm: Fock-Goncharov} implies that $(b_+,a_t\cdot b_+,(a_t)_+)$ is not a positive triple of flags. This implies that $\rho_t$ is not a positive representation for any identification $\pi_1(S)\simeq F_2$.

Proof of (3). From the definition of $\iota$ (see Section \ref{sec: total positivity}), one sees that if $g\in\iota(\PGL_2(\R))$, then $\frac{\lambda_k}{\lambda_{k+1}}(g)=\frac{\lambda_j}{\lambda_{j+1}}(g)$ for all $k,j=1,\dots,n-1$. But the eigenvalues of $a_t$ do not satisfy this condition unless $t=2$.
\end{proof}

\subsubsection{Examples that converge to the trivial representation.}\label{sec: example2}
Next, we construct a family $\rho_t:F_2\to\PGL(V)$ of non-positive, primitive stable representations that converges to the trivial representation, and whose images do not lie in $\iota(\PGL_2(\R))$ for any irreducible representation $\iota:\PGL_2(\R)\to\PGL(V)$. 

Choose a basis $(e_1,\dots,e_n)$ of $V$. For all $t>0$, let $a_t\in\PGL(V)$ be represented by a diagonal matrix whose diagonal entries are $2^{\frac{t(n-1)}{2}},2^{\frac{t(n-3)}{2}},\dots,2^{\frac{t(3-n)}{2}},2^{t(1-n)}$ down the diagonal. Then let $b_t\in\PGL(V)$ be the positive loxodromic element given by the following conditions:
\begin{itemize}
\item The eigenvalues of $b_t$ (up to scaling by a non-zero number) are $2^{\frac{t(n-1)}{2}}$, $2^{\frac{t(n-3)}{2}}$, $\dots$, $2^{\frac{t(3-n)}{2}}$, $2^{\frac{t(1-n)}{2}}$.
\item For all $i=1,\dots,n-1$, $(b_t)_+^{(i)}$ is spanned by the last $i$ columns of $U_n$. 
\item For all $i=1,\dots,n-1$, $(b_t)_-^{(i)}$ is spanned by the last $i$ columns of $W_n$. 
\end{itemize}
Here, $U_n$ and $W_n$ are the $n\times n$ matrices given by (\ref{eqn: Un}) and (\ref{eqn: Wn}) respectively. Let $\{\gamma_1,\gamma_2\}$ be a generating pair for $F_2$, and let $\rho_t:F_2\to\PGL(V)$ be the representation defined by $\rho_t(\gamma_1):=a_t$ and $\rho_t(\gamma_2):=b_t$. 

\begin{proposition} 
\begin{enumerate}
\item For all $t>0$, $\rho_t$ is primitive stable. 
\item The family $\rho_t$ converges to the trivial representation as $t\to 0$.
\item For all $t>0$, $\rho_t$ does not lie in $\iota(\PGL_2(\R))$ for any irreducible representation $\iota:\PGL_2(\R)\to\PGL(V)$.
\end{enumerate}
\end{proposition}

\begin{proof}
The same arguments used to prove (1) and (3) of Proposition \ref{prop: example 1} also prove (1) and (3) respectively. (2) is obvious from the definition of $\rho_t$.
\end{proof}

\subsubsection{Non-discrete and non-faithful examples when $n=3$} \label{sec: example3}
Choose a basis $(e_1,e_2,e_3)$ of $V$ to identify $V\simeq \R^3$. For any real number $t$, let $a,b_t\in\PGL_3(\mathbb{R})$ be projective transformations given by

\[ a:=\begin{bmatrix} 1/2 & 0& 0 \\ 0&1&0\\0&0&2 \end{bmatrix}, \ b_t:=\begin{bmatrix} 2t+5 & -4t+2& 2t-3 \\ -2t+1&4t+2&-2t+1\\2t-3&-4t+2&2t+5 \end{bmatrix}.
\]
Let $\{\gamma_1,\gamma_2\}$ be a pair of generators of $F_2$ 
and $\rho_t:F_2\to\PGL(\mathbb{R}^3)$ be the representation defined by $\rho_t(\gamma_1)=a$ and $\rho_t(\gamma_2)=b_t$.

\begin{proposition}\label{prop:example1}
\begin{enumerate}
\item If $t>1$, then $\rho_t$ is primitive stable.
\item If $t<\frac{35}{2}$ and $\cos^{-1}\left(\frac{-35+306t-32t^2}{256t}\right)\in \mathbb Q\cdot \pi$, then $\rho_t$ is non-faithful,
\item If $t<\frac{35}{2}$ and $\cos^{-1}\left(\frac{-35+306t-32t^2}{256t}\right) \notin \mathbb Q\cdot \pi$, then $\rho_t$ is non-discrete,
\item If $t\neq 2$, then $\rho_t(F_2)$ does not lie in a conjugate of $\mathrm{PO}(2,1)$.
\end{enumerate}
\end{proposition}

\begin{proof}
Proof of (1). Let $\lambda_1(a)\geq\lambda_2(a)\geq\lambda_3(a)$ (resp. $\lambda_1(b_t)\geq\lambda_2(b_t)\geq\lambda_3(b_t)$) denote the eigenvalues of $a$ (resp. $b_t$). It is easy to calculate that $\lambda_1(a)=2$, $\lambda_2(a)=1$, $\lambda_3(a)=\frac{1}{2}$, and $\lambda_1(b_t)=t$, $\lambda_2(b_t)=1$, $\lambda_3(b_t)=\frac{1}{2}$. It follows that $a$ is loxodromic, and $b_t$ is positive loxodromic when $t>1$. Also, one can calculate that
\begin{itemize}
\item $a_+^{(1)}=\begin{bmatrix}0\\0\\1\end{bmatrix}$, $a_+^{(2)}=\begin{bmatrix}1:0:0\end{bmatrix}$,
\item $a_-^{(1)}=\begin{bmatrix}1\\0\\0\end{bmatrix}$, $a_-^{(2)}=\begin{bmatrix}0:0:1\end{bmatrix}$,
\end{itemize}
and that when $t>1$, 
\begin{itemize}
\item $(b_t)_+^{(1)}=\begin{bmatrix}1\\-1\\1\end{bmatrix}$, $(b_t)_+^{(2)}=\begin{bmatrix}1:2:1\end{bmatrix}$,
\item $(b_t)_-^{(1)}=\begin{bmatrix}1\\1\\1\end{bmatrix}$, $(b_t)_-^{(2)}=\begin{bmatrix}1:-2:1\end{bmatrix}$,
\end{itemize}

From this, it is a straightforward calculation (using Theorem \ref{thm: Fock-Goncharov}) to verify that when $t>1$, 
$\big((b_t)_-,a_+,(b_t)_+,a_-\big)$ is positive. Thus, Theorem \ref{thm: general n} implies that $
\rho_t$ is primitive stable.

Proof of (2) and (3). Consider the commutator $[a,b_t]:=ab_t a^{-1}b_t^{-1}$ of $a$ and $b_t$. An explicit computation gives that the characteristic polynomial $P_t(x)$ of $[a,b_t]$ is 
\[ P_t(x)=(1-x)\left(x^2+\frac{35-306t+32t^2}{128t}x+1\right)=:(1-x)Q_t(x)\]
The discriminant of the polynomial $Q_t(x)$ is 
\[\frac{1225 - 21420 t + 30340 t^2 - 19584 t^3 + 1024 t^4}{16384 t^2},\]
which is negative if and only if $\frac{1}{16}<t<\frac{35}{2}$. Thus, when $1<t<\frac{35}{2}$, the commutator $[a,b_t]$ is conjugate to the projective matrix
\[ \begin{bmatrix} 1&0&0\\0& \cos\theta &-\sin\theta\\0&\sin\theta &\cos \theta\end{bmatrix}, \text{ where }\ \theta:=\cos^{-1}\left(\frac{-35+306t-32t^2}{256t}\right).\]
Obviously, if $\theta$ is rational, then $\rho_t$ is a non-faithful representation, and if $\theta$ is irrational, then $\rho_t$ is a non-discrete representation. 

Proof of (3). Every loxodromic element $g\in\PO(2,1)$ has the property that $\lambda_1(g)=\frac{1}{\lambda_3(g)}$. Since $b_t$ does not have this property when $t\neq 2$, (3) follows.
\end{proof}




\appendix

\section{Converse to Proposition \ref{prop: forward primitive}}\label{sec: converse}
The goal of this appendix is to prove a converse to Proposition \ref{prop: forward primitive}. More precisely:

\begin{theorem}\label{thm: n=2}
Suppose that $n=2$. If $\rho:F_2\to\PGL(V)$ is primitive stable, then there is a map $R_3$ from the vertex set $A_3$ of $\Lambda_3$ into $\PGL(V)$ such that $\rho$ is $(R_3,\Lambda_3)$-weakly positive.
\end{theorem}

Recall that if $\pi_1(\Sigma)\simeq F_2$, then $\Sigma$ is either the one-holed torus $\Sigma_{1,1}$, the one-holed Klein bottle $C_{1,1}$, the one-holed M\"obius band $C_{0,2}$, or the three-holed sphere $\Sigma_{0,3}$. For any such $\Sigma$, the Gromov boundary $\partial_\infty\pi_1(\Sigma)$ of $\pi_1(\Sigma)$ admits two natural cyclic orders which are reverses of each other. The following observation lists some well-known properties of the fundamental group of these surfaces. 

\begin{observation}\label{obs: surface}
Recall that for any non-identity element $\gamma\in F_2$, the attracting and repelling fixed point of $\gamma$ in $\partial_\infty F_2$ are denoted by $\gamma_+$ and $\gamma_-$ respectively.
\begin{enumerate}
\item For any pair of generators $\{\gamma_1,\gamma_2\}$ of $\pi_1(\Sigma_{1,1})$, 
\[(\gamma_1)_-<(\gamma_2)_+<(\gamma_1)_+<(\gamma_2)_-<(\gamma_1)_-\]
in one of the two cyclic orders on $\partial_\infty\pi_1(\Sigma_{1,1})$, see Figure \ref{fig: order}(i).
\item There is a pair of generators $\{\gamma_1,\gamma_2\}$ for $\pi_1(C_{1,1})$ such that if we denote $\gamma_3:=\gamma_2^{-1}\gamma_1^{-1}$ and $\gamma_3':=\gamma_1^{-1}\gamma_2^{-1}$, then
\begin{itemize}
\item each of the conjugacy classes $[\gamma_1]$, $[\gamma_2]$, and $[\gamma_3]=[\gamma_3']$ corresponds to the free homotopy class of a simple, oriented closed curve in $C_{1,1}$, 
\item $\gamma_3$ and $\gamma_3'$ preserve both cyclic orders on $\partial_\infty\pi_1(C_{1,1})$, while $\gamma_1$ and $\gamma_2$ switch them,
\item $(\gamma_1)_-<(\gamma_3')_+<(\gamma_3)_-<(\gamma_1)_+<(\gamma_2)_-<(\gamma_3)_+<(\gamma_3')_-<(\gamma_2)_+<(\gamma_1)_-$ in one of the two cyclic orders on $\partial_\infty\pi_1(C_{1,1})$, see Figure \ref{fig: order}(ii).
\end{itemize} 
\item There is a pair of generators $\{\gamma_1,\gamma_2\}$ for $\pi_1(C_{0,2})$ such that if we denote $\gamma_1':=\gamma_2^{-1}\gamma_1\gamma_2$, $\gamma_1'':=\gamma_2\gamma_1\gamma_2^{-1}$, $\gamma_2':=\gamma_1^{-1}\gamma_2\gamma_1$, $\gamma_2'':=\gamma_1\gamma_2\gamma_1^{-1}$, $\gamma_3:=\gamma_2^{-1}\gamma_1^{-1}$, $\gamma_3':=\gamma_1^{-1}\gamma_2^{-1}$, $\gamma_4:=\gamma_2\gamma_1^{-1}$, and $\gamma_4':=\gamma_1^{-1}\gamma_2$, then
\begin{itemize}
\item each of the conjugacy classes $[\gamma_1]=[\gamma_1']=[\gamma_1'']$, $[\gamma_2]=[\gamma_2']=[\gamma_2'']$, $[\gamma_3]=[\gamma_3']$, and $[\gamma_4]=[\gamma_4']$ corresponds to the free homotopy class of a simple, oriented closed curves in $C_{0,2}$, 
\item $\gamma_3$, $\gamma_3'$, $\gamma_4$, and $\gamma_4'$ preserve both cyclic orders on $\partial\pi_1(C_{0,2})$, while $\gamma_1$, $\gamma_1'$, $\gamma_1''$, $\gamma_2$, $\gamma_2'$, and $\gamma_2''$ switch them,
\item $(\gamma_2')_+<(\gamma_1)_-<(\gamma_2')_-<(\gamma_3')_+<(\gamma_3')_-<(\gamma_1'')_+<(\gamma_2)_+<(\gamma_1'')_-<(\gamma_4)_+<(\gamma_4)_-<(\gamma_2'')_-<(\gamma_1)_+<(\gamma_2'')_+<(\gamma_3)_-<(\gamma_3)_+<(\gamma_1')_-<(\gamma_2)_-<(\gamma_1')_+<(\gamma_4')_-<(\gamma_4')_+<(\gamma_2')_+$ in one of the two cyclic orders on $\partial\pi_1(C_{0,2})$, see Figure \ref{fig: order}(iii).
\end{itemize} 
\item There is a superbasis $\{\gamma_1,\gamma_2,\gamma_3\}$ for $\pi_1(\Sigma_{0,3})$ such that
\begin{itemize}
\item each of the conjugacy classes $[\gamma_1]$, $[\gamma_2]$, and $[\gamma_3]$ corresponds to the free homotopy class of a simple, oriented closed curve in $\Sigma_{0,3}$,
\item $(\gamma_1)_-<(\gamma_1)_+<(\gamma_3)_-<(\gamma_3)_+<(\gamma_2)_-<(\gamma_2)_+<(\gamma_1)_-$
 in one of the two cyclic orders on $\partial\pi_1(\Sigma_{0,3})$, see Figure \ref{fig: order}(iv).
\end{itemize}
\end{enumerate}
\end{observation}

\begin{figure}[h]
    \centering
    \includegraphics[width=0.8\textwidth]{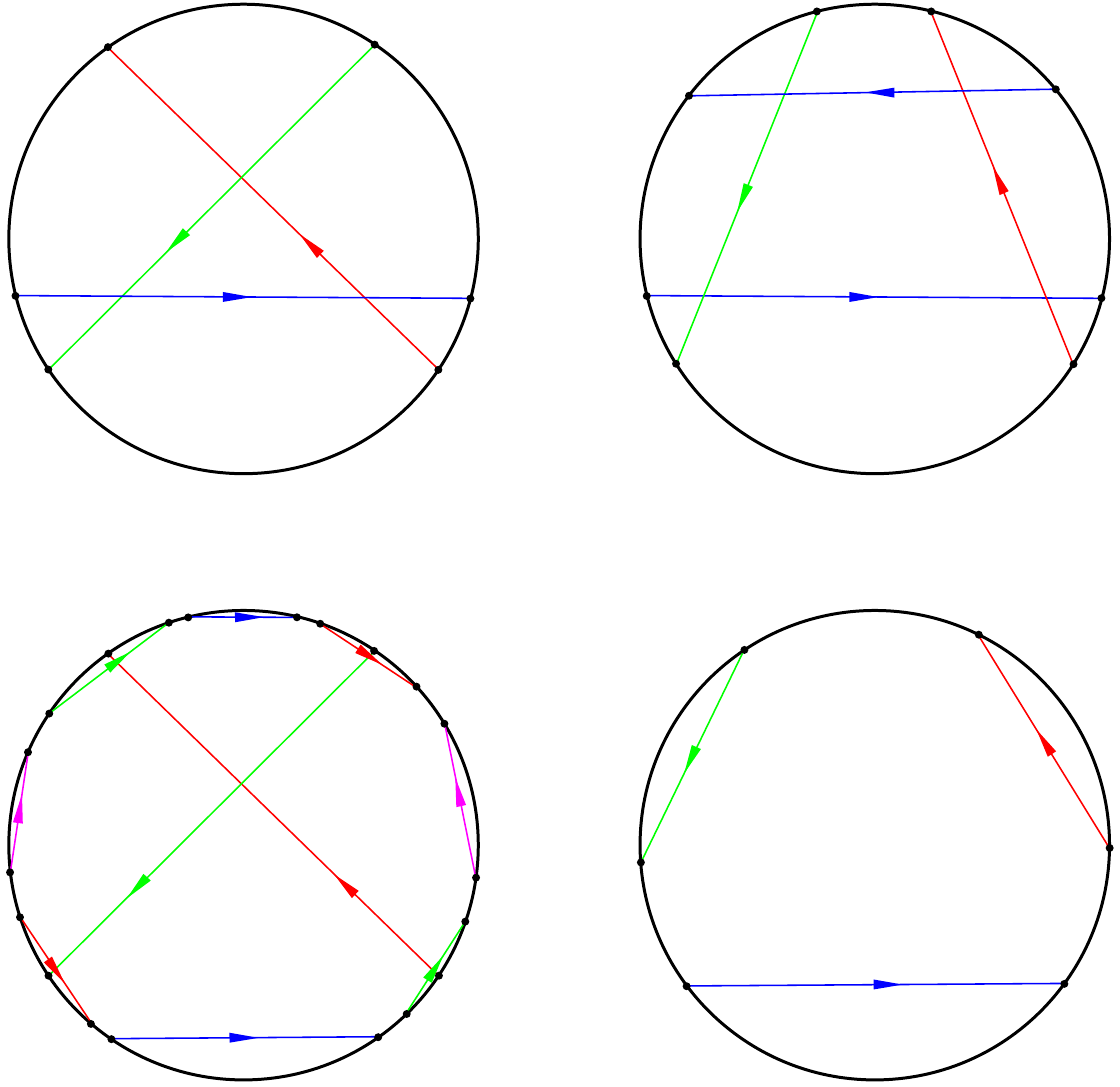}
    \tiny
\put (-218, 307){$(\gamma_1)_-$}
\put (-315, 307){$(\gamma_2)_+$}
\put (-330, 206){$(\gamma_1)_+$}
\put (-200, 206){$(\gamma_2)_-$}
\put (-104, 318){$(\gamma_1)_-$}
\put (-144, 293){$(\gamma_3')_+$}
\put (-157, 230){$(\gamma_3)_-$}
\put (-147, 209){$(\gamma_1)_+$}
\put (-12, 208){$(\gamma_2)_-$}
\put (-4, 230){$(\gamma_3)_+$}
\put (-17, 293){$(\gamma_3')_-$}
\put (-55, 318){$(\gamma_2)_+$}
\put (-218, 130){$(\gamma_1)_-$}
\put (-251, 142){$(\gamma_3')_+$}
\put (-278, 142){$(\gamma_3')_-$}
\put (-314, 130){$(\gamma_2)_+$}
\put (-339, 100){$(\gamma_4)_+$}
\put (-344, 60){$(\gamma_4)_-$}
\put (-331, 29){$(\gamma_1)_+$}
\put (-310, 10){$(\gamma_3)_-$}
\put (-217, 10){$(\gamma_3)_+$}
\put (-198, 30){$(\gamma_2)_-$}
\put (-187, 60){$(\gamma_4')_-$}
\put (-196, 106){$(\gamma_4')_+$}
\put (-126, 131){$(\gamma_1)_-$}
\put (-159, 63){$(\gamma_1)_+$}
\put (-144, 27){$(\gamma_3)_-$}
\put (-14, 28){$(\gamma_3)_+$}
\put (-2, 66){$(\gamma_4)_-$}
\put (-41, 135){$(\gamma_4)_+$}
\small
\put (-265, 165){$(i)$}
\put (-80, 165){$(ii)$}
\put (-267, -13){$(iii)$}
\put (-78, -13){$(iv)$}
    \caption{Cyclic orders of points along $\partial_\infty\pi_1(\Sigma)$ when (i) $\Sigma=\Sigma_{1,1}$, (ii) $\Sigma=C_{1,1}$, (iii) $\Sigma=C_{0,2}$, (iv) $\Sigma=\Sigma_{0,3}$.}
    \label{fig: order}
\end{figure}
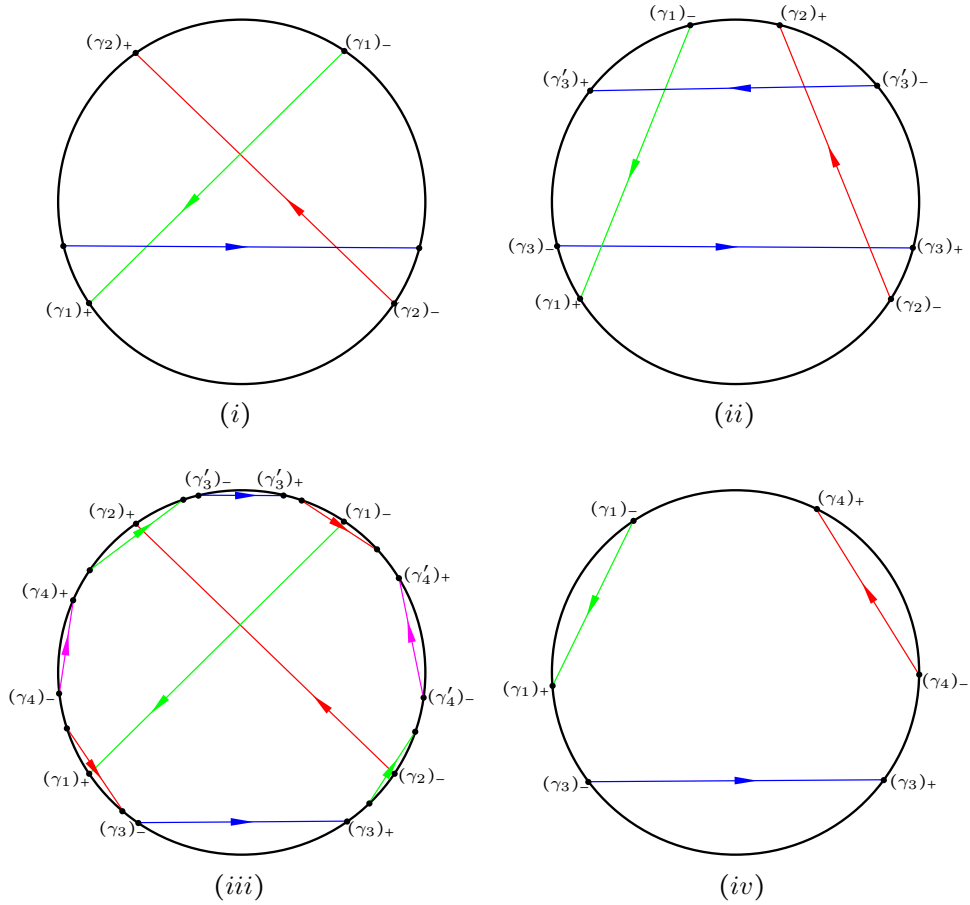

When $n=2$, choose a basis of $V$ to identify $V\simeq\R^2$. The $\PGL_2(\R)$-Riemannian symmetric space is the hyperbolic plane $\H^2:=\{z\in\C:\mathrm{Im}(z)>0\}$, and the standard inner product on $\R^2$ corresponds to the base point $i\in\H^2$. Also, $\Fc(\R^2)=\mathbb{P}(\R^2)$, which one can identify as the visual boundary of $\H^2$. Furthermore, there is a unique simple root $\alpha$ of $\PGL_2(\R)$, and for any $g\in\PGL_2(\R)$, $\alpha(g)$ is simply the distance in $\H^2$ between $i$ and $g\cdot i$. 

Observe that there is a well-defined trace map
\[\tr:\{[g,h]:=ghg^{-1}h^{-1}:g,h\in\PGL_2(\R)\}\to\R\] 
given by $\tr([g,h]):=\tr(\bar{g}\bar{h}\bar{g}^{-1}\bar{h}^{-1})$ for any (equiv. some) representatives $\bar{g},\bar{h}\in\GL_2(\R)$ of $g$ and $h$ respectively. This defines a function 
\[\mathbf k:\Hom(F_2,\PGL_2(\R))\to\R\] 
by $\mathbf k:\rho\mapsto\tr([\rho(\gamma_1),\rho(\gamma_2)])$, where $\{\gamma_1,\gamma_2\}$ is any pair of generators of $F_2$. It is a consequence of a result by Nielsen \cite{Nielsen} that ${\bf k}$ is well-defined (see \cite[Theorem 3.9]{MaKaSo} or \cite[Proposition 5.1]{LS} for a proof). One can verify by direct calculation that if $\mathbf k(\rho)=2$, then $\rho$ is reducible, and hence not primitive stable (see \cite[Theorem 1.4]{TanWongZhang} for a more general result). 

The following theorem is a summary of results due to Goldman \cite{Go} and Goldman-McShane-Stantchev-Tan \cite{GMST}. Informally, it says that primitive stable representations from $F_2$ to $\PGL_2(\mathbb R)$ arise as holonomies of (possibly singular) hyperbolic structures on the four surfaces whose fundamental group is $F_2$, namely the one-holed torus $\Sigma_{1,1}$, the three-holed sphere $\Sigma_{0,3}$, the one-holed Klein bottle $C_{1,1}$, and the one-holed M\"obius band $C_{0,2}$.


\begin{theorem}\label{thm: Goldman}
Let $\rho:F_2\to\PGL_2(\R)$ be a primitive stable representation.
\begin{enumerate}
\item \,\cite[Section 3]{Go} Suppose that $\rho(F_2)\subset\PSL_2(\R)$ and $\mathbf k(\rho)<2$. Then there is an identification of $F_2\simeq\pi_1(\Sigma_{1,1})$ such that $\rho$ is the holonomy of a hyperbolic structure on $\Sigma_{1,1}$ (possibly with a cone point at the hole). Furthermore, for any pair of generators $\{\gamma_1,\gamma_2\}$ of $\pi_1(\Sigma_{1,1})$, $g_i:=\rho(\gamma_i)$ is loxodromic for $i=1,2$, and
\begin{equation}\label{eqn: Case 1}
\big((g_1)_-,(g_2)_+,(g_1)_+,(g_2)_-\big)
\end{equation}
is a positive tuple in $\mathbb{P}(\R^2)$.
\item \,\cite[Proposition 5.2, Section 8]{GMST} Suppose that $\rho(F_2)\not\subset\PSL_2(\R)$ and $\mathbf k(\rho)>2$. Then there is an identification $F_2\simeq\pi_1(C_{1,1})$ such that $\rho$ is the holonomy of a hyperbolic structure on $C_{1,1}$ (possibly with a cone point at the hole). Furthermore, if $\gamma_1,\gamma_2,\gamma_3,\gamma_3'\in\pi_1(C_{1,1})$ are as given by Observation \ref{obs: surface}(2), set $g_i:=\rho(\gamma_i)$ for $i=1,2,3$, and set $g_3':=\rho(\gamma_3')$. Then $g_3,g_3'\in\PSL_2(\R)$ and $g_1,g_2\notin\PSL_2(\R)$ are loxodromic, and
\begin{equation}\label{eqn: Case 2}
\big((g_1)_-,(g_3')_+,(g_3)_-,(g_1)_+,(g_2)_- ,(g_3)_+,(g_3')_-,(g_2)_+\big)
\end{equation}
is a positive tuple in $\mathbb{P}(\R^2)$.
\item \,\cite[Proposition 5.1, Section 9]{GMST} Suppose that $\rho(F_2)\not\subset\PSL_2(\R)$ and $\mathbf k(\rho)<2$. Then there is an identification $F_2\simeq\pi_1(C_{0,2})$ such that $\rho$ is the holonomy of a convex cocompact hyperbolic structure on $C_{0,2}$. Furthermore, if 
\[\gamma_1,\gamma_2,\gamma_3,\gamma_4,\gamma_1',\gamma_2',\gamma_3',\gamma_4',\gamma_1'',\gamma_2''\in\pi_1(C_{0,2})\] 
are as given by Observation \ref{obs: surface}(3), set $g_i:=\rho(\gamma_i)$ for $i=1,\dots,4$, set $g_i':=\rho(\gamma_i')$ for $i=1,\dots,4$, and set $g_i'':=\rho(\gamma_i'')$ for $i=1,2$. Then $g_3,g_3',g_4,g_4'\in\PSL_2(\R)$ and $g_1,g_1',g_1'',g_2,g_2',g_2''\notin\PSL_2(\R)$ are loxodromic, and
\begin{align}\label{eqn: Case 3}
&\big((g_2')_+,(g_1)_-,(g_2')_-,(g_3')_+,(g_3')_-,(g_1'')_+,(g_2)_+,(g_1'')_-,(g_4)_+,(g_4)_-,\\
&\hspace{2cm}(g_2'')_-,(g_1)_+,(g_2'')_+,(g_3)_-,(g_3)_+,(g_1')_-,(g_2)_-,(g_1')_+,(g_4')_-,(g_4')_+\big)\nonumber
\end{align}
is a positive tuple in $\mathbb{P}(\R^2)$.
\item \,\cite[Theorem 5.2.1]{Go} Suppose that $\rho(F_2)\subset\PSL_2(\R)$ and $\mathbf k(\rho)>2$. Then there is an identification $F_2\simeq\pi_1(\Sigma_{0,3})$ such that $\rho$ is the holonomy of a convex cocompact hyperbolic structure on $\Sigma_{0,3}$. Furthermore, if $\{\gamma_1,\gamma_2,\gamma_3\}$ is the superbasis for $\pi_1(\Sigma_{0,3})$ given by Observation \ref{obs: surface}(4), then $g_i:=\rho(\gamma_i)$ is loxodromic for $i=1,2,3$, and
\begin{equation}\label{eqn: Case 4}
\big((g_1)_-,(g_1)_+,(g_3)_-,(g_3)_+,(g_2)_-,(g_2)_+\big)
\end{equation}
is a positive tuple in $\mathbb{P}(\R^2)$.
\end{enumerate}
\end{theorem}

\begin{remark}
The results of Goldman \cite{Go} and Goldman-McShane-Stantchev-Tan \cite{GMST} were stated for representations that satisfy the Bowditch Q-conditions. However, it is easily seen from the definitions that primitive stable representations satisfy Bowditch Q-conditions. Our proof of Theorem \ref{thm: n=2}, together with Proposition~\ref{prop: forward primitive}, in fact imply that representations from $F_2$ to $\PGL_2(\R)$ that satisfy the Bowditch Q-conditions are primitive stable. This was previously proven by Lupi \cite{Lu}, and was later generalized to representations from $F_2$ to $\PSL_2(\mathbb C)$ independently by Lee-Xu \cite{LX} and Series \cite{Series}.
\end{remark}

\begin{proof}[Proof of Theorem \ref{thm: n=2}]
By Theorem \ref{thm: Goldman}, $\rho$ arises as the holonomy representation of a hyperbolic structure on one of the four surfaces whose fundamental group is isomorphic to $F_2$. We consider the four cases separately. For simplicity, let $\Lambda_3^i=\mathsf{K}(v_i,w_i)\cup \mathsf{K}(v_i',w_i')$ and $R_3^i$ be the restriction map of $R_3$ to ${\{v_i,w_i,v_i',w_i'\}}$ for $i=1,2,3$. Note that $\Lambda_3$ is the disjoint union of $\Lambda_3^1$, $\Lambda_3^2$ and $\Lambda_3^3$, that is, $\Lambda_3=\Lambda_3^1\cup \Lambda_3^2\cup \Lambda_3^3$, and moreover $R_3=R_3^1\cup R_3^2\cup R_3^3$. Hence $\rho$ is $(R_3,\Lambda_3)$-weakly positive if and only if $\rho$ is $(R_3^i,\Lambda_3^i)$-weakly positive for all $i=1,2,3$.

{\bf Case 1: $\rho(F_2)\subset\PSL_2(\R)$ and $\mathbf k(\rho)<2$.} Choose any superbasis $\{\gamma_1,\gamma_2,\gamma_3\}$ of $F_2$. It is sufficient to show that $\rho$ is $(R_3^i, \Lambda_3^i)$-weakly positive for $i=1,2,3$. Let $g_i:=\rho(\gamma_i)$ for $i=1,2$. By the positivity of (\ref{eqn: Case 1}), there exists points $F,H\in\mathbb{P}(\R^2)$ such that
\[\big((g_1)_-,(g_2)_+,F,(g_1)_+,(g_2)_-,H\big)\]
is a positive tuple of flags in $\mathbb{P}(\R^2)$, see Figure \ref{fig: fourcase1}. It is straightforward to verify that $\rho\circ R_3^1$ is $\Lambda_3^1$-admissible by assigning $(F,H,(g_2)_-)$ to $v_1$ and $w_1$, and $(F,H,(g_1)_-)$ to $v_1'$ and $w_1'$, so $\rho$ is $(R_3^1,\Lambda_3^1)$-weakly positive.
Replacing $(\gamma_1,\gamma_2)$ with $(\gamma_2,\gamma_3)$ (resp. $(\gamma_3,\gamma_1)$) proves that $\rho$ is $(R_3^2,\Lambda_3^2)$-weakly positive (resp. $(R_3^3,\Lambda_3^3)$-weakly positive).

\begin{figure}[h]
    \centering
    \includegraphics[width=0.4\textwidth]{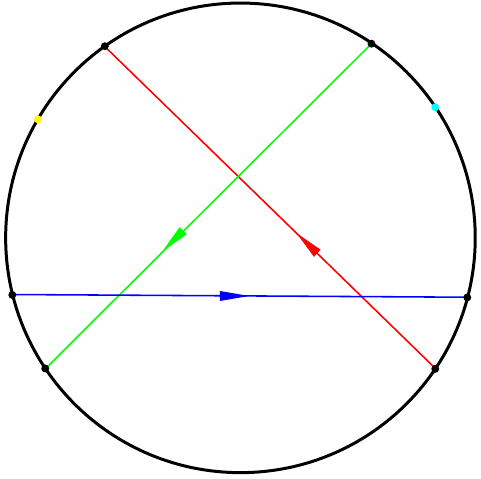}
    \tiny
\put (-40, 152){$(g_1)_-$}
\put (-145, 151){$(g_2)_+$}
\put (-157, 123){$F$}
\put (-164, 33){$(g_1)_+$}
\put (-15, 33){$(g_2)_-$}
\put (-13, 125){$H$}
    \caption{Case 1.}
    \label{fig: fourcase1}
\end{figure}

{\bf Case 2: $\rho(F_2)\not\subset\PSL_2(\R)$ and $\mathbf k(\rho)>2$.} Let $\{\gamma_1,\gamma_2\}$ be the pair of generators given by Observation \ref{obs: surface}(2). Let $g_i:=\rho(\gamma_i)$ for $i=1,2,3$, and let $g_3':=\rho(\gamma_3')$. By the positivity of (\ref{eqn: Case 2}), it is straightforward to verify that $\rho\circ R_3^1$ is $\Lambda_3^1$-admissible by assigning $((g_3')_+,(g_3')_-,(g_3)_+)$ to $v_1$ and $w_1$, and $((g_3)_+,(g_3)_-,(g_3')_+)$ to $v_1'$ and $w_1'$, so $\rho$ is $(R_3^1,\Lambda_3^1)$-weakly positive. It now suffices to show that $\rho$ is $(R_3^3,\Lambda_3^3)$-weakly positive; replacing $(g_1,g_3)$ with $(g_2^{-1},g_3^{-1})$ proves that $\rho$ is $(R_3^2,\Lambda_3^2)$-weakly positive. By the positivity of (\ref{eqn: Case 2}), there are points $K,F\in\mathbb{P}(\R^2)$ such that
\[\big((g_1)_-,(g_3')_+,K,F,(g_3)_-\big).\]
is a positive tuple in $\mathbb{P}(\R^2)$, see Figure \ref{fig: fourcase2}. Set $H:=g_1\cdot K$. It is straightforward to verify that $\rho \circ R_3^3$ is $\Lambda_3^3$-admissible by assigning $(F,H,(g_2)_-)$ to $v_3$ and $w_3$, and $(F,H,(g_1)_+)$ to $v_3'$ and $w_3'$, so $\rho$ is $(R_3^3,\Lambda_3^3)$-weakly positive.

\begin{figure}[h]
    \centering
    \includegraphics[width=0.4\textwidth]{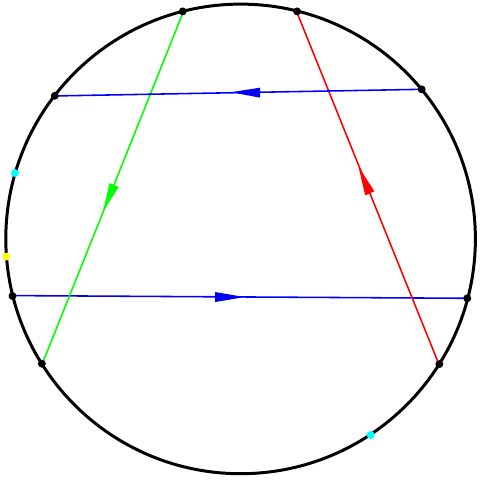}
    \tiny
\put (-63, 163){$(g_2)_+$}
\put (-116, 163){$(g_1)_-$}
\put (-163, 133){$(g_3')_+$}
\put (-167, 103){$K$}
\put (-169, 75){$F$}
\put (-177, 60){$(g_3)_-$}
\put (-166, 35){$(g_1)_+$}
\put (-37, 10){$H$}
\put (-13, 35){$(g_2)_-$}
\put (-3, 60){$(g_3)_+$}
\put (-18, 133){$(g_3')_-$}
    \caption{Case 2.}
    \label{fig: fourcase2}
\end{figure}

{\bf Case 3: $\rho(F_2)\not\subset\PSL_2(\R)$ and $\mathbf k(\rho)<2$.} Let $\{\gamma_1,\gamma_2\}$ be the pair of generators given by Observation \ref{obs: surface}(3). Let $g_i:=\rho(\gamma_i)$ and $g_i':=\rho(\gamma_i')$ for $i=1,\dots,4$, and let $g_i'':=\rho(\gamma_i'')$ for $i=1,2$. The positivity of (\ref{eqn: Case 3}) implies that $\big((g_4)_+,(g_4)_-,(g_4')_-,(g_4')_+\big)$ is a positive tuple of flags, and it is easy to verify that both $g_1^{-1}\cdot (g_4)_\pm=(g_4')_\pm=g_2^{-1}\cdot (g_4)_\pm$. Choose $K\in\Fc(V)$ such that $\big((g_4)_+,K,(g_4)_-,(g_4')_-,(g_4')_+\big)$ is positive. The fact that $g_1$ and $g_2$ both switch the orientation on $\mathbb P(\R^2)$ then implies that the tuples
\[\big((g_4)_+,(g_4)_-,(g_4')_-,g_1^{-1}\cdot K,(g_4')_+\big)\,\text{ and }\,\big((g_4)_+,(g_4)_-,(g_4')_-,g_2^{-1}\cdot K,(g_4')_+\big)\]
are both positive. Since $g_4'=g_1^{-1}g_2$ and $g_4'$ preserves the orientation on $\mathbb P(\R^2)$, it follows that the tuple
\[\big((g_4)_+,K,(g_4)_-,(g'_4)_-,g_2^{-1}\cdot K,g_1^{-1}\cdot K,(g_4')_+\big)\]
is also positive. Thus, there are points $F,H\in\mathbb{P}(\R^2)$ such that
\[\big((g_4)_+,F,(g_4)_-,(g_4')_-,g_2^{-1}\cdot F,H,g_1^{-1}\cdot F, (g_4')_+\big)\]
is positive, see Figure \ref{fig: fourcase3}. From this, one deduces that $\rho\circ R_3^1$ is $\Lambda_3^1$-admissible by assigning $(F,H,(g_4')_-)$ to $v_1$ and $w_1$, and $(F,H,(g_4')_+)$ to $v_1'$ and $w_1'$, so $\rho$ is $(R_3^1,\Lambda_3^1)$-weakly positive. 

\begin{figure}[h]
    \centering
    \includegraphics[width=0.4\textwidth]{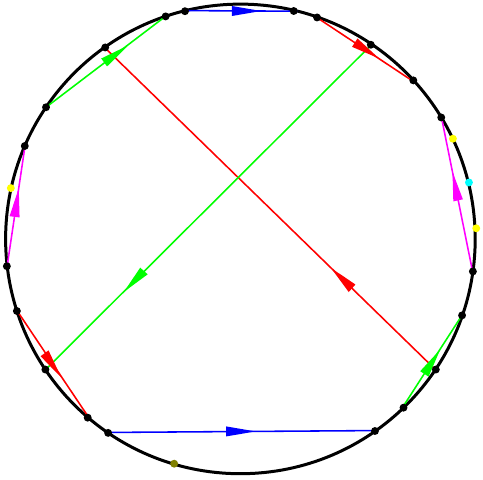}
    \tiny
\put (-145, 150){$(g_2)_+$}
\put (-172, 116){$(g_4)_+$}
\put (-167, 98){$F$}
\put (-179, 70){$(g_4)_-$}
\put (-165, 34){$(g_1)_+$}
\put (-141, 11){$(g_3)_-$}
\put (-110, 0){$G$}
\put (-38, 10){$(g_3)_+$}
\put (-14, 36){$(g_2)_-$}
\put (-3, 66){$(g_4')_-$}
\put (-3, 99){$H$}
\put (-1, 84){$g_2^{-1}\cdot F$}
\put (-8, 115){$g_1^{-1}\cdot F$}
\put (-15, 126){$(g_4')_+$}
\put (-40, 151){$(g_1)_-$}
    \caption{Case 3.}
    \label{fig: fourcase3}
\end{figure}

It now suffices to show that $\rho$ is $(R_3^3,\Lambda_3^3)$-weakly positive; replacing $(g_1,g_3)$ with $(g_2^{-1},g_3^{-1})$ proves that $\rho$ is $(R_3^2,\Lambda_3^2)$-weakly positive. By the positivity of (\ref{eqn: Case 3}), there are points $G\in\mathbb{P}(\R^2)$ such that
\[\big((g_1)_-,(g_1)_+,(g_3)_-,G,(g_3)_+\big)\]
is positive. One easily verifies that $\rho\circ R_3^3$ is $\Lambda_3^3$-admissible by assigning $(G,(g_4)_-,(g_4)_+)$ to $v_3$ and $w_3$, and $(G,(g_4)_+,(g_4)_-)$ to $v_3'$ and $w_3'$, so $\rho$ is $(R_3^3,\Lambda_3^3)$, so $\rho$ is $(R_3^3,\Lambda_3^3)$-weakly positive.

\begin{figure}[h]
    \centering
    \includegraphics[width=0.4\textwidth]{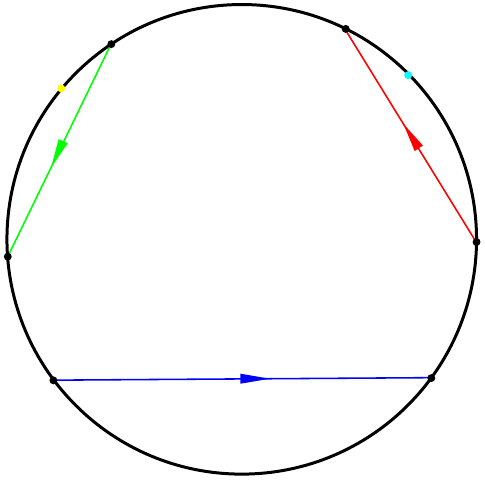}
    \tiny
\put (-142, 150){$(g_1)_-$}
\put (-150, 132){$F$}
\put (-180, 72){$(g_1)_+$}
\put (-1, 80){$(g_2)_-$}
\put (-25, 138){$H$}
\put (-46, 155){$(g_2)_+$}
    \caption{Case 4.}
    \label{fig: fourcase4}
\end{figure}

{\bf Case 4: $\rho(F_2)\subset\PSL_2(\R)$ and $\mathbf k(\rho)>2$.} Let $\{\gamma_1,\gamma_2,\gamma_3\}$ be the superbasis given by Observation \ref{obs: surface}(4). 
It is sufficient to show that $\rho$ is $(R_3^1,\Lambda_3^1)$-weakly positive; replacing $(\gamma_1,\gamma_2)$ with $(\gamma_2,\gamma_3)$ (resp. $(\gamma_3,\gamma_1)$) proves that $\rho$ is $(R_3^2,\Lambda_3^2)$-weakly positive (resp. $(R_3^3,\Lambda_3^3)$-weakly positive). Let $g_i:=\rho(\gamma_i)$ for $i=1,2,3$. By the positivity of (\ref{eqn: Case 4}), there are points $F,H\in\mathbb{P}(\R^2)$ such that
\[\big((g_1)_-,F,(g_1)_+,(g_2)_-,H,(g_2)_+\big)\]
is positive, see Figure \ref{fig: fourcase4}. It is easy to verify that $\rho\circ R_3^1$ is $\Lambda_3^1$-admissible by assigning $(F,H,(g_2)_-)$ to $v_1$ and $w_1$, and $(F,H,(g_2)_+)$ to $v_1'$ and $w_1'$, so $\rho$ is $(R_3^1,\Lambda_3^1)$-weakly positive.
\end{proof}

Using arguments similar to the proof of Theorem \ref{thm: n=2}, one can verify that: 
\begin{itemize}
\item In Case 1, $\rho$ is $(R_2,\Lambda_2)$-weakly positive for any pair of generators $\{\gamma_1,\gamma_2\}$ of $F_2$.
\item In Case 2, $\rho$ is $(R_2,\Lambda_2)$-weakly positive for the pairs of generators $\{\gamma_1,\gamma_3\}$ and $\{\gamma_2,\gamma_3\}$, where $\gamma_1,\gamma_2,\gamma_3$ are given by Observation \ref{obs: surface}(2).
\item In Case 3, $\rho$ is $(R_2,\Lambda_2)$-weakly positive for the pair of generators $\{\gamma_1,\gamma_2\}$, where $\gamma_1,\gamma_2$ are given by Observation \ref{obs: surface}(3).
\end{itemize}
However, in Case 4, $\rho$ is not $(R_2,\Lambda_2)$-weakly positive for any pair of generators of $F_2$.

\vspace{1\baselineskip}\noindent
Department of Mathematics,
Jeju National University,
Jeju 63243,
Republic of Korea\\
\url{sungwoon@jejunu.ac.kr}\\
Department of Mathematics,
National University of Singapore,
10 Lower Kent Ridge Rd, Singapore 119076\\
\url{mattansp@nus.edu.sg}, \url{matzt@nus.edu.sg}

\end{document}